\documentclass[a4paper, 11pt]{article}
\usepackage{fullpage}
\usepackage{graphicx}
\usepackage{amsmath, amsthm, amssymb, color, comment, mathtools, url}
\usepackage[unicode]{hyperref}
\usepackage{booktabs}

\newtheorem{theorem}{Theorem}
\newtheorem{lemma}[theorem]{Lemma}
\newtheorem{proposition}[theorem]{Proposition}

\newtheorem{claim}[theorem]{Claim}
\newtheorem{corollary}[theorem]{Corollary}

\theoremstyle{definition}
\newtheorem{definition}[theorem]{Definition}

\theoremstyle{remark}

\renewcommand{\ll}{\mbox{\rm [\![}}
\newcommand{\rr}{\mbox{\rm ]\!]}}
\def\around#1#2{#1\ll#2\rr}
\newcounter{rnc}
\newcommand{\Rn}[1]{\setcounter{rnc}{#1}\Roman{rnc}}
\newcommand{\id}{1_\Gamma}
\renewcommand{\G}{{\mathbf G}}
\newcommand{\bH}{{\mathbf H}}
\newcommand{\vE}{{\vec{E}}}
\newcommand{\vF}{{\vec{F}}}
\newcommand{\vG}{{\vec{G}}}
\newcommand{\bHead}{\mathrm{head}}
\newcommand{\cD}{{\cal D}}
\newcommand{\cDab}{{\cal D}_{\alpha, \beta}}
\newcommand{\cDabi}{{\cal D}_{\alpha_i, \beta_i}}
\newcommand{\cDabp}{{\cal D}_{\alpha', \beta'}}
\newcommand{\cDabpi}{{\cal D}_{\alpha'_i, \beta'_i}}

\newcommand{\tG}{{\tilde{G}}}

\newcommand{\tE}{{\tilde{E}}}
\newcommand{\tF}{{\tilde{F}}}

\newcommand{\ZZ}{{\mathbb Z}}

\newcommand{\two}{ /_{\!2} }
\newcommand{\three}{ /_{\!3} }
\newcommand{\conti}{ /_{\!i} }

\newcommand{\algoTTL}{{\sc TestTwoLabels}}
\newcommand{\algoFTP}{{\sc FindThreePaths}}
\newcommand{\ourtitle}{Finding~a~Path with~Two~Labels~Forbidden in~Group-Labeled~Graphs}

\title{\ourtitle\thanks{A preliminary version \cite{KKY} of this paper appeared in ICALP 2015.}}
\author{
  Yasushi Kawase\\[1mm]Tokyo Institute of Technology, Tokyo 152-8550, Japan.\\[0mm]Email: {\tt kawase.y.ab@m.titech.ac.jp}\vspace{2mm} \and
  Yusuke Kobayashi\\[1mm]Kyoto University, Kyoto 606-8502, Japan.\\[0mm]Email: {\tt yusuke@kurims.kyoto-u.ac.jp}\vspace{2mm} \and
  Yutaro Yamaguchi\thanks{The corresponding author. The full postal address is as follows: Department of Information and Physical Sciences, Osaka University, 1-5 Yamadaoka, Suita, Osaka 565-0871, Japan.}\\[1mm]Osaka University, Osaka 565-0871, Japan.\\[0mm]Email: \texttt{yutaro\_yamaguchi@ist.osaka-u.ac.jp}}
\date{\empty}

\begin{document}
\maketitle
\thispagestyle{empty}

\begin{abstract}
The parity of the length of paths and cycles
is a classical and well-studied topic in graph theory and theoretical computer science.
The parity constraints can be extended to label constraints
in a group-labeled graph, which is a directed graph with each arc labeled by an element of a group.
Recently, paths and cycles in group-labeled graphs have been investigated,
such as packing non-zero paths and cycles, where ``non-zero'' means that the identity element is a unique forbidden label.

In this paper, we present a solution to
finding an $s$--$t$ path with two labels forbidden in a group-labeled graph.
This also leads to an elementary solution to finding a zero $s$--$t$ path
in a $\ZZ_3$-labeled graph, which is the first nontrivial case of finding a zero path.
This situation in fact generalizes the 2-disjoint paths problem in undirected graphs,
which also motivates us to consider that setting.
More precisely, we provide a polynomial-time algorithm
for testing whether there are at most two possible labels of $s$--$t$ paths
in a group-labeled graph or not, and finding
$s$--$t$ paths attaining at least three distinct labels if exist.
The algorithm is based on a necessary and sufficient condition
for a group-labeled graph to have exactly two possible labels of $s$--$t$ paths,
which is the main technical contribution of this paper.
\end{abstract}


\clearpage
\thispagestyle{empty}
\tableofcontents
\clearpage
\setcounter{page}{1}

\section{Introduction}
\subsection{Background}\label{sec:background}
The parity of the length of paths and cycles in a graph
is a classical and well-studied topic in graph theory and theoretical computer science.
As the simplest example,
one can easily check the bipartiteness of a given undirected graph by determining whether it contains a cycle of odd length or not.
Also in a directed graph, a directed cycle of odd length can be detected in polynomial time by using the ear decomposition.
It is also an important problem to test whether a given directed graph contains a directed cycle of even length
or not, which is known to be equivalent to P\'{o}lya's permanent problem~\cite{Polya} (see, e.g., \cite{McCuaig}). 
A polynomial-time algorithm for this problem was 
devised by Robertson, Seymour, and Thomas~\cite{RST}.

In this paper, we focus on paths connecting two specified vertices $s$ and $t$. 
It is easy to test whether a given undirected graph
contains an $s$--$t$ path of odd (or even) length or not, 
whereas the same problem is NP-complete in the directed case~\cite{LP} (follows from~\cite{FHW}). 
A natural generalization of this problem is 
to consider paths of length $p$ modulo $q$.
One can easily see that, by considering the case when $q = 2$, 
the following  problems both generalize the problem of finding an odd (or even) $s$--$t$ path in an undirected graph:
\begin{itemize}
  \setlength{\itemsep}{.5mm}
\item
  finding an $s$--$t$ path of length $p$ modulo $q$ in an undirected graph, and
\item
  finding an $s$--$t$ path whose length is NOT $p$ modulo $q$ in an undirected graph, 
  which is equivalent to determining whether  
  all $s$--$t$ paths are of length $p$ modulo $q$ or not.
\end{itemize}
Although these two generalizations seem similar to each other, 
they are essentially different when $q \geq 3$. 
A linear-time algorithm for the second generalization 
was given by Arkin, Papadimitriou, and Yannakakis~\cite{APY} for any $q$, 
whereas not so much is known about the first generalization. 

\medskip
Recently, as another generalization of the parity constraints (including several other concepts such as contractibility in surfaces),
paths and cycles in a group-labeled graph have been investigated,
where a group-labeled graph is a directed graph with each arc labeled by an element of a group.
More specifically, for a fixed group $\Gamma$,
a pair of a directed graph and a mapping from its arc set to $\Gamma$
is called a {\em $\Gamma$-labeled graph}.
In a $\Gamma$-labeled graph, the label of a walk is defined
by sequential applications of the group operation of $\Gamma$ to the labels of the traversed arcs,
where each arc can be traversed in the backward direction by inverting its label
(see Section~\ref{sec:notations} for the precise definition).
Analogously to paths of length $p$ modulo $q$, 
it is natural to consider the following two problems\footnote{We remark that the group-labeled graphs do not generalize the setting ``$p$ modulo $q$'' when $q \geq 3$, because the labels are inverted when arcs are traversed in the backward direction.}: for a given element $\alpha \in \Gamma$,
\begin{itemize}
  \setlength{\itemsep}{.5mm}
\item[(\Rn{1})]
  finding an $s$--$t$ path of label $\alpha$ in a $\Gamma$-labeled graph, and
\item[(\Rn{2})]
  finding an $s$--$t$ path whose label is NOT $\alpha$ in a $\Gamma$-labeled graph,
  which is equivalent to determining whether  
  all $s$--$t$ paths are of label $\alpha$ or not.
\end{itemize}
Note that, when we consider Problem (\Rn{1}) or (\Rn{2}),
by changing uniformly the labels of the arcs around $s$ if necessary, 
we may assume that $\alpha$ is the identity element $\id \in \Gamma$.
Hence, each problem is equivalent to finding a path whose label is $\id$ or is not $\id$ in a $\Gamma$-labeled graph.
In what follows, we assume the black-box access to the underlying group $\Gamma$,
i.e., we can perform elementary operations in constant time
(see Section~\ref{sec:computation} for the precise assumption).

If the underlying group $\Gamma$ is $\ZZ_2 = \ZZ / 2\ZZ = (\{0, 1\}, +)$,
then the label of a path corresponds to the parity of the number of traversed arcs of label $1$.
Hence, by assigning label $1$ to all arcs,
both problems can formulate the problem of finding an odd (or even) $s$--$t$ path in an undirected graph.
We note that, in a $\ZZ_2$-labeled graph, 
finding an $s$--$t$ path of label $\alpha \in \ZZ_2$
is equivalent to finding an $s$--$t$ path whose label is not $\alpha + 1 \in \ZZ_2$,
but such an equivalence does not hold for any other nontrivial group. 

As shown in Section~\ref{sec:non-zero},
Problem (\Rn{2}) can be reduced to testing whether a $\Gamma$-labeled graph contains
a cycle whose label is not $\id$; such a cycle is called a {\em non-zero} cycle in some contexts.\footnote{Additive terms and notation are often used, whereas the commutativity is not assumed as with this paper.}
Based on this fact, Problem (\Rn{2}) can be easily solved in polynomial time 
for any group $\Gamma$ (Proposition~\ref{prop:non-zero}).
We mention that there are several results for packing non-zero paths
\cite{non-zero, non-zero_algorithm, TY, Yamaguchi} and non-zero cycles
\cite{HJW2019, KW, LRS2017, Wollan} with some conditions. 

On the other hand, the difficulty of Problem (\Rn{1}) 
is heavily dependent on the group $\Gamma$.
When $\Gamma \simeq \ZZ_2$, 
since Problems (\Rn{1}) and (\Rn{2}) are equivalent as discussed above,
it can be easily solved in polynomial time. 
When $\Gamma = \ZZ$ (as the additive group), 
Problem (\Rn{1}) is NP-complete
since the directed $s$--$t$ Hamiltonian path problem
reduces to this problem by labeling each arc with $1 \in \ZZ$
and letting $\alpha \coloneqq n - 1 \in \ZZ$,
where $n$ denotes the number of vertices.
Huynh~\cite{Huynh} showed the polynomial-time solvability
of Problem (I) for any fixed finite abelian group,
which is deeply dependent on the graph minor theory.

\subsection{Our contribution}
To investigate the gap between Problems (\Rn{1}) and (\Rn{2}),
we make a new approach to these problems
by generalizing Problem (\Rn{2}) so that multiple labels are forbidden.
In this paper, we provide a solution to the case when two labels are forbidden.
For a $\Gamma$-labeled graph $\G$ and two distinct vertices $s$ and $t$ in the graph,
let $l(\G; s, t)$ denote the set of all possible labels of $s$--$t$ paths in $\G$.

\begin{theorem}\label{thm:non-zero2}
Let $\G$ be a $\Gamma$-labeled graph with two specified vertices $s$ and $t$.
Then, for any distinct $\alpha, \beta \in \Gamma$, in polynomial time,
one can either find an $s$--$t$ path in $\G$ whose label is neither $\alpha$ nor $\beta$,
or conclude that $l(\G; s, t) \subseteq \{\alpha, \beta\}$.
\end{theorem}

The main technical contribution is to give
a characterization of $\Gamma$-labeled graphs $\G$ with two specified vertices $s$ and $t$
such that $l(\G; s, t) = \{\alpha, \beta\}$, which can be tested in polynomial time.
After it turns out that $l(\G; s, t) \not\subseteq \{\alpha, \beta\}$,
an $s$--$t$ path of label $\gamma \in \Gamma \setminus \{\alpha, \beta\}$
can be found by a rather na\"ive, brute-force strategy (see Section~\ref{sec:algorithm}).

We postpone the precise statement of our characterization (Theorem~\ref{thm:characterization})
to Section~\ref{sec:results}, and provide here only a high-level description.
Roughly speaking,
we show that $l(\G; s, t) = \{\alpha, \beta\}$ for distinct $\alpha, \beta \in \Gamma$ if and only if 
$\G$ is obtained from ``nice'' planar graphs (and some trivial graphs) by ``gluing'' them together 
(see Section~\ref{sec:characterization} for details).
It is interesting that planarity, which is a topological condition, appears in the characterization.

It is worth remarking that our result provides an elementary solution (without relying on the graph minor theory) to
the first nontrivial case of Problem (\Rn{1}), i.e., when $\Gamma \simeq \ZZ_3 = \ZZ / 3\ZZ = (\{0, \pm1\}, +)$.

\begin{corollary}\label{cor:algorithm}
Let $\G$ be a $\ZZ_3$-labeled graph with two specified vertices $s$ and $t$.
Then one can compute $l(\G; s, t)$ in polynomial time. 
Furthermore, for each $\alpha \in l(\G; s, t)$, 
one can find an $s$--$t$ path of label $\alpha$ in $\G$ in polynomial time. 
\end{corollary}

\subsection{Disjoint paths problem}
Problem (\Rn{1}) in a $\ZZ_3$-labeled graph 
in fact generalizes the {\em $2$-disjoint paths problem},
which also motivates us to consider the situation when two labels are forbidden.
The 2-disjoint paths problem asks whether for given distinct vertices $s_1, s_2, t_1, t_2$ in an undirected graph, there exists an $s_i$--$t_i$ path $P_i$ for each $i \in \{1, 2\}$ such that $P_1$ and $P_2$ are disjoint.
We can reduce the 2-disjoint paths problem to Problem (I) in a $\ZZ_3$-labeled graph as follows:
let $s \coloneqq s_1$ and $t \coloneqq t_2$, replace every edge in the given graph
by an arc with label $0$, add one arc from $t_1$ to $s_2$ with label $1$,
and ask whether the constructed $\ZZ_3$-labeled graph contains
an $s$--$t$ path of label $1$ or not.
Then, the desired two disjoint paths exist if and only if the answer is YES.

The 2-disjoint paths problem can be solved in polynomial time~\cite{Shiloach,2path,Thomassen}, 
and the following theorem characterizes the existence of two disjoint paths. 

\begin{theorem}[Seymour~{\cite{2path}}]\label{thm:2path}
  Let $G = (V, E)$ be an undirected graph and 
  $s_1, t_1, s_2, t_2 \in V$ distinct vertices. 
  Then, there exist two vertex-disjoint paths 
  $P_i$ connecting $s_i$ and $t_i$ $(i = 1, 2)$ 
  if and only if there is no family of disjoint vertex sets 
  $X_1, X_2, \ldots , X_k \subseteq V \setminus \{s_1, t_1, s_2, t_2\}$ such that
  \begin{itemize}
    \setlength{\itemsep}{.5mm}
  \item[$1.$]
    $N_G(X_i) \cap X_j = \emptyset$ if $i \neq j$, where $N_G(X_i)$ denotes the neighborhood of $X_i$ in $G$, 
  \item[$2.$]
    $|N_G(X_i)| \leq 3$ for $i = 1, 2, \ldots, k$, and
  \item[$3.$]
    if $G'$ is the graph obtained from $G$ by deleting $X_i$ and adding a new edge
    joining each pair of distinct vertices in $N_G(X_i)$ for each $i \in \{1, 2, \ldots, k\}$, 
    then $G'$ can be embedded in the plane so that $s_1, s_2, t_1, t_2$ are  
    on the outer boundary in this order. 
  \end{itemize}
\end{theorem}

Our characterization (Theorem~\ref{thm:characterization})
is inspired by (and further extends) Theorem~\ref{thm:2path},
and we also use this theorem in the proof. 

\medskip
We next mention that 
the {\em $k$-disjoint paths problem} is also regarded as a special case of 
Problem (I) for any fixed integer $k \geq 2$.\footnote{
This was also observed in \cite[p.~11]{Huynh}.
However, the reduction in \cite{Huynh} is inadequate,
which cannot distinguish two pairs of an $s_i$--$t_i$ path and an $s_j$--$t_j$ path
and of an $s_i$--$s_j$ path and a $t_i$--$t_j$ path for any distinct $i$ and $j$.}
The $k$-disjoint paths problem asks whether
for given $2k$ distinct vertices $s_i, t_i$ $(i = 1, 2, \ldots, k)$ in an undirected graph, 
there exist $k$ disjoint paths such that each path connects $s_i$ and $t_i$.
This problem can be formulated as Problem (I) using the alternating group
$A_{2k-1} = \{\, \sigma \mid \sigma~\text{is an even permutation of $\{1, 2, \ldots, 2k - 1\}$} \,\}$
(which is indeed isomorphic to $\ZZ_3$ when $k = 2$) as follows:
replace each edge by an arc with the identity permutation,
add an arc from $t_i$ to $s_{i+1}$ with label $(2i - 1 \ \, 2i + 1 \ \, 2i) \in A_{2k-1}$
(which is identity except $2i - 1 \mapsto 2i + 1 \mapsto 2i \mapsto 2i - 1$)
for each $i = 1, 2, \ldots, k - 1$,
and ask whether there exists an $s_1$--$t_k$ path of label
\[\sigma^\ast \coloneqq (2k - 3 \ \, 2k - 1 \ \, 2k - 2)\cdots(3~5~4)(1~3~2)\]
or not.
It is easy to check that $\sigma^\ast$ is the unique permutation
mapping $1$ to $2k - 1$ which can be constructed
in such an $A_{2k-1}$-labeled graph.

Although the $k$-disjoint paths problem can be solved in polynomial time for fixed $k$~\cite{GM13}, 
its solution requires sophisticated arguments based on the graph minor theory. 
This suggests that 
Problem (I) is a challenging problem  
even if the order of the group $\Gamma$ is bounded.

\subsection{Organization}
The rest of this paper is organized as follows.
In Section~\ref{sec:preliminaries},
we define several terms, notations, and basic operations,
and describe well-known properties.
Section~\ref{sec:results} is devoted to stating our characterization of $\Gamma$-labeled graphs
with exactly two possible labels of $s$--$t$ paths.
Based on the characterization, we present an algorithm for our problem
and prove Theorem~\ref{thm:non-zero2} in Section~\ref{sec:algorithm}.
Finally, in Section~\ref{sec:proof}, we verify the correctness of our characterization.

\clearpage
\section{Preliminaries}\label{sec:preliminaries}
\subsection{Terms and notations}\label{sec:notations}
Throughout this paper, let $\Gamma$ be a group (which can be non-abelian or infinite),
for which we adopt  the multiplicative notation with denoting the identity element by $\id$.

\subsubsection{$\Gamma$-Labeled graphs}
A {\em $\Gamma$-labeled graph} is a pair $\G = (\vec{G}, \psi)$ of a directed graph $\vec{G} = (V, \vec{E})$ and a mapping $\psi \colon \vec{E} \to \Gamma$,
called a \emph{label function}.
For each arc $\vec{e} = uv \in \vec{E}$, we refer to $\vec{e}$ as an \emph{arc from $u$ to $v$ with label $\psi(\vec{e})$},
and we denote by $e = \{u, v\}$ an edge obtained from $\vec{e}$ by ignoring the direction (and label),
which is referred to as an \emph{edge between $u$ and $v$}. 
We denote by $G = (V, E)$ the \emph{underlying graph} of $\vec{G}$, i.e., $E \coloneqq \{\, e  \mid \vec{e} \in \vec{E} \,\}$.
The direction information in $\vec{G}$ is used only to define the labels of walks in the underlying graph $G$,
and we often refer to each arc (with label) in $\G = (\vec{G}, \psi)$ as an edge in $G$ when we do not care its direction (and label).
As described in Section~\ref{sec:graphs}, we also use (and naturally extend) usual terms and notations for undirected graphs also for $\Gamma$-labeled graphs.

Throughout this paper, a $\Gamma$-labeled graph is assumed to be finite, and has no loop but may have parallel edges.
In other words, for the underlying graph $G = (V, E)$, the vertex set $V$ is finite,
and $E$ is a finite multiset of $2$-element subsets of $V$.

\subsubsection{Walks and labels}
Let $G = (V, E)$ be an undirected graph. 
For vertices $v_0, v_1, \dots , v_\ell \in V$ and edges $e_1, e_2, \dots , e_\ell \in E$
with $e_i = \{v_{i-1}, v_i\}$ $(i = 1, 2, \ldots, \ell)$,
an alternating sequence $W = (v_0, e_1, v_1, e_2, v_2, \ldots , e_{\ell}, v_{\ell})$
is called a {\em walk} (or a {\em $v_0$--$v_\ell$ walk} in particular) in $G$.
A walk $W$ is a {\em path} if $v_0, v_1, \ldots, v_\ell$ are all distinct.
Also, $W$ is said to be {\em closed} if $v_0 = v_{\ell}$,
and is called a {\em cycle} if, in addition, $e_1, e_2, \dots, e_\ell$ and $v_0, v_1, \ldots, v_{\ell-1}$ are respectively distinct.
We call $v_0$ and $v_\ell$ (which may coincide) the {\em end vertices of $W$},
and each $v_i$ $(1 \leq i \leq \ell - 1)$ an {\em inner vertex of $W$}.
For $i, j$ with $0 \leq i < j \leq \ell$, let $W[v_i, v_j]$ denote the subwalk
$(v_i, e_{i+1}, v_{i+1}, \ldots, e_j, v_j)$ of $W$ (we use this notation only when it is uniquely determined).
Let $\overline{W}$ denote the reversed walk of $W$, i.e.,
$\overline{W} = (v_\ell, e_\ell, \ldots, v_1, e_1, v_0)$.
The sets of vertices and of edges that appear in $W$ are denoted by $V(W)$ and $E(W)$, respectively,
i.e., $V(W) \coloneqq \{v_0, v_1, \dots, v_\ell\}$ and $E(W) \coloneqq \{e_1, e_2, \dots, e_\ell\}$ (where the multiplicity is ignored).

Let $\G = (\vec{G}, \psi)$ be a $\Gamma$-labeled graph, where we denote by $G = (V, E)$ the underlying graph of $\vec{G} = (V, \vec{E})$.
For each edge $e = \{u, v\} \in E$, we define $\psi_\G(e, v) \coloneqq \psi(\vec{e})$ if the corresponding arc $\vec{e} = uv \in \vec{E}$ \emph{enters} $v$, and $\psi_\G(e, v) \coloneqq \psi(\vec{e})^{-1}$ if $\vec{e} = vu \in \vec{E}$ \emph{leaves} $v$.
For a walk $W = (v_0, e_1, v_1, e_2, v_2, \ldots , e_{\ell}, v_{\ell})$ in $G$,
the {\em label} of $W$ in $\G$ is defined as the product
$\psi_\G(W) \coloneqq \psi_\G(e_\ell, v_\ell) \cdots \psi_\G(e_2, v_2) \cdot \psi_\G(e_1, v_1)$.
Note that $\psi_\G(\overline{W}) = \psi_\G(W)^{-1}$ for the reversed walk $\overline{W}$ of $W$.
We also call $W$ a walk in $\G$ to emphasize that the label information is included.
A walk $W$ in $\G$ is said to be {\em zero} (or {\em balanced} when $W$ is closed) if $\psi_\G(W) = \id$,
and {\em non-zero} (or {\em unbalanced} when $W$ is closed) otherwise, i.e, if $\psi_\G(W) \neq \id$.
We also say that $\G$ is {\em balanced} if all the cycles in $\G$ are balanced.\footnote{Note that whether a cycle in $\G$ is balanced or not 
does not depend on the choices of the direction and the end vertex,
because $\psi_\G(\overline{C}) = \psi_\G(C)^{-1}$ and $\psi_\G(C') = \psi_\G(e_1, v_1)\cdot\psi_\G(C)\cdot\psi_\G(e_1, v_1)^{-1}$
for any cycles $C = (v_0, e_1, v_1, e_2, v_2, \ldots, e_\ell, v_\ell = v_0)$ and $C' = (v_1, e_2, v_2, \ldots, e_\ell, v_\ell = v_0, e_1, v_1)$ in $G$.}

For a $\Gamma$-labeled graph $\G$ and two distinct vertices $s$ and $t$ in the underlying graph,
let $l(\G; s, t)$ denote the set of all possible labels $\psi_\G(P)$ of $s$--$t$ paths $P$ in $\G$. 
When $l(\G; s, t) = \{\alpha\}$ for some $\alpha \in \Gamma$, 
we also denote the element $\alpha$ itself by $l(\G; s, t)$.

\subsubsection{Graphs}\label{sec:graphs}
Let $\G = (\vG, \psi)$ be a $\Gamma$-labeled graph. 
Let $V(\G)$ and $E(\G)$ denote the vertex set $V$ and the edge set $E$ of the underlying graph $G = (V, E)$,
and $\vE(\G)$ the arc set $\vE$ of the directed graph $\vG = (V, \vE)$.
For a vertex set $X \subseteq V$,
we denote by $\delta_\G(X)$ the set of edges between $X$ and $V \setminus X$ in $G$
and by $N_\G(X)$ the set of vertices adjacent to $X$ in $G$, i.e.,
$\delta_\G(X) \coloneqq \{\, e \in E \mid |e \cap X| = 1 \,\}$ and
$N_\G(X) \coloneqq \bigl(\bigcup \delta_\G(X)\bigr) \setminus X = \{\, y \in V \setminus X \mid \delta_\G(X) \cap \delta_\G(\{y\}) \neq \emptyset \,\}$.
To simplify notation, we often write $x$ instead of $\{x\}$.

We define \emph{subgraphs} of $\G$.
For a vertex set $X \subseteq V$,
we denote by $E(X)$ and $\vE(X)$ the sets of edges included in $X$ and of corresponding arcs in $\vG$, respectively, i.e.,
$E(X) \coloneqq \{\, e \in E \mid e \subseteq X \,\}$ and $\vE(X) \coloneqq \{\, \vec{e} \in \vE \mid e \in E(X) \,\}$.
Then, the induced subgraphs in the usual sense of graphs are defined by $G[X] \coloneqq (X, E(X))$ and $\vG[X] \coloneqq (X, \vE(X))$.
Let $\G[X] \coloneqq \bigl(\vG[X], \psi|_{\vE(X)}\bigr)$ denote the subgraph of $\G$ \emph{induced by $X$},
where $\psi|_{\vE(X)} \colon \vE(X) \to \Gamma$ is the restriction of $\psi$ to $\vE(X)$, i.e.,
$\psi|_{\vE(X)}(\vec{e}) = \psi(\vec{e})$ for every $\vec{e} \in \vE(X)$.
We denote by $\G - X$ the subgraph of $\G$ obtained by removing
all vertices in $X$, i.e., $\G - X \coloneqq \G[V \setminus X]$.
For an edge set $F \subseteq E$, let $\vF \coloneqq \{\, \vec{e} \in \vE \mid e \in F \,\}$.
We denote by $\G - F$ 
the subgraph of $\G$ obtained by removing all edges in $F$,
i.e., $\G - F \coloneqq \bigl(\vG - \vF, \psi|_{\vE \setminus \vF}\bigr)$,
where $\vG - \vF = (V, \vE \setminus \vF)$.
Define $\around{\G}{X} \coloneqq \G[X \cup N_\G(X)] - E(N_\G(X))$.

We say that $\G$ is \emph{connected} if so is the underlying graph $G$ in the usual sense,
and other connectivity concepts are extended as follows.
For an integer $k \geq 1$, a proper vertex subset $X \subsetneq V$ with $|X| = k$ is called a {\em $k$-cut} in $\G$ if $\G - X$ is not connected.
We say that $\G$ is {\em $k$-connected} if $|V| > k$
and $\G$ contains no $k'$-cut for every $k' < k$.
A {\em $k$-connected component} of $\G$ is
a maximal $k$-connected induced subgraph $\G[X]$
$(X \subseteq V \text{ with } |X| \geq k)$.
For vertex sets $X, Y, Z \subseteq V$,
we say that $X$ {\em separates $Y$ from $Z$ in $\G$} if
every two vertices $y \in Y \setminus X$ and $z \in Z \setminus X$
are contained in different connected components of $\G - X$.
In particular, if $X$ separates $Y$ and $Z$ in $\G$
and $Y \setminus X \neq \emptyset \neq Z \setminus X$,
then $X$ is an $|X|$-cut in $\G$.

We also extend usual concepts of planar embedding. 
Suppose that $\G$ is connected and embedded in the plane.
We call the unique unbounded face the {\em outer face} of $\G$,
and any other face an {\em inner face}.
For a face $F$ of $\G$,
let ${\rm bd}(F)$ denote a closed walk\footnote{Such a walk is not unique because there are multiple choices of the direction and of the end vertex, but it does not matter because we will only consider whether ${\rm bd}(F)$ is balanced or not.} obtained
by walking once around the boundary of $F$ in an arbitrary direction from an arbitrary vertex.

\subsection{Assumptions}
\subsubsection{On computation model}\label{sec:computation}
The problems considered in this paper involve computation on groups.
For the group $\Gamma$ in question, we discuss computational aspects of the problems under the following assumptions.
Since the input size depends not only on the graph size but also on how the elements of $\Gamma$ are represented,
we assume that each element of $\Gamma$ that appears in the input $\Gamma$-labeled graph is represented by some symbol
whose size is bounded by a constant (independent of the graph size).
Since every element $\alpha \in \Gamma$ that has to be considered in the problems is the label of some walk,
it can be represented by a finite sequence of such symbols and inversions, called a \emph{word},
and a word representing $\alpha$ may not be unique.
We assume that the following operations for the words can be performed in constant time:
for any two words representing $\alpha, \beta \in \Gamma$,
getting a word representing the inverse element $\alpha^{-1} \in \Gamma$,
computing a word representing the product $\alpha\beta \in \Gamma$, and
testing whether $\alpha = \id$ or not.\footnote{One can also test whether $\alpha = \beta$ or not in constant time
by testing whether $\alpha\beta^{-1} = \id$ or not, because a word representing $\alpha\beta^{-1}$ as well as $\beta^{-1}$
can be obtained in constant time due to the first two assumptions.}

\subsubsection{On graphs and labels}
In the definition of the labels of walks,
an arc from $u$ to $v$ with label $\alpha$ and an arc from $v$ to $u$ with label $\alpha^{-1}$ work the same.
We say that two arcs (with labels) are \emph{equivalent} if they work the same, i.e.,
either they are parallel and have the same label or
they connect the same vertices with the opposite directions and have the inverse labels of each other.
We say that two $\Gamma$-labeled graphs $\G_1$ and $\G_2$ with the same vertex set $V$ are \emph{equivalent}
if one is obtained from the other by replacing some arcs with equivalent arcs, i.e.,
there exists a bijection $\pi \colon E(\G_1) \to E(\G_2)$ such that $\pi(e) = e$ as 2-element subsets of $V$ 
and $\psi_{\G_1}(e, v) = \psi_{\G_2}(\pi(e), v)$ for every $e \in E(\G_1)$ and $v \in e$.
Then, for any walk $W$ in $\G_1$, there exists a corresponding walk $W'$ in $\G_2$ with $\psi_{\G_2}(W') = \psi_{\G_1}(W)$, and vice versa.

In this paper, we use label functions only to discuss the labels of walks,
and hence we do not need to distinguish equivalent $\Gamma$-labeled graphs.
We sometimes regard two equivalent $\Gamma$-labeled graphs as ``equal'' by replacing some arcs with equivalent arcs (i.e., by reversing the directions and inverting the labels) if necessary.
In particular, when discussing $l(\G; s, t)$ for a $\Gamma$-labeled graph $\G = (\vec{G}, \psi)$ with $s, t \in V(\G)$,
for convenience, we always assume that, in $\vec{G}$, all arcs around $s$ leave $s$
and all arcs around $t$ enter $t$. 

In addition, we may assume that $\G$ has no vertex or edge trivially redundant as follows:
\begin{itemize}
  \setlength{\itemsep}{.5mm}
\item
  any vertex in $V(\G)$ is contained in some $s$--$t$ path in $\G$, and
\item
  $\G$ has no equivalent arcs, i.e.,
  $\psi_\G(e_1, v) \neq \psi_\G(e_2, v)$ for any parallel edges $e_1, e_2 \in E(\G)$ and $v \in e_1 = e_2$
  (where the equality is as 2-element subsets of $V(\G)$ and does not hold as elements of $E(\G)$).
\end{itemize}
Let $\cD$ be the set of all triplets $(\G, s, t)$ such that $\G$ is a $\Gamma$-labeled graph
and two distinct vertices $s, t \in V(\G)$ satisfying the above two conditions.
For a $\Gamma$-labeled graph $\G$, it is easy to check the second condition and to remove the redundant edges,
where recall the computational assumption on the group $\Gamma$ (cf.~Section~\ref{sec:computation}).
Furthermore, the following lemma guarantees that one can efficiently obtain a unique maximal subgraph $\G'$ of $\G$
such that $(\G', s, t) \in \cD$ and $l(\G'; s, t) = l(\G; s, t)$
by computing a 2-connected component of a graph (e.g., by \cite{HT}).

\begin{lemma}\label{lem:cD}
  Let $\G$ be a $\Gamma$-labeled graph with $|V(\G)| \geq 3$ that has no equivalent arcs.
  For any distinct vertices $s, t \in V(\G)$,
  we have $(\G, s, t) \in \cD$ if and only if the graph $G_{st} = G + e_{st}$ obtained from the underlying graph $G$
  by adding a new edge $e_{st} = \{s, t\}$ is $2$-connected.
\end{lemma}

\begin{proof}
We first show that if $(\G, s, t) \in \cD$ then $G_{st} = G + e_{st}$ is 2-connected.
We prove the contraposition, i.e., suppose that $G_{st}$ is not 2-connected,
and see that some vertex $v \in V(\G)$ is not contained in any $s$--$t$ path in $\G$.
Since $G_{st}$ is not 2-connected, there exists a vertex $x \in V(\G)$ such that $G_{st} - x$ is not connected.
If $x \not\in \{s, t\}$, then $x$ separates some vertex $v \in V(\G) \setminus \{s, t, x\}$ from $\{s, t\}$ in $G$ as well as in $G_{st}$ (in which $s$ and $t$ are adjacent).
This implies that any $s$--$v$ or $v$--$t$ walk in $G$ (if exists) intersects $x$ in between.
Since any $s$--$t$ walk in $G$ containing $v$ is divided into an $s$--$v$ walk and a $v$--$t$ walk,
it intersects $x$ at least twice and hence is not a path.
Otherwise, by symmetry, we may assume $x = s$.
In this case, $s$ separates some vertex $v \in V(\G) \setminus \{s, t\}$ from $t$ in $G$ as well as in $G_{st}$,
and then any $s$--$t$ path in $G$ cannot contain $v$.

To the contrary,
suppose that $G_{st} = G + e_{st}$ is 2-connected, and we show that $(\G, s, t) \in \cD$.
Let $G_{srt}$ be the graph obtained from $G_{st}$ by subdividing the edge $e_{st}$ with one subdividing vertex $r$,
i.e., $G_{srt}$ is obtained from $G$ by adding a new vertex $r$ and two new edges $e_{rs} = \{r, s\}$ and $e_{rt} = \{r, t\}$.
We then see that $G_{srt}$ is also 2-connected as follows.
Suppose to the contrary that $G_{srt}$ is not 2-connected,
i.e., $G_{srt}$ has a 1-cut $x \in V(\G) \cup \{r\}$.
If $x \in V(\G)$, then $x$ separates some vertex $v \in V(\G) \setminus \{s, t, r\}$ from $\{s, t, r\}$ in $G_{srt}$,
and hence $x$ is also a 1-cut in $G_{st}$, which separates $v$ from $\{s, t\}$, a contradiction.
Otherwise, $x = r$ separates some vertex $v \in V(\G) \setminus \{s, t\}$ from $s$ or $t$ in $G_{srt}$,
and then $t$ or $s$, respectively, is a 1-cut in $G_{st}$ that separates $v$ from $s$ or $t$, a contradiction.
Thus, $G_{srt}$ is 2-connected.

Fix any vertex $v \in V(\G) \setminus \{s, t\}$.
By Menger's theorem (see, e.g., \cite[Chapter~3]{Diestel}), there exist two $r$--$v$ paths in $G_{srt}$ that do not share their inner vertices.
Since $r$ has only two neighbors $s$ and $t$, one of these paths starts with the edge $e_{rs} = \{r, s\}$ and the other starts with $e_{rt} = \{r, t\}$, say $P_s$ and $P_t$, respectively.
Then, the path obtained by concatenating $P_s[s, v]$ and $\overline{P_t}[v, t]$ is an $s$--$t$ path in $G$ that contains $v$.
\end{proof}

\subsection{Finding a non-zero path}\label{sec:non-zero}
In this section, we show that a non-zero $s$--$t$ path can be found
(i.e., Problem (\Rn{2}) can be solved)
efficiently by using well-known properties of $\Gamma$-labeled graphs.
The following techniques are often utilized
in dealing with $\Gamma$-labeled graphs
(see, e.g., \cite{non-zero, non-zero_algorithm, TY}).

\begin{definition}[Shifting]\label{def:shifting}
  Let $\G = (\vec{G}, \psi)$ be a $\Gamma$-labeled graph.
  For a vertex $v \in V(\G)$ and an element $\alpha \in \Gamma$,
  we say that $\psi' \colon \vec{E}(\G) \to \Gamma$ (or $\G' = (\vec{G}, \psi')$) is obtained by
  {\em shifting $\psi$} (or $\G$, respectively) {\em by $\alpha$ at $v$} if, for each edge $e \in E(\G)$,
  \[\psi'(\vec{e}) = \begin{cases}
    \alpha \cdot \psi(\vec{e}) & (\vec{e} \text{ enters } v),\\
    \psi(\vec{e}) \cdot \alpha^{-1} & (\vec{e} \text{ leaves } v),\\
    \psi(\vec{e}) & (e \not\in \delta_\G(v)).
  \end{cases}\]
\end{definition}

Suppose that $\G'$ is obtained by shifting $\G$ by $\alpha \in \Gamma$ at $v \in V(\G)$.
Then, $\psi_{\G'}(W) = \psi_\G(W)$ for any walk $W$ in $\G$ whose end vertices are both different from $v$,
and $\psi_{\G'}(C) = \alpha \cdot \psi_\G(C) \cdot \alpha^{-1}$ for any closed walk (in particular, any cycle) $C$ in $\G$ whose end vertex is $v$.

By definition, two sequential applications of shifting by $\alpha$ and (later) by $\beta$ at the same vertex $v$
are always the same operation as shifting by $\beta\alpha$ at $v$ just once.
Furthermore, when we apply shifting operations at different vertices,
the result does not depend on the order of applications,
since the label $\psi(\vec{e})$ of each arc $\vec{e} = uv$ is changed only by shifting at its end vertices $u$ and $v$,
and shifting operations at $u$ and at $v$ affect $\psi(\vec{e})$ in the opposite sides. 

We say that two $\Gamma$-labeled graphs $\G_1$ and $\G_2$ with the same vertex set $V$
are {\em $(s, t)$-equivalent} if there exists a mapping $\varphi \colon V \setminus \{s, t\} \to \Gamma$
such that $\G_2$ is obtained from $\G_1$ after shifting by $\varphi(v)$ at each $v$
(or equivalently $\G_1$ is obtained from $\G_2$ after shifting by $\varphi(v)^{-1}$ at each $v$).
Note that $l(\G_1; s, t) = l(\G_2; s, t)$ if $\G_1$ and $\G_2$ are $(s, t)$-equivalent.

\begin{lemma}\label{lem:shifting}
  If a $\Gamma$-labeled graph $\G = (\vec{G}, \psi)$ with distinct vertices $s, t \in V(\G)$ is connected and balanced,
  then there exists $\alpha \in \Gamma$ $($in fact, $\alpha = l(\G; s, t)$$)$ such that $\G$ is $(s, t)$-equivalent to the $\Gamma$-labeled graph $\G' = (\vec{G}, \psi')$ defined by
  \[\psi'(\vec{e}) \coloneqq \begin{cases}
    \alpha & (\text{$e \in \delta_\G(s)$}),\\
    \id & (\text{\rm otherwise}),
  \end{cases}\]
  for each edge $e \in E(\G)$, where recall the assumption that all arcs around $s$ leave $s$ in $\vec{G}$.
\end{lemma}

\begin{proof}
Take an arbitrary spanning tree $T$ of the underlying graph $G = (V, E)$.
Consider the following procedure.
Let $X \coloneqq \{t\}$ and $\psi'' \coloneqq \psi$.
While $X \neq V$,
take a neighbor $v \in N_T(X)$,
update $\psi''$ by shifting by $\psi''(\vec{e})$ or $\psi''(\vec{e})^{-1}$ at $v$ for a unique edge $e$ connecting $v$ and $X$ in $T$
so that $\psi''(\vec{e}) = \id$ after the shifting, and let $X \coloneqq X \cup \{v\}$.
This procedure takes $\mathrm{O}(|E|)$ time,
since it can be done just by performing breadth first search once
and shifting operations $|V| - 1$ times, where note that the label of each arc changes at most twice.

After the procedure,
we have $\psi''(\vec{e}) = \id$ for every edge $e$ in $T$,
and also for every edge $e \in E(\G)$ since $\G$ is balanced.
Suppose that we performed a shifting operation by $\alpha$ at $s$.
Then, the $\Gamma$-labeled graph $\G' = (\vec{G}, \psi')$ obtained by shifting $(\vec{G}, \psi'')$ by $\alpha^{-1}$ at $s$ after the procedure
is $(s, t)$-equivalent to $\G$. 
\end{proof}

\begin{lemma}\label{lem:balanced2}
  For any $(\G, s, t) \in \cD$,
  we have $|l(\G; s, t)| = 1$ if and only if $\G$ is balanced.
\end{lemma}

\begin{proof}
It is obvious from Lemma~\ref{lem:shifting} that $|l(\G; s, t)| = 1$ if $\G$ is balanced.
To prove the inversion,
suppose that $\G$ is not balanced and $|V(\G)| \geq 3$ (because the case when $V(\G) = \{s, t\}$ is trivial),
and let $C$ be an unbalanced cycle in $\G$.
Since $(\G, s, t) \in \cD$ implies that $\G + e_{st}$ is 2-connected by Lemma~\ref{lem:cD},
where $e_{st} = \{s, t\}$ is a new edge (with an arbitrary label),
there exist two disjoint paths between $\{s, t\}$ and $V(C)$ in $\G$ by Menger's theorem
(each path may be of length $0$, i.e., possibly $\{s, t\} \cap V(C) \neq \emptyset$).
Take such disjoint paths, say an $s$--$x$ path $P$ and a $t$--$y$ path $Q$ with $x, y \in V(C)$,
so that they are internally disjoint from $C$ (e.g., by taking minimal ones), i.e., $V(P) \cap V(Q) = \emptyset$ and $(V(P) \cup V(Q)) \cap V(C) = \{x, y\}$.
Since $\psi_\G(\overline{C}[x, y])^{-1} \cdot \psi_\G(C[x, y]) = \psi_\G(C) \neq \id$
(where $x$ is chosen as the end vertex of the cycle $C$),
we have $\psi_\G(C[x, y]) \neq \psi_\G(\overline{C}[x, y])$.
Hence, by extending $C[x, y]$ and $\overline{C}[x, y]$ along $P$ and $Q$,
we can construct two $s$--$t$ paths in $\G$ whose labels are distinct,
which implies $|l(\G; s, t)| \geq 2$.
\end{proof}

Lemmas~\ref{lem:cD}, \ref{lem:shifting}, and \ref{lem:balanced2}
lead to the following proposition.

\begin{proposition}\label{prop:non-zero}
  Let $\G$ be a $\Gamma$-labeled graph with two specified vertices $s, t \in V(\G)$.
  Then, for any $\alpha \in \Gamma$, in polynomial time,
  one can either find an $s$--$t$ path $P$ in $\G$ with $\psi_\G(P) \neq \alpha$,
  or conclude that $l(\G; s, t) \subseteq \{\alpha\}$.
\end{proposition}

\clearpage
\section{Characterization}\label{sec:results}
In this section,
we provide a complete characterization of triplets $(\G, s, t) \in \cD$
with $l(\G; s, t) = \{\alpha, \beta\}$ for some distinct $\alpha, \beta \in \Gamma$.
Since $|l(\G; s, t)| = 1$ if and only if $\G$ is balanced (Lemma~\ref{lem:balanced2}),
our characterization leads to the first nontrivial classification of
$\Gamma$-labeled graphs in terms of the number of possible labels
of $s$--$t$ paths, and the classification is also complete when $\Gamma \simeq \ZZ_3$.
We consider two cases separately: when $\alpha\beta^{-1} = \beta\alpha^{-1}$
and when $\alpha\beta^{-1} \neq \beta\alpha^{-1}$.

\subsection{Two labels $\alpha, \beta$ with $\alpha\beta^{-1} = \beta\alpha^{-1}$}\label{sec:easy}
First, we give a characterization in the case when $\alpha\beta^{-1} = \beta\alpha^{-1}$.
Note that this case does not appear when $\Gamma \simeq \ZZ_3$.
The following proposition holds analogously to Lemmas~\ref{lem:shifting}
and \ref{lem:balanced2} in Section~\ref{sec:non-zero},
which characterize triplets $(\G, s, t) \in \cD$ with $|l(\G; s, t)| = 1$.

\begin{proposition}\label{prop:2-cyclic}
  Let $\alpha$ and $\beta$ be distinct elements in $\Gamma$
  with $\alpha\beta^{-1} = \beta\alpha^{-1}$.
  For any $(\G, s, t) \in \cD$, we have $l(\G; s, t) = \{\alpha, \beta\}$
  if and only if $\G = (\vec{G}, \psi)$ is not balanced and is $(s, t)$-equivalent to a $\Gamma$-labeled graph $\G' = (\vec{G}, \psi')$ such that
  \begin{equation}\label{eq:2cyclic}
    \psi'(\vec{e}) = \begin{cases}
      \alpha~\text{\rm or}~\beta & (\text{$e \in \delta_\G(s)$}),\\
      \id~\text{\rm or}~\alpha\beta^{-1} & (\text{\rm otherwise}),
    \end{cases} \tag{$\ast$}
  \end{equation}
  for every edge $e \in E(\G)$, where recall the assumption that all arcs around $s$ leave $s$ in $\vec{G}$.
  Moreover, one can find such $\G'$ in polynomial time if exists.
\end{proposition}

\begin{proof}
It is easy to see that $l(\G; s, t) = \{\alpha, \beta\}$ if $\G$ is not balanced and such $\G'$ exists as follows.
Lemma~\ref{lem:balanced2} immediately implies $|l(\G; s, t)| \geq 2$.
Furthermore, 
the label of any path of length at least $1$ in $\G'$ starting at $s$ is $\alpha$ or $\beta$ as follows.
This can be shown by induction on the path length $\ell \geq 1$.
The base case when $\ell = 1$ is obvious from \eqref{eq:2cyclic}.
When $\ell > 1$, let $P = (s = v_0, e_1, v_1, \ldots, e_\ell, v_\ell)$.
Since the last edge $e_\ell$ is not around $s$, the label $\psi_{\G'}(P) = \psi_{\G'}(e_\ell, v_\ell) \cdot \psi_{\G'}(P[s, v_{\ell - 1}])$ is either
$\id \cdot \alpha = \alpha$, $\id \cdot \beta = \beta$, $\alpha\beta^{-1} \alpha = \beta\alpha^{-1}\alpha = \beta$, or $\alpha\beta^{-1}\beta= \alpha$
by \eqref{eq:2cyclic}, the induction hypothesis, and the assumption that $\alpha\beta^{-1} = \beta\alpha^{-1} = (\alpha\beta^{-1})^{-1}$.
Thus, the label of any $s$--$t$ path in $\G'$ is also $\alpha$ or $\beta$,
and the $(s, t)$-equivalence between $\G$ and $\G'$ leads to $l(\G; s, t) = l(\G'; s, t) = \{\alpha, \beta\}$.

The converse is rather nontrivial.
Suppose that $l(\G; s, t) = \{\alpha, \beta\}$.
Then, Lemma~\ref{lem:balanced2} immediately implies that $\G$ is not balanced.
To construct $\G'$ in the statement,
as with the proof of Lemma~\ref{lem:shifting},
take an arbitrary spanning tree $T$ of the underlying graph $G$,
and obtain $\G'' = (\vec{G}, \psi'')$ by shifting $\G$ at the vertices other than $t$ such that $\psi''(\vec{e}) = \id$ for every edge $e \in E(T)$.
Since $l(\G; s, t) = \{\alpha, \beta\}$ and the label of the unique $s$--$t$ path in $T$ is $\id$ in $\G''$,
we performed a shifting operation by $\alpha$ or $\beta$ at $s$
(recall Definition~\ref{def:shifting} and the assumption that all arcs around $s$ leave $s$).
Hence, by shifting $\G''$ by $\alpha^{-1}$ or $\beta^{-1}$, respectively, at $s$ after the above procedure,
we obtain a $\Gamma$-labeled graph $\G' = (\vec{G}, \psi')$ that is $(s, t)$-equivalent to $\G$.

In what follows, we confirm that $\G'$ indeed satisfies \eqref{eq:2cyclic} for every edge $e \in E(\G)$.
Suppose that $|V(\G)| \geq 3$ (because the case when $V(\G) = \{s, t\}$ is trivial)
and to the contrary that some edge $e = \{x, y\} \in E(\G)$ violates \eqref{eq:2cyclic}.
Let $\tE \subsetneq E(\G)$ be the set of edges satisfying \eqref{eq:2cyclic}.
Note that $E(T) \subseteq \tE$, and hence $\tE$ forms a connected spanning subgraph $\tG$ of $G$.
Since $G + e_{st}$ is 2-connected for a new edge $e_{st} = \{s, t\}$ by Lemma~\ref{lem:cD}, 
there exist two disjoint paths between $\{s, t\}$ and $\{x, y\}$ in $G$ by Menger's theorem,
and hence $G$ has an $s$--$t$ path traversing the edge $e = \{x, y\} \not\in \tE$.
Take an $s$--$t$ path $P$ in $G$ with $E(P) \setminus \tE \neq \emptyset$ so that $|E(P) \setminus \tE|$ is minimized.

If $|E(P) \setminus \tE| = 1$, then $\psi_{\G'}(P) \not\in \{\alpha, \beta\}$,
which contradicts $l(\G'; s, t) = l(\G; s, t) = \{\alpha, \beta\}$.
Otherwise, we have $|E(P) \setminus \tE| \geq 2$.
Let $e_1, e_2 \in E(P) \setminus \tE$ be the first two such edges
traversed in walking along $P$, and $Q$ the subpath of $P$
connecting $e_1$ and $e_2$ (hence, $E(Q) \subseteq \tE$, and possibly $E(Q) = \emptyset$).
Since $\tG$ is a connected spanning subgraph of $G$, there exists a path $R$
from some $u \in V(Q)$ to some $w \in V(P) \setminus V(Q)$ in $\tG$ whose inner vertices are all disjoint from $V(P)$
(such a path $R$ is obtained, e.g., by taking any minimal path between $V(Q)$ and $V(P) \setminus V(Q)$ in $\tG$).
We then construct an $s$--$t$ path $P'$ from $P$ by replacing $P[u, w]$ (or $P[w, u]$)
with $R$ (or $\bar{R}$) so that
$\emptyset \neq E(P') \setminus \tE \subsetneq E(P) \setminus \tE$,
where note that $|E(P') \cap \{e_1, e_2\}| = 1$.
This implies that $1 \leq |E(P') \setminus \tE| \leq |E(P) \setminus \tE| - 1$,
which contradicts the choice of $P$.
\end{proof}

\subsection{Two labels $\alpha, \beta$ with $\alpha\beta^{-1} \neq \beta\alpha^{-1}$}\label{sec:characterization}
We next discuss the case when $\alpha\beta^{-1} \neq \beta\alpha^{-1}$, which is much more difficult.
We state our characterization with a subset $\cDab \subseteq \cD$ for whose member $l(\G; s, t) = \{\alpha, \beta\}$ trivially holds,
and successively define it through Definitions~\ref{def:cDab0}--\ref{def:cDab}.
In short, $(\G, s, t) \in \cDab$ if
$\G$ is constructed by ``gluing'' together
``nice'' planar (as well as several trivial) $\Gamma$-labeled graphs
and their derivations.

\begin{theorem}\label{thm:characterization}
  Let $\alpha$ and $\beta$ be distinct elements in $\Gamma$
  with $\alpha\beta^{-1} \neq \beta\alpha^{-1}$.
  For any $(\G, s, t) \in \mathcal D$, 
  we have $l(\G; s, t) = \{\alpha, \beta\}$ 
  if and only if $(\G, s, t) \in \cDab$.
\end{theorem}

We first prepare basic ingredients of $\Gamma$-labeled graphs in $\cDab$ as follows.

\begin{figure}[tbp]\hspace{-3mm}
  \begin{tabular}{cc}
    \begin{minipage}[b]{0.5\hsize}
      \begin{center}
        \includegraphics[scale=0.7]{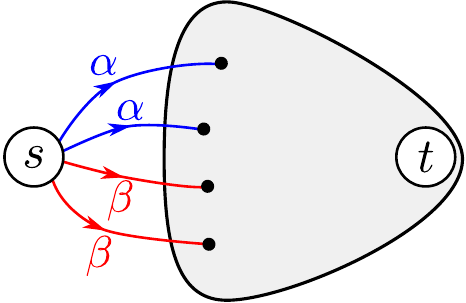}
      \end{center}\vspace{-4mm}
      \caption{$(\G, s, t) \in \cDab^0$ in Case (A1).}
      \label{fig:cDab0A1}
    \end{minipage}
    \begin{minipage}[b]{0.5\hsize}
      \begin{center}
        \includegraphics[scale=0.7]{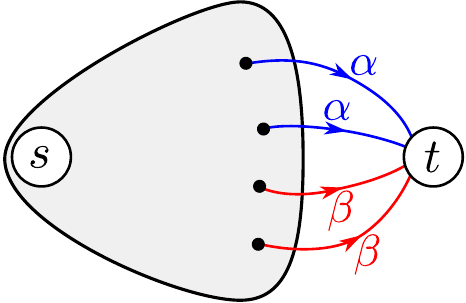}
      \end{center}\vspace{-4mm}
      \caption{$(\G, s, t) \in \cDab^0$ in Case (A2).}
      \label{fig:cDab0A2}
    \end{minipage}\\[7mm]
    \begin{minipage}[b]{0.5\hsize}
      \begin{center}
        \includegraphics[scale=0.7]{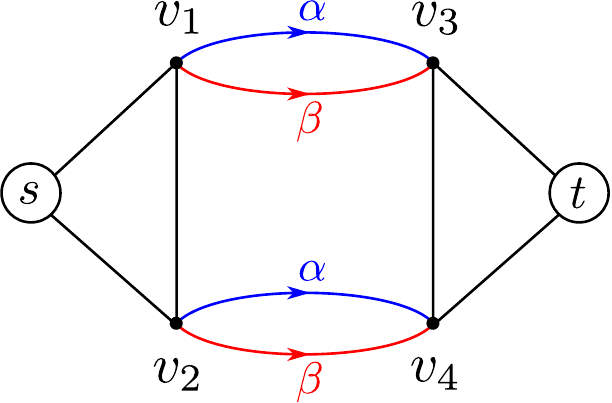}
      \end{center}\vspace{-5mm}
      \caption{$(\G, s, t) \in \cDab^0$ in Case (B).}
      \label{fig:cDab0B}
    \end{minipage}
    \begin{minipage}[b]{0.5\hsize}
      \begin{center}
        \includegraphics[scale=0.7]{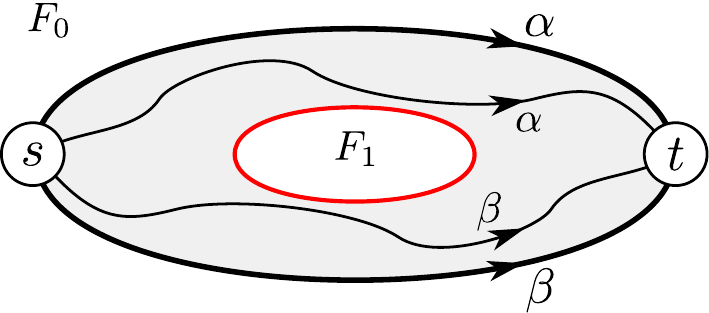}
      \end{center}\vspace{-3mm}
      \caption{$(\G, s, t) \in \cDab^0$ in Case (C).}\vspace{1mm}
      \label{fig:cDab0C}
    \end{minipage}
  \end{tabular}\vspace{-2mm}
\end{figure}

\begin{definition}\label{def:cDab0}
  {\rm For distinct $\alpha, \beta \in \Gamma$ with $\alpha\beta^{-1} \neq \beta\alpha^{-1}$, let $\cDab^0$ be
  the set of all triplets $(\G, s, t) \in \cD$
  satisfying one of the following conditions.
  \begin{itemize}
    \setlength{\itemsep}{.5mm}
  \item[(A)]
    There exists a $\Gamma$-labeled graph $\G' = (\vec{G}, \psi')$
    that is not balanced and is $(s, t)$-equivalent to $\G$ such that either\vspace{-1.5mm}
    \begin{itemize}
      \setlength{\itemsep}{.5mm}
    \item[(A1)]
      $\psi'(\vec{e}) = \alpha$ or $\beta$ for each edge $e \in \delta_\G(s)$ and $\psi'(\vec{e}) = \id$ otherwise,
      where recall the assumption that all arcs around $s$ leave $s$ (see Fig.~\ref{fig:cDab0A1}), or
    \item[(A2)]
      $\psi'(\vec{e}) = \alpha$ or $\beta$ for each edge $e \in \delta_\G(t)$ and $\psi'(\vec{e}) = \id$ otherwise,
      where recall the assumption that all arcs around $t$ enter $t$ (see Fig.~\ref{fig:cDab0A2}).
    \end{itemize}\vspace{-1mm}
  \item[(B)]
    $\G$ (or some equivalent $\Gamma$-labeled graph)
    is $(s, t)$-equivalent to a $\Gamma$-labeled graph
    that consists of six vertices, say $s, v_1, v_2, v_3, v_4, t$,
    six arcs $sv_1, sv_2, v_1v_2, v_3v_4, v_3t, v_4t$ with label $\id$,
    and two pairs of two parallel arcs from $v_i$ to $v_{i+2}$ $(i = 1, 2)$
    whose labels are both $\alpha$ and $\beta$ (see Fig.~\ref{fig:cDab0B}).
  \item[(C)]
    $\G$ can be embedded in the plane with two specified faces $F_0$ and $F_1$ (see Fig.~\ref{fig:cDab0C}) such that\vspace{-1.5mm}
    \begin{itemize}
      \setlength{\itemsep}{.5mm}
    \item
      $F_0$ is the outer face with both $s$ and $t$ on its boundary,
    \item
      one $s$--$t$ path along ${\rm bd}(F_0)$ is of label $\alpha$ and the other is of $\beta$, and
    \item
      $F_1$ is a unique inner face whose boundary is unbalanced, i.e., $\psi_\G({\rm bd}(F_1)) \neq \id$
	and $\psi_\G({\rm bd}(F)) = \id$ for any face $F$ other than $F_0$ or $F_1$.
    \end{itemize}
  \end{itemize}}
\end{definition}

Next, we introduce two new operations for $(\G, s, t) \in \cD$. 
Recall that, for each vertex set $X \subseteq V(\G)$, we define $\around{\G}{X} \coloneqq \G[X \cup N_\G(X)] - E(N_\G(X))$.

\begin{definition}[2-contraction]\label{def:2-contraction}
  {\rm For a vertex set $X \subseteq V(\G) \setminus \{s, t\}$
  such that $N_\G(X) = \{x, y\}$ for some distinct $x, y \in V(\G)$
  and $\around{\G}{X}$ is connected,
  the {\em $2$-contraction} of $X$ is the following operation (see Fig.~\ref{fig:2-contraction}):
  \begin{itemize}
    \setlength{\itemsep}{.5mm}
  \item
    remove all vertices in $X$ together with the incident edges, and
  \item
    add a new arc from $x$ to $y$ with label $\gamma$ for each $\gamma \in l(\around{\G}{X}; x, y)$
    if an equivalent arc not yet exists.
  \end{itemize}
  The resulting $\Gamma$-labeled graph is denoted by $\G \two X$.
  A vertex set $X \subseteq V(\G) \setminus \{s, t\}$
  is said to be {\em $2$-contractible in $\G$}
  if the 2-contraction of $X$ can be performed in $\G$
  and $\around{\G}{X} \neq \G$.\footnote{The latter condition is for excluding $V(\G) \setminus \{s, t\}$ unless $\{s, t\} \in E(\G)$, whose 2-contraction can be always performed and just results in a $\Gamma$-labeled graph consisting of parallel arcs from $s$ to $t$ with labels in $l(\G; s, t)$.}}
\end{definition}

\begin{figure}[tbp]
 \begin{center}
  \includegraphics[scale=0.7]{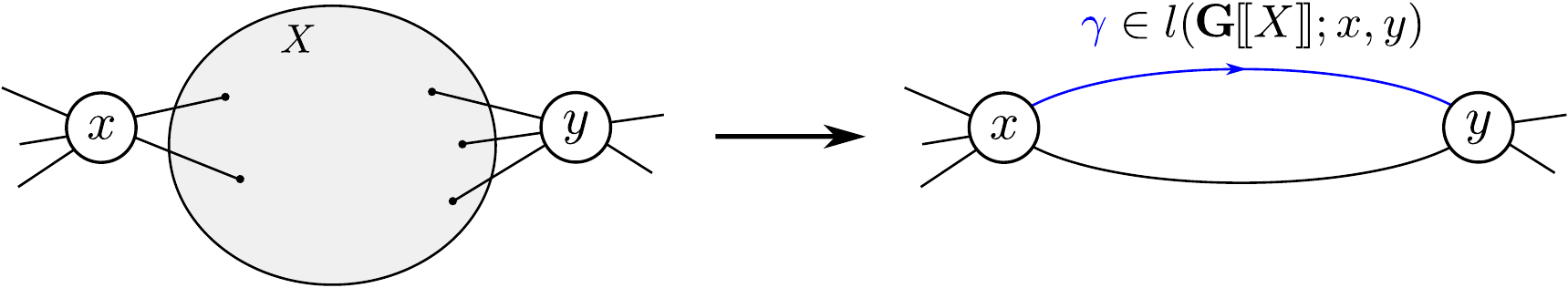}
 \end{center}\vspace{-8mm}
 \caption{2-contraction.}
 \label{fig:2-contraction}
\end{figure}

\begin{figure}[tbp]
 \begin{center}\vspace{5mm}
  \includegraphics[scale=0.6]{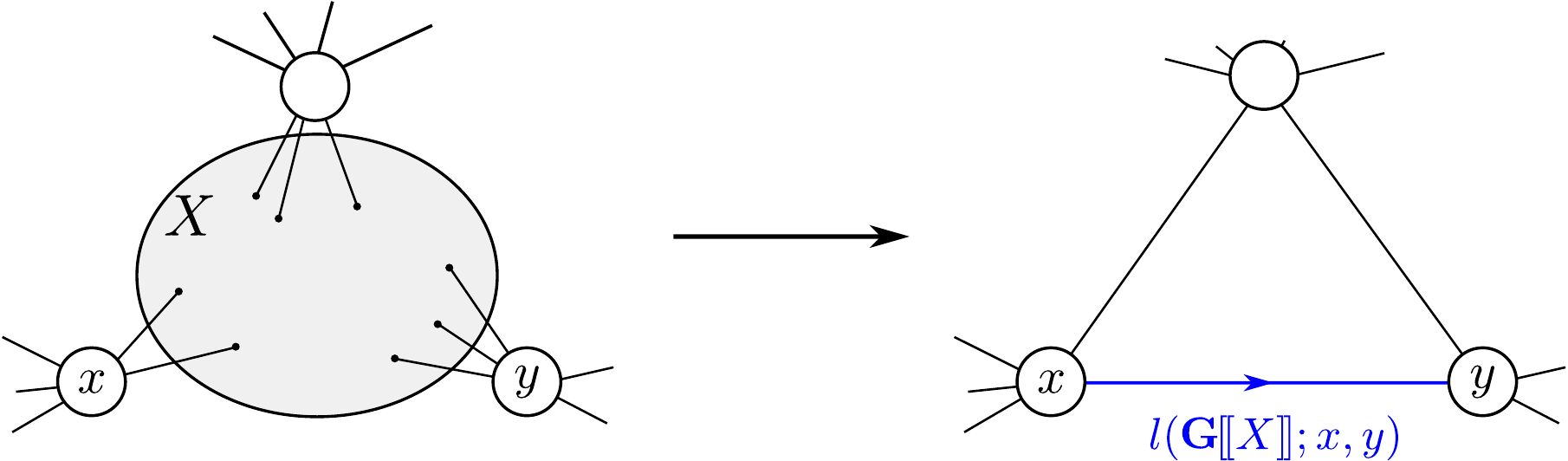}
 \end{center}\vspace{-6mm}
 \caption{3-contraction.}
 \label{fig:3-contraction}
\end{figure}

\begin{definition}[3-contraction]\label{def:3-contraction}
  {\rm For a vertex set $X \subseteq V(\G) \setminus \{s, t\}$ with $|N_\G(X)| = 3$
  such that $\G[X]$ is connected and $\around{\G}{X}$ is balanced,
  the {\em $3$-contraction} of $X$ is the following operation  (see Fig.~\ref{fig:3-contraction}):
  \begin{itemize}
    \setlength{\itemsep}{.5mm}
  \item
    remove all vertices in $X$ together with the incident edges, and
  \item
    add a new arc from $x$ to $y$ with label $l(\around{\G}{X}; x, y)$
    (which consists of a single element by Lemma~\ref{lem:balanced2})
    for each pair of distinct $x, y \in N_\G(X)$ if an equivalent arc not yet exists.
  \end{itemize}
  The resulting $\Gamma$-labeled graph is denoted by $\G \three X$.
  A vertex set $X \subseteq V(\G) \setminus \{s, t\}$
  is said to be {\em $3$-contractible in $\G$}
  if the 3-contraction of $X$ can be performed in $\G$.
  The cycle in $\G \three X$ formed by the three arcs added to $\G$ (or the equivalent arcs that already exist in $\G$) is balanced,
  and we refer to the cycle\footnote{The end vertex and the direction are not essential, and are fixed only when they are necessary.} (or its edge set)
as the \emph{balanced triangle} or just the \emph{triangle}.}
\end{definition}

The 2-contraction 
and the 3-contraction are analogous to the operation that is performed
in the condition 3 in Theorem~\ref{thm:2path},
and we use the same term ``contraction'' to refer to each of them.
We observe that any contraction makes no essential effect on our problem as follows.

\begin{lemma}\label{lem:contraction}
For any $(\G, s, t) \in \cD$,
fix $i \in \{2, 3\}$ and an $i$-contractible vertex set $X \subseteq V(\G) \setminus \{s, t\}$,
and let $\G' \coloneqq \G \conti X$.
Then, $(\G', s, t) \in \cD$ and $l(\G'; s, t) = l(\G; s, t)$.
\end{lemma}

\begin{proof}
We first confirm that $(\G', s, t) \in \cD$ if $(\G, s, t) \in \cD$.
If some vertex $v \in V(\G') = V(\G) \setminus X$ is not contained in any $s$--$t$ path in $\G'$,
then this is true in $\G$, which contradicts $(\G, s, t) \in \cD$.
Moreover, $\G'$ has no loop or equivalent arcs because this is true for $\G$ and the second procedures in Definitions~\ref{def:2-contraction} and \ref{def:3-contraction} do not yield such an arc.

Next we see $l(\G'; s, t) = l(\G; s, t)$.
Any $s$--$t$ path $P$ in $\G$ cannot enter $\around{\G}{X}$ after leaving it once,
i.e., if $E(P) \cap E(\around{\G}{X}) \neq \emptyset$,
then the edges in $E(P) \cap E(\around{\G}{X})$ form a path with end vertices in $N_\G(X)$, say from $x$ to $y$.
Since $\G'$ has an arc from $x$ to $y$ with label $\psi_{\G[\![X]\!]}(P[x, y]) = \psi_{\G}(P[x, y])$ (or an equivalent arc) by definition,
we can obtain an $s$--$t$ path $P'$ in $\G'$ of label $\psi_{\G}(P) = \psi_{\G}(P[y, t]) \cdot \psi_{\G}(P[x, y]) \cdot \psi_{\G}(P[s, x])$
from $P$ by replacing $P[x, y]$ with the corresponding edge $e = \{x, y\} \in E(\G')$.
To the contrary, any $s$--$t$ path $P'$ in $\G'$ that traverses an edge $e = \{x, y\} \in E(\G') \setminus E(\G)$ 
can be expanded to an $s$--$t$ path $P$ in $\G$ of label $\psi_{\G'}(P') = \psi_{\G'}(P'[y, t]) \cdot \psi_{\G'}(e, y) \cdot \psi_{\G'}(P'[s, x])$
by replacing $e$ with an $x$--$y$ path in $\around{\G}{X}$ of label $\psi_{\G'}(e, y)$.

We remark that in order to assure that $P'$ can be expanded to $P$,
we utilize the fact that $\G[X]$ is connected and $\around{\G}{X}$ is balanced when $i = 3$ (cf.~Definition~\ref{def:3-contraction}).
We may assume that $P'$ traverses at most one edge in the triangle
(if two of them are traversed, then they must be successive in $P'$,
and can be replaced by the other edge in the balanced triangle without changing the label).
Moreover, such an edge $e = \{x, y\} \in E(\G') \setminus E(\G)$ can be expanded by an $x$--$y$ path
in $\around{\G}{X}$ (whose label is unique due to the balancedness)
that does not intersect $N_\G(X) \setminus \{x, y\}$ because $\G[X]$ is connected.
\end{proof}

We are now ready to define $\cDab$,
which is, roughly speaking, the set of triplets $(\G, s, t) \in \cD$ that can be reduced to some $(\G', s, t) \in \cDab^0$
by a sequence of 2-contractions and 3-contractions.

\begin{definition}\label{def:cDab1}
  {\rm For distinct $\alpha, \beta \in \Gamma$ with $\alpha\beta^{-1} \neq \beta\alpha^{-1}$,
  we define $\cDab^1$ as the minimal set of triplets $(\G, s, t) \in \cD$
  with the following conditions:
  \begin{itemize}
    \setlength{\itemsep}{.5mm}
  \item
    $\cDab^0 \subseteq \cDab^1$, and
  \item
    if $(\G \three X, s, t) \in \cDab^1$
    for some 3-contractible $X \subseteq V(\G) \setminus \{s, t\}$,
    then $(\G, s, t) \in \cDab^1$.
  \end{itemize}}
\end{definition}

We remark that any sequence of 3-contractions can be replaced
by 3-contractions of disjoint 3-contractible vertex sets $X_1, X_2, \dots, X_k$
such that $N_\G(X_i) \cap X_j = \emptyset$ if $i \neq j$ as in Theorem~\ref{thm:2path},
which is proved later in Section~\ref{sec:lemmas} (cf.~Lemma~\ref{lem:3-contraction}).

\begin{definition}\label{def:cDab}
  {\rm For distinct $\alpha, \beta \in \Gamma$ with $\alpha\beta^{-1} \neq \beta\alpha^{-1}$, 
  we define $\cDab$ as the minimal set of triplets $(\G, s, t) \in \cD$
  with the following conditions:
  \begin{itemize}
    \setlength{\itemsep}{.5mm}
  \item
    $\cDab^1 \subseteq \cDab$, and
  \item
    if $(\G \two X, s, t) \in \cDab$
    for some 2-contractible $X \subseteq V(\G) \setminus \{s, t\}$
    such that either $\around{\G}{X}$ is balanced or $(\around{\G}{X}, x, y) \in \cDabp^1$,
    where $N_\G(X) = \{x, y\}$ and
    $\alpha', \beta' \in \Gamma$ satisfy $\alpha'{\beta'}^{-1} \neq \beta'{\alpha'}^{-1}$,
    then $(\G, s, t) \in \cDab$.
  \end{itemize}}
\end{definition}

Since any contraction does not change $l(\G; s, t)$ (by Lemma~\ref{lem:contraction}),
it is easy to see that $l(\G; s, t) = \{\alpha, \beta\}$ if $(\G, s, t) \in \cDab$.
The converse direction in Theorem~\ref{thm:characterization} is nontrivial, whose proof is presented later in Section~\ref{sec:proof}.

\clearpage
\section{Algorithm}\label{sec:algorithm}
In this section, we give a proof of Theorem~\ref{thm:non-zero2}
with assuming Theorem~\ref{thm:characterization}.
That is, we present a polynomial-time algorithm that tests, given a $\Gamma$-labeled graph $\G$ with $s, t \in V(\G)$ and $\alpha, \beta \in \Gamma$,
whether $l(\G; s, t) \subseteq \{\alpha, \beta\}$ or not and
returns an $s$--$t$ path $P$ in $\G$ with $\psi_\G(P) \not\in \{\alpha, \beta\}$ if $l(\G; s, t) \not\subseteq \{\alpha, \beta\}$.
We note again that, when $\Gamma \simeq \ZZ_3$,
using such an algorithm,
one can compute $l(\G; s, t)$ itself and find an $s$--$t$ path of label $\gamma$ for each $\gamma \in l(\G; s, t)$ (Corollary~\ref{cor:algorithm}).

\subsection{Algorithm description}
To prove Theorem~\ref{thm:non-zero2},
we give two algorithms, which slightly go farther than required.
One tests whether $|l(\G; s, t)| \leq 2$ or not,
and returns at most two $s$--$t$ paths in $\G$ that attain all labels in $l(\G; s, t)$ when $|l(\G; s, t)| \leq 2$.
The other finds three $s$--$t$ paths in $\G$ whose labels are distinct
when it has turned out that $|l(\G; s, t)| \geq 3$.

We first present the former algorithm.
In particular, when $\Gamma \simeq \ZZ_3$, this algorithm completely computes $l(\G; s, t)$ itself.
%
%
\begin{description}
  \setlength{\itemsep}{0mm}
\item[\algoTTL$(\G, s, t)$]
\item[Input:]
  A $\Gamma$-labeled graph $\G$ and distinct vertices $s, t \in V(\G)$.
\item[Output:]
  The set $l(\G; s, t)$ of all possible labels of $s$--$t$ paths in $\G$
  with those which attain the labels if $|l(\G; s, t)| \leq 2$,
  and a message ``$|l(\G; s, t)| \geq 3$'' otherwise.\vspace{1mm}

\item[Step 0.]
  If $s$ and $t$ are not connected in $\G$, then halt with returning $\emptyset$ ($\G$ contains no $s$--$t$ path).
  Otherwise, let $\G'$ be the unique maximal subgraph of $\G$ such that $(\G', s, t) \in \cD$ and $l(\G'; s, t) = l(\G; s, t)$,
  which is obtained by removing redundant edges and by computing a 2-connected component of a graph (cf.~Lemma~\ref{lem:cD}).
  If there exist at least three parallel edges in $E(\G')$ between the same pair of vertices, then halt with reporting ``$|l(\G; s, t)| \geq 3$.''

\item[Step 1.]
  Test whether $\G'$ is balanced or not by Lemma~\ref{lem:shifting}
  (i.e., take an arbitrary spanning tree, and apply shifting along it).
  If $\G'$ is balanced, then halt with returning an arbitrary $s$--$t$ path in $\G'$ and its label.
  Otherwise, by utilizing an unbalanced cycle (which is detected in testing balancedness),
  construct two $s$--$t$ paths $P$ and $Q$ in $\G'$ whose labels are distinct (cf.~the proof of Lemma~\ref{lem:balanced2}),
  say $\alpha, \beta \in \Gamma$.
  In the following steps,
  we check whether $l(\G'; s, t) = \{\alpha, \beta\}$ or not.

\item[Step 2.]
  If $\alpha\beta^{-1} = \beta\alpha^{-1}$,
  then check the condition in Proposition~\ref{prop:2-cyclic}.
  Return $\{\alpha, \beta\}$ with the two $s$--$t$ paths $P$ and $Q$ if it is satisfied, and report ``$|l(\G; s, t)| \geq 3$'' otherwise.
  Otherwise (i.e., if $\alpha\beta^{-1} \neq \beta\alpha^{-1}$),
  to make $\G'$ 2-connected (unless $V(\G') = \{s, t\}$),
  add to $\G'$ a new arc from $s$ to $t$ with label $\alpha$ (or $\beta$)
  if $s$ and $t$ are not adjacent in $\G'$.

\item[Step 3.]
  While $\G'$ is not 3-connected and $|V(\G')| \geq 4$, do the following procedure.
  Let $\{x, y\} \subsetneq V(\G')$ be a 2-cut in $\G'$,
  and $X$ the vertex set of a connected component of $\G' - \{x, y\}$
  with $X \cap \{s, t\} = \emptyset$
  (such $X$ exists, since $s$ and $t$ are adjacent in $\G'$).
  Test whether $|l(\around{\G'}{X}; x, y)| \leq 2$ or not
  recursively by \algoTTL$(\around{\G'}{X}, x, y)$.
  Update $\G' \leftarrow \G' \two X$ (2-contraction)
  if $|l(\around{\G'}{X}; x, y)| \leq 2$,
  and halt with reporting ``$|l(\G; s, t)| \geq 3$'' otherwise.
    
\item[Step 4.]
  While there exists a 3-contractible vertex set
  $X \subseteq V(\G') \setminus \{s, t\}$,
  update $\G' \leftarrow \G' \three X$ (3-contraction).
  Note that here we use Lemma~\ref{lem:shifting} to check the balancedness of $\around{\G'}{X}$. 

\item[Step 5.]
  If $|V(\G')| \leq 6$,
  then compute $l(\G', s, t)$ by enumerating all $s$--$t$ paths in $\G'$
  and return the result.
  Otherwise, test whether $(\G', s, t) \in \cDab^0$ or not
  by Lemma~\ref{lem:step5} (which will be stated and proved in Section~\ref{sec:proof_algorithm}).
  Return $\{\alpha, \beta\}$ with the two $s$--$t$ paths $P$ and $Q$
  if $(\G', s, t) \in \cDab^0$, and report ``$|l(\G; s, t)| \geq 3$'' otherwise.
\end{description}

Next, we show the latter algorithm,
which finds three $s$--$t$ paths in $\G$ whose labels are distinct
when it has turned out that $|l(\G; s, t)| \geq 3$.
In particular, when $\Gamma \simeq \ZZ_3$, this algorithm finds
three $s$--$t$ paths which attain all labels.
\begin{description}
  \setlength{\itemsep}{0mm}
\item[\algoFTP$(\G, s, t)$]
\item[Input:]
  A $\Gamma$-labeled graph $\G$ and distinct vertices $s, t \in V(\G)$
  such that $|l(\G; s, t)| \geq 3$.
\item[Output:]
  Three $s$--$t$ paths in $\G$ whose labels are distinct.\vspace{1mm}
\item[Step 1.]
  If $N_\G(s) = \{t\}$, then halt with returning three $s$--$t$ paths in $\G$
  each of which consists of a single edge $\{s, t\} \in E(\G)$.
  Otherwise, for each $s' \in N_\G(s) \setminus \{t\}$,
  test whether $|l(\G - s; s', t)| \leq 2$ or not
  by \algoTTL$(\G - s, s', t)$.
\item[Step 2.]
  If $|l(\G - s; s', t)| \leq 2$ for all $s' \in N_\G(s) \setminus \{t\}$,
  then we have already obtained (at most two)
  $s'$--$t$ paths in $\G - s$ which attain all labels in $l(\G - s; s', t)$.
  Choose three $s$--$t$ paths in $\G$ whose labels are distinct among the following ones and halt with returning them:\vspace{-1.5mm}
  \begin{itemize}
    \setlength{\itemsep}{.5mm}
  \item
    the $s$--$t$ paths obtained by extending such $s'$--$t$ paths in $\G - s$ using an edge (possibly parallel edges) $\{s, s'\} \in E(\G)$ for each $s' \in N_\G(s) \setminus \{t\}$, and
  \item 
    the $s$--$t$ paths each of which consists of a single edge $\{s, t\} \in E(\G)$.
  \end{itemize}\vspace{-.5mm}
  
\item[Step 3.]
  Otherwise, for at least one $\tilde{s} \in N_\G(s) \setminus \{t\}$,
  we obtained $|l(\G - s; \tilde{s}, t)| \geq 3$.
  Then, recursively by \algoFTP$(\G - s, \tilde{s}, t)$,
  find three $\tilde{s}$--$t$ paths in $\G - s$ whose labels are distinct.
  Extend the three $\tilde{s}$--$t$ paths using an edge $\{s, \tilde{s}\} \in E$,
  and return the extended $s$--$t$ paths in $\G$.
\end{description}

\subsection{Correctness (proof of Theorem~\ref{thm:non-zero2})}\label{sec:proof_algorithm}
Before starting the proof, 
we show the detailed procedure of Step~5 in \algoTTL.

\begin{lemma}\label{lem:step5}
  Let $(\G, s, t) \in \cD$.
  Suppose that $\G$ is $3$-connected and contains no $3$-contractible vertex set,
  $|V(\G)| > 6$, $\{s, t\} \in E(\G)$, and $\{\alpha, \beta\} \subseteq l(\G; s, t)$
  for some distinct $\alpha, \beta \in \Gamma$ with $\alpha\beta^{-1} \neq \beta\alpha^{-1}$.
  Then, one can test whether $(\G, s, t) \in \cDab^0$ or not in polynomial time.
\end{lemma}

\begin{proof}
Since $|V(\G)| > 6$, it is not necessary to consider Case (B) in Definition~\ref{def:cDab0}.
Besides, Case (A) is easily checked by testing whether $\G - s$ or $\G - t$ is balanced or not.
Hence, in what follows, we assume that $(\G, s, t) \in \cDab^0$ is not in Case (A) or (B)
and focus on Case (C).

First, test the planarity of $\G$.
If $\G$ is not planar, then we can conclude $(\G, s, t) \not\in \cDab^0$.
Otherwise, compute a planar embedding of $\G$ in which both $s$ and $t$ are on the outer boundary
(since $\{s, t\} \in E(\G)$, both $s$ and $t$ must be on the boundary of some face).
It should be noted that
such a planar embedding can be computed in polynomial time, e.g., by \cite{HT2}.
Moreover, since $\G$ is 3-connected,
the face set is unique if there are no parallel edges
(see, e.g., \cite[Chapter~4]{Diestel}).
Although there may be parallel edges in $\G$,
we can say that the number of parallel edges is bounded as seen below.

\medskip\noindent
\underline{{\bf Claim.}~~We may assume that no parallel edges exist between $s$ and $t$.}

\medskip
Suppose that $\G$ has parallel edges between $s$ and $t$.
If $\psi_\G(e, t) \not\in \{\alpha, \beta\}$ for some $e = \{s, t\} \in E(\G)$,
then $l(\G; s, t) \neq \{\alpha, \beta\}$ and hence we conclude $(\G, s, t) \not\in \cDab^0$.
Otherwise, since $\G$ has no equivalent arcs (by $(\G, s, t) \in \cD$),
there are exactly two parallel edges $e_\alpha, e_\beta \in E(\G)$ with $\psi_\G(e_\alpha, t) = \alpha$ and $\psi_\G(e_\beta, t) = \beta$.
Since $|V(\G)| > 6$ and any vertex in $V(\G)$ is contained in some $s$--$t$ path in $\G$ (by $(\G, s, t) \in \cD$),
there exists an $s$--$t$ path $P$ in $\G - \{e_\alpha, e_\beta\}$, and let $\gamma \coloneqq \psi_\G(P)$.
If $\alpha \neq \gamma \neq \beta$, then $|l(\G; s, t)| \geq 3$.
Otherwise, we can remove $e_\gamma$ from $\G$
without violating the hypotheses of this lemma and with preserving whether $(\G, s, t) \in \cDab^0$ or not.

\medskip\noindent
\underline{{\bf Claim.}~~We may assume that there exists at most one pair of two parallel edges.}

\medskip
Suppose that $\G$ has parallel edges between $x, y \in V(\G)$ with $\{x, y\} \neq \{s, t\}$.
Since $\G$ is 3-connected, $\{x, y\}$ is not a 2-cut in $\G$.
Hence, some pair of two parallel edges between $x$ and $y$ forms an inner face whose boundary is unbalanced.
If there are at least two such inner faces, then the condition of Case (C) is violated, and we conclude $(\G, s, t) \not\in \cDab^0$.
Otherwise, $\G$ has only one such pair of two parallel edges between $x$ and $y$, and such $\{x, y\}$ is also unique.

\medskip
Recall that we have to test whether there exists a planar embedding of $\G$
such that the outer boundary is unbalanced
and there exists a unique inner face whose boundary is unbalanced.
Since a pair of parallel edges is unique if exists,
there are at most two possible face sets of $\G$.
Furthermore, since there exists exactly one edge $e \in E(\G)$ between $s$ and $t$,
both of the two faces whose boundaries share the edge $e$
can be the outer face, i.e., there are two choices of the outer face.
It can be done in polynomial time
to check, in each of the at most four $( = 2 \times 2 )$ cases,
whether exactly one inner face has an unbalanced boundary or not,
and hence one can do the whole procedure in polynomial time.
\end{proof}

We are now ready to prove Theorem~\ref{thm:non-zero2},
where recall that we assume Theorem~\ref{thm:characterization}.

\begin{proof}[Proof of Theorem~$\ref{thm:non-zero2}$]
Recall that our goal is to test whether $|l(\G; s, t)| \leq 2$ or not, and
to find $\min\{3,\, |l(\G; s, t)|\}$ $s$--$t$ paths in $\G$ whose labels are distinct.
These are achieved as follows.
For the input triplet $(\G, s, t)$ (which may not be in $\cD$),
we first test whether $|l(\G; s, t)| \leq 2$ or not by \algoTTL$(\G, s, t)$.
If we obtain $|l(\G; s, t)| \leq 2$, then we also obtain
at most two $s$--$t$ paths in $\G$ which attain all labels in $l(\G; s, t)$.
Otherwise, we can obtain three $s$--$t$ paths whose labels are distinct
by \algoFTP$(\G, s, t)$.
Hence, it suffices to show the correctness and polynomiality of these two algorithms.

The correctness of \algoFTP\ is obvious, and that of \algoTTL\ is confirmed as follows.
First, we have $(\G', s, t) \in \cD$ and $l(\G'; s, t) = l(\G; s, t)$ at any step by Lemma~\ref{lem:contraction}
(i.e., since the 2-contractions in Step~3 and the 3-contractions in Step~4 neither violate $(\G', s, t) \in \cD$ nor change $l(\G'; s, t)$).
Next, if $\G'$ contains three parallel edges between the same pair of vertices, say $x$ and $y$, at some time,
then one can immediately construct three $s$--$t$ paths of distinct labels in $\G$ by extending those parallel edges
(corresponding to three arcs that are not equivalent)
using two disjoint paths between $\{s, t\}$ and $\{x, y\}$ in $\G'$, which exist by Menger's theorem and Lemma~\ref{lem:cD}.
Finally, the following two facts for $(\G', s, t) \in \cD$ and $\alpha, \beta \in \Gamma$ with $\alpha\beta^{-1} \neq \beta\alpha^{-1}$ are implicitly used in Step~5:
one is that $l(\G'; s, t) = \{\alpha, \beta\}$ is equivalent to $(\G', s, t) \in \cDab$ (by Theorem~\ref{thm:characterization}),
and the other is that, if $\G'$ contains no contractible vertex set, then $(\G', s, t) \in \cDab$ is equivalent to $(\G', s, t) \in \cDab^0$
(by Definitions~\ref{def:cDab1} and \ref{def:cDab}).

We finally confirm the polynomiality of the two algorithms.
Let $T_{\rm labels}(n, m)$ and $T_{\rm paths}(n, m)$ denote the computational time
of \algoTTL$(\G, s, t)$ and of \algoFTP$(\G, s, t)$, respectively,
where $n \coloneqq |V(\G)|$ and $m \coloneqq |E(\G)|$.
It is easy to see that \algoTTL\ runs in polynomial time, i.e., $T_{\rm labels}(n, m)$ is polynomially bounded.
Note that, in the recursion step (Step~3),
we just divide the current graph $\G'$ into two smaller graphs $\G' \two X$ and $\around{\G'}{X}$,
for which we have $|V(\G' \two X)| + |V(\around{\G'}{X})| = |V(\G')| + 2$ (only $x$ and $y$ are shared)
and $|E(\G' \two X)| + |E(\around{\G'}{X})| \leq |E(\G')| + |l(\around{\G'}{X}; x, y)| \leq |E(\G')| + 2$ (recall that the algorithm halts once $|l(\around{\G'}{X}; x, y)| \geq 3$ turns out).
Also, in the 3-contraction step (Step~4), it suffices to check all 3-cuts in $\G'$, whose number is ${\rm O}(n^3)$.
For \algoFTP, by a recurrence relation
\[T_{\rm paths}(n, m) \leq n \cdot T_{\rm labels}(n - 1, m - d_s) + T_{\rm paths}(n - 1, m - d_s) + \text{poly}(n, m),\]
where $d_s \coloneqq |\delta_\G(s)|$,
we have $T_{\rm paths}(n, m) \leq n^2 \cdot T_{\rm labels}(n, m) + \text{poly}(n, m)$.
Hence, $T_{\rm paths}(n, m)$ is also polynomially bounded.
\end{proof}

\clearpage
\section{Proof of Characterization (Necessity Part of Theorem~\ref{thm:characterization})}\label{sec:proof}
In this section,
we give a proof of the necessity part of Theorem~\ref{thm:characterization}.
The proof is done by contradiction, which is sketched in Section~\ref{sec:sketch}.
Several useful lemmas are prepared in Section~\ref{sec:lemmas}.
The main part of the proof begins in Section~\ref{sec:counterexample} by taking a minimal counterexample to derive a contradiction,
and is completed by a case analysis in Section~\ref{sec:case}.

\subsection{Proof sketch}\label{sec:sketch}
To derive a contradiction, assume that 
there exist $\alpha, \beta \in \Gamma$ with $\alpha\beta^{-1} \neq \beta\alpha^{-1}$
and a triplet $(\G, s, t) \in \cD$ such that $l(\G; s, t) = \{ \alpha, \beta \}$ but $(\G, s, t) \not\in \cDab$. 
We choose such $\alpha, \beta \in \Gamma$ and $(\G, s, t) \in \cD$
so that $\G$ is as small as possible (more precisely, the total number of vertices and edges in $\G$ is minimum).

Fix an arbitrary edge $e_0 \in \delta_\G(s)$,
and consider the graph $\G'\coloneqq \G - e_0$.  
By using the minimality of $\G$,
we can show that $(\G', s, t) \in \cDab$
(Claims~\ref{cl:cD} and \ref{cl:2label}).
We consider the following two cases separately:
when $(\G', s, t) \in \cDab^1$ and when not
(in Sections~\ref{sec:case1} and \ref{sec:case2}, respectively).

In both cases, we derive a contradiction by concluding one of the following situations.
\begin{itemize}
  \setlength{\itemsep}{.5mm}
\item
  $(\G', s, t) \in \cDab^0$ or $(\G' \three X, s, t) \in \cDab^0$ for some 3-contractible vertex set $X$,
  and the removed edge $e_0$ can be added without violating the conditions in Definition~\ref{def:cDab0},
  e.g., an embedding of $\G'$ in the plane can be extended to that of $\G = \G' + e_0$ with preserving the conditions of Case (C).
  This contradicts that $(\G, s, t) \not\in \cDab$.
\item
  $\G$ contains a contractible vertex set, 
  which contradicts that $\G$ is a minimal counterexample
  (cf.~Claims~\ref{cl:2-contractible} and \ref{cl:3-contractible}).
\item
  We can construct an $s$--$t$ path of label
  $\gamma \in \Gamma \setminus \{\alpha, \beta\}$ in $\G$
  (mostly by using $e_0$), 
  which contradicts that $l(\G; s, t) = \{ \alpha, \beta \}$. 
\end{itemize}
In each case, we have a contradiction,
which completes the proof. 
We remark again that Theorem~\ref{thm:2path} plays an important role in this case analysis.

\subsection{Useful lemmas}\label{sec:lemmas}
Before starting the proof, we show several lemmas which are utilized in it.
Fix distinct elements $\alpha, \beta \in \Gamma$ with $\alpha\beta^{-1} \neq \beta\alpha^{-1}$.

We first rephrase the conditions of $\cDab^1$ and of $\cDab$
so that they are easy to use.
The next lemma claims that, for any $(\G, s, t) \in \cDab^1$, one can obtain some $(\G', s, t) \in \cDab^0$
by the 3-contractions of nonadjacent vertex sets (as in Theorem~\ref{thm:2path}).
In particular, the 3-contractions of nonadjacent vertex sets do not interfere with each other,
and hence the resulting graph $\G'$ does not depend on the order of performing 3-contractions.

\begin{lemma}\label{lem:3-contraction}
  For any $(\G, s, t) \in \cD$,
  we have $(\G, s, t) \in \cDab^1$ if and only if there exist disjoint $3$-contractible vertex sets
  $X_1, X_2, \dots, X_k \subseteq V(\G) \setminus \{s, t\}$ such that
  \begin{itemize}
    \setlength{\itemsep}{.5mm}
  \item
    $N_\G(X_i) \cap X_j = \emptyset$ if $i \neq j$, and
  \item
    if we define $\G_k \coloneqq \G \three X_k$ and $\G_i \coloneqq \G_{i+1} \three X_i$ $(i = k - 1, \dots, 2, 1)$,
    then $(\G_1, s, t) \in \cDab^0$.
  \end{itemize}
\end{lemma}

\begin{proof}
The minimality of $\cDab^1$ in Definition~\ref{def:cDab1} requires that $(\G, s, t) \in \cDab^1$ if and only if
there exists a sequence $\G_1, \G_2, \dots, \G_{k+1} = \G$ of $\Gamma$-labeled graphs such that
$(\G_1, s, t) \in \cDab^0$ and $\G_i = \G_{i+1} \three X_i$ for some 3-contractible $X_i \subseteq V(\G_{i+1}) \setminus \{s, t\}$ for each $i = 1, 2, \ldots, k$.
Take such a sequence so that $k$ is minimized.
We show that then $N_\G(X_i) \cap X_j$ if $i \neq j$.

Suppose to the contrary that we have $N_\G(X_i) \cap X_j \neq \emptyset$
for some $i, j$ with $i < j$ (by symmetry, i.e., $N_\G(X_i) \cap X_j \neq \emptyset$ if and only if
$N_\G(X_j) \cap X_i \neq \emptyset$).
Choose such a pair of $i$ and $j$ so that $j - i$ is minimized.
Then, the 3-contractions of $X_{i+1}, \ldots, X_j \subseteq V(\G_{j+1}) \setminus \{s, t\}$
that yield $\G_{i+1}$ from $\G_{j+1}$ can be performed in an arbitrary order
(i.e., do not interfere with each other),
because $N_{\G_{j+1}}(X_h) \cap X_\ell = \emptyset$ if $i+1 \leq h < \ell \leq j$ by the minimality of $j - i$.
Hence, we may assume that $i = j - 1$ by exchanging $X_{i+1}$ and $X_j$ if necessary.

In what follows, we show that $Y \coloneqq X_j \cup X_{j-1}$ is 3-contractible in $\G_{j+1}$ and $\G_{j+1} \three Y = \G_{j-1}$,
which yield a shorter sequence of 3-contractions, contradicting the minimality of $k$.
Since any edge between $X_j$ and $X_{j-1}$ in $\G$ is not removed by the 3-contraction of $X_h$ $(h > j)$,
we have $N_{\G_{j+1}}(X_j) \cap X_{j-1} \supseteq N_\G(X_j) \cap X_{j-1} \neq \emptyset$
and then $1 \leq |N_{\G_{j+1}}(X_j) \cap X_{j-1}| \leq |N_{\G_{j+1}}(X_j)| = 3$.
We discuss the three cases with respect to the value $|N_{\G_{j+1}}(X_j) \cap X_{j-1}|$ separately.
Note that we often use the following fact: the 3-contraction of $X_j$ in $\G_{j+1}$ just replaces the balanced subgraph $\around{\G_{j+1}}{X_j}$
(the only part that is removed) with a balanced triangle on $N_{\G_{j+1}}(X_j)$ (the only part that is added),
which preserves the labels of corresponding paths and cycles in $\G_{j+1}$ and in $\G_j = \G_{j+1} \three X_j$,
where some subpath in $\around{\G_{j+1}}{X_j}$ may be replaced by the corresponding edge in the triangle 
(cf.~Lemma~\ref{lem:contraction}).

\begin{figure}[tb]
 \begin{center}
  \includegraphics[scale=1.1]{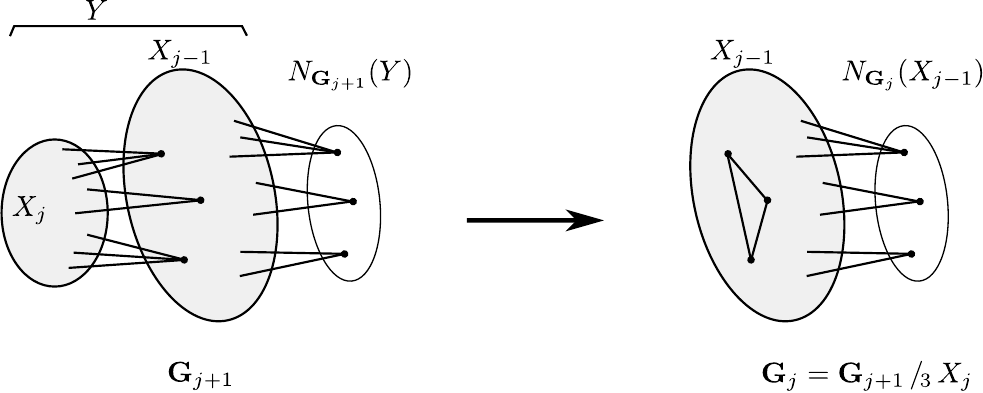}
 \end{center}\vspace{-5mm}
 \caption{When $|N_{\G_{j+1}}(X_j) \cap X_{j-1}| = 3$.}
 \label{fig:3-intersecting}
\end{figure}

Suppose that $|N_{\G_{j+1}}(X_j) \cap X_{j-1}| = 3$, or equivalently $N_{\G_{j+1}}(X_j) \subseteq X_{j-1}$ (see Fig.~\ref{fig:3-intersecting}).
In this case, we have $N_{\G_{j+1}}(Y) = N_{\G_{j+1}}(X_{j-1}) \setminus X_j = N_{\G_j}(X_{j-1})$.
Moreover, $\G_{j+1}[Y]$ is connected as so are $\G_j[X_{j-1}]$ and $\G_{j+1}[X_j]$.
Furthermore, $\around{\G_{j+1}}{Y}$ is balanced as so is $\around{\G_j}{X_{j-1}}$,
because, if $\around{\G_{j+1}}{Y}$ contains an unbalanced cycle, then $\around{\G_j}{X_{j-1}}$ must contain a corresponding unbalanced cycle.
Therefore, $Y$ is indeed 3-contractible in $\G_{j+1}$ and $\G_{j+1} \three Y = \G_{j-1}$.

\begin{figure}[tb]
 \begin{center}\vspace{2mm}
  \includegraphics[scale=1.1]{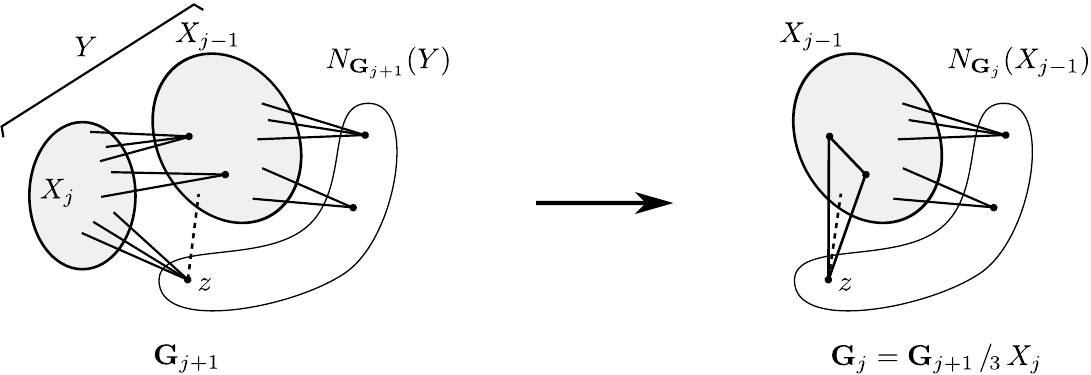}
 \end{center}\vspace{-5mm}
 \caption{When $|N_{\G_{j+1}}(X_j) \cap X_{j-1}| = 2$.}\vspace{-3mm}
 \label{fig:2-intersecting}
\end{figure}

Suppose that $|N_{\G_{j+1}}(X_j) \cap X_{j-1}| = 2$,
and let $z$ be the unique vertex in $N_{\G_{j+1}}(X_j) \setminus X_{j-1}$ (see Fig.~\ref{fig:2-intersecting}).
We then have $z \in N_{\G_{j+1}}(X_j) \setminus X_{j-1} \subseteq N_{\G_{j+1}}(Y)$,
and $z \in N_{\G_j}(X_{j-1})$ due to the edges between $N_{\G_{j+1}}(X_j) \cap X_{j-1}$ and $z$.
By definition of the 3-contraction (of $X_j$ in $\G_{j+1}$),
we also see that $N_{\G_{j+1}}(Y) \setminus \{z\} = N_{\G_{j+1}}(X_{j-1}) \setminus (X_j \cup \{z\}) = N_{\G_j}(X_{j-1}) \setminus \{z\}$,
and hence $N_{\G_{j+1}}(Y) = N_{\G_j}(X_{j-1})$.
The connectivity of $\G_{j+1}[Y]$ and the balancedness of $\around{\G_{j+1}}{Y}$ are assured as with the previous case, and then we are done.

\begin{figure}[tb]
 \begin{center}
  \includegraphics[scale=1.1]{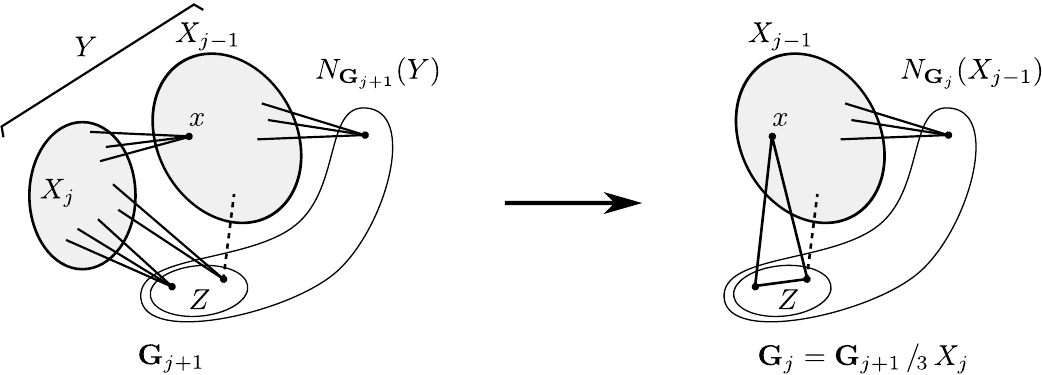}
 \end{center}\vspace{-5mm}
 \caption{When $|N_{\G_{j+1}}(X_j) \cap X_{j-1}| = 1$.}\vspace{-3mm}
 \label{fig:1-intersecting}
\end{figure}

Suppose that $|N_{\G_{j+1}}(X_j) \cap X_{j-1}| = 1$,
let $x$ be the unique vertex in $N_{\G_{j+1}}(X_j) \cap X_{j-1}$, and define $Z \coloneqq N_{\G_{j+1}}(X_j) \setminus \{x\}$ (see Fig.~\ref{fig:1-intersecting}).
We then have $Z = N_{\G_{j+1}}(X_j) \setminus X_{j-1} \subseteq N_{\G_{j+1}}(Y)$,
and $Z \subseteq N_{\G_j}(X_{j-1})$ due to the edges between $x \in X_{j-1}$ and $Z$.
By definition of the 3-contraction (of $X_j$ in $\G_{j+1}$),
we also see that $N_{\G_{j+1}}(Y) \setminus Z = N_{\G_{j+1}}(X_{j-1}) \setminus (X_j \cup Z) = N_{\G_j}(X_{j-1}) \setminus Z$,
and hence $N_{\G_{j+1}}(Y) = N_{\G_j}(X_{j-1})$.
The connectivity of $\G_{j+1}[Y]$ is assured as with the previous two cases,
but we need to be slightly careful to confirm the balancedness of $\around{\G_{j+1}}{Y}$
because the balancedness of $\around{\G_j}{X_{j-1}} = \G_j[X_{j-1} \cup N_{\G_j}(X_{j-1})] - E(N_{\G_j}(X_{j-1}))$ do not care
a possible unbalanced cycle in $\around{\G_{j+1}}{Y}$ that intersects the both vertices in $Z$ with a subpath in $\around{\G_{j+1}}{X_j}$,
which corresponds to an edge $e = Z \subseteq N_{\G_j}(X_{j-1})$ in the triangle after the 3-contraction of $X_j$.

Suppose that $\around{\G_{j+1}}{Y}$ contains such an unbalanced cycle,
and let $C$ be the corresponding unbalanced cycle in $\around{\G_j}{X_{j-1}} + e$
obtained by replacing the subpath in $\around{\G_{j+1}}{X_j}$ between the two vertices in $Z$ with the edge $e$.
If $x \not\in V(C)$, then we can obtain an unbalanced cycle $C'$ in $\around{\G_j}{X_{j-1}}$, a contradiction, from $C$ by replacing the edge $e$ with
the two other edges in the triangle between $x$ and $Z$ (note that $\psi_{\G_{j}}(C) = \psi_{\G_j}(C')$ due to the balancedness of the triangle).
Otherwise, without loss of generality, we regard $x \in V(C)$ as the end vertex of $C = (x = v_0, e_1, v_1, \dots, e_\ell, v_\ell = x)$,
and suppose that $e = e_i = \{v_{i-1}, v_i\}$.
Let $C'$ be the cycle obtained by concatenating $C[v_0, v_{i-1}]$ and the edge $e' = \{v_{i-1}, x\}$ in the triangle,
and $C''$ that obtained by concatenating the edge $e'' = \{x, v_i\}$ in the triangle and $C[v_i, v_\ell]$.
Since $C'$ and $C''$ are in the balanced subgraph $\around{\G_j}{X_{j-1}}$,
we have $\psi_{\G_j}(e', v_{i-1}) = \psi_{\G_j}(C[v_0, v_{i-1}])$ and $\psi_{\G_j}(e'', x) = \psi_{\G_j}(C[v_i, v_\ell])$.
Hence, the triangle $T = (x, e', v_{i-1}, e, v_i, e'', x)$ yielded by the 3-contraction of $X_j$ is unbalanced
as $\psi_{\G_j}(T) = \psi_{\G_j}(C[v_i, v_\ell]) \cdot \psi_{\G_j}(e, v_i) \cdot \psi_{\G_j}(C[v_0, v_{i-1}]) = \psi_{\G_j}(C) \neq \id$, a contradiction. 
Thus we see that $\around{\G_{j+1}}{Y}$ is indeed balanced, which completes the proof. 
\end{proof}

\begin{lemma}\label{lem:G_0}
  For any $(\G, s, t) \in \cD$, we have $(\G, s, t) \in \cDab$ if and only if
  there exists a sequence $\G_0, \G_1, \ldots, \G_r = \G$ of $\Gamma$-labeled graphs satisfying the following conditions$:$
  \begin{itemize}
    \setlength{\itemsep}{.5mm}
  \item
    $\G_0$ consists of two vertices $s$ and $t$ and two parallel edges $e_\alpha$ and $e_\beta$
    between $s$ and $t$ with $\psi_{\G_0}(e_\gamma, t) = \gamma$ for each $\gamma \in \{\alpha, \beta\}$, and
  \item
    $\G_{i-1} = \G_i \two X_i$ for some $X_i \subseteq V(\G_i) \setminus \{s, t\}$
    such that either $\around{\G_i}{X_i}$ is balanced or $(\around{\G_i}{X_i}, x_i, y_i) \in \cDabi^1$,
    where $N_{\G_i}(X_i) = \{x_i, y_i\}$ and $\alpha_i, \beta_i \in \Gamma$ satisfy
    $\alpha_i\beta_i^{-1} \neq \beta_i\alpha_i^{-1}$, for each $i = 1, 2, \ldots, r$.
\end{itemize}
\end{lemma}

\begin{proof}
If there exists such a sequence of $\Gamma$-labeled graphs,
then $(\G, s, t) \in \cDab$ immediately holds by Definition~\ref{def:cDab} and the fact that $(\G_0, s, t) \in \cDab^0 \subseteq \cDab^1$.
We show the converse direction.
Note that the minimality of $\cDab$ in Definition~\ref{def:cDab} requires that, if $(\G, s, t) \in \cDab$, then
there exists a sequence $\G_1, \G_2, \dots, \G_r = \G$ of $\Gamma$-labeled graphs such that
$(\G_1, s, t) \in \cDab^1$ and the second condition in the lemma holds for $i \geq 2$.
Hence, it suffices to show that, for any $(\G, s, t) \in \cDab^1$ with $|V(\G)| \geq 3$,
the 2-contraction of $X \coloneqq V(\G) \setminus \{s, t\}$ in $\G$ results in $\G_0$,
and either $\around{\G}{X}$ is balanced or $(\around{\G}{X}, s, t) \in \cDab^1$.
Since $V(\G \two X) = \{s, t\}$ and $l(\G \two X; s, t) = l(\G; s, t) = \{\alpha, \beta\}$,
we have $\G \two X = \G_0$ by definition (cf.~Definition~\ref{def:2-contraction}).
In what follows, we see that, if $\around{\G}{X}$ is not balanced, then $(\around{\G}{X}, s, t) \in \cDab^1$,
which concludes that either $\around{\G}{X}$ is balanced or $(\around{\G}{X}, s, t) \in \cDab^1$.

Suppose that $\around{\G}{X}$ is not balanced.
Then, Lemma~\ref{lem:balanced2} implies $|l(\around{\G}{X}; s, t)| \geq 2$, and hence $l(\around{\G}{X}; s, t) = l(\G; s, t) = \{\alpha, \beta\}$.
Since $(\G, s, t) \in \cDab^1$,
Lemma~\ref{lem:3-contraction} implies that
there exist nonadjacent 3-contractible vertex sets $Y_1, Y_2, \dots, Y_k \subseteq V(\G) \setminus \{s, t\}$
whose 3-contractions in $\G$ result in a $\Gamma$-labeled graph $\G'$ with $(\G', s, t) \in \cDab^0$ independently of the order.
As $X = V(\G) \setminus \{s, t\}$, the graph $\around{\G}{X}$ is obtained from $\G$ by removing the edges between $s$ and $t$ if exist.
Since any such edge does not make effect on the 3-contractibility of a subset of $V(\G) \setminus \{s, t\}$,
the same sequence of 3-contractions of $Y_1, Y_2, \dots, Y_k$ can be applied to $\around{\G}{X}$,
which results in a $\Gamma$-labeled graph $\G''$ with $l(\G''; s, t) = l(\around{\G}{X}; s, t) = \{\alpha, \beta\}$
such that $\G''$ is a subgraph of $\G'$ and includes $\G' - E(\{s, t\})$ as a subgraph.
No matter in what case $(\G', s, t) \in \cDab^0$ is in Definition~\ref{def:cDab0},
we see $(\G'', s, t) \in \cDab^0$ in the same case, which concludes $(\around{\G}{X}, s, t) \in \cDab^1$.
\end{proof}

For sake of completeness,
we confirm several properties that are intuitively almost trivial.

\begin{lemma}\label{lem:cDab}
  For any $(\G, s, t) \in \cDab$, we have the following properties.
  \begin{itemize}
    \setlength{\itemsep}{.5mm}
  \item[$(1)$]
    Let $\G'$ be the graph obtained from $\G$
    by shifting by $\gamma \in \Gamma$ at $s$.
    Then, $(\G', s, t) \in \cDabp$,
    where $\alpha' \coloneqq \alpha\gamma^{-1}$ and $\beta' \coloneqq \beta\gamma^{-1}$.
  \item[$(2)$]
    Let $\G' \coloneqq \G + s' + e'$ be the graph obtained from $\G$
    by adding a new vertex $s'$ and a new edge $e' = \{s', s\}$ with $\psi_{\G'}(e', s) = \gamma \in \Gamma$.
    Then, $(\G', s', t) \in \cDabp$,
    where $\alpha' \coloneqq \alpha\gamma$ and $\beta' \coloneqq \beta\gamma$.
  \item[$(3)$]
    If $\G = \G' \two X$ for a $\Gamma$-labeled graph $\G'$
    and $X \subseteq V(\G') \setminus \{s, t\}$
    with $(\around{\G'}{X}, x, y) \in \cDabp$,
    where $N_{\G'}(X) = \{x, y\}$ and
    $\alpha', \beta' \in \Gamma$ satisfy
    $\alpha'{\beta'}^{-1} \neq \beta'{\alpha'}^{-1}$,
    then $(\G', s, t) \in \cDab$.
  \end{itemize}
\end{lemma}

\begin{proof}
$(1)$
We first confirm that,
if $(\G, s, t) \in \cDab^0$, then $(\G', s, t) \in \cDabp^0$.
Case (A1) and Case (C) are obvious (cf.~Definition~\ref{def:cDab0}).
In Case (A2),
apply shifting by $\gamma$ at each $v \in V(\G) \setminus \{s, t\}$,
and in Case (B), do so at $v_1$ and $v_2$.

We next show that,
if $(\G, s, t) \in \cDab^1$, then $(\G', s, t) \in \cDabp^1$.
Suppose that $(\G, s, t) \in \cDab^1$.
Then, by Lemma~\ref{lem:3-contraction},
one can obtain a $\Gamma$-labeled graph $\tilde{\G}$ such that $(\tilde{\G}, s, t) \in \cDab^0$
from $\G$ by applying independent 3-contractions.
Since any shifting does not make effect on the balancedness,
the same 3-contractions can be applied to $\G'$,
which results in a $\Gamma$-labeled graph $\tilde{\G}'$ with $(\tilde{\G}', s, t) \in \cDabp^0$.
This concludes $(\G', s, t) \in \cDabp^1$.

By Lemma~\ref{lem:G_0}, one can obtain $\G_0$ (a $\Gamma$-labeled graph consisting of only two parallel edges between $s$ and $t$)
from $\G_r = \G$ by a sequence of 2-contractions of some $X_i$ $(i = r, \dots, 2, 1)$ such that
either $\around{\G_i}{X_i}$ is balanced or $(\around{\G_i}{X_i}, x_i, y_i) \in \cDabi^1$.
We prove that the same sequence of 2-contractions can be applied to $\G'$.

Define $\G'_r \coloneqq \G'$.
Then, we can inductively construct a $\Gamma$-labeled graph
$\G'_{i-1} \coloneqq \G'_i \two X_i$,
which coincides with the one obtained from $\G_{i-1}$ by shifting by $\gamma$ at $s$.
This means that we finally obtain a $\Gamma$-labeled graph $\G'_0$
from $\G'$ by the 2-contractions of $X_i$ $(i = r, \ldots, 2, 1)$,
which satisfies $(\G'_0, s, t) \in \cDabp^0$ (in Cases (A) and (C)).
Thus we have $(\G', s, t) \in \cDabp$,
since either $\around{\G'_i}{X_i}$ is balanced or $(\around{\G'_i}{X_i}, x_i, y_i) \in \cDabpi^1$,
where $\alpha'_i = \alpha_i$ and $\beta'_i = \beta_i$ if $s \not\in \{x_i, y_i\}$, and
$\alpha'_i = \alpha_i\gamma^{-1}$ and $\beta'_i = \beta_i\gamma^{-1}$ otherwise
(assume $x_i = s$ without loss of generality by the symmetry of $x_i$ and $y_i$).

\medskip\noindent
$(2)$
Similarly to (1), by Lemma~\ref{lem:G_0},
one can obtain $\G_0$ from $\G$ by a sequence of 2-contractions.
The same sequence of 2-contractions can be applied to $\G'$,
which results in $\G'_0 \coloneqq \G_0 + s' + e'$ such that $(\G'_0, s', t) \in \cDabp^0$ (in Cases (A) and (C)).
Thus we have $(\G', s, t) \in \cDabp$.

\medskip\noindent
$(3)$
Similarly, by Lemma~\ref{lem:G_0}, there exists a sequence $\bH_0, \bH_1, \ldots, \bH_r = \around{\G'}{X}$
such that $\bH_0$ consists of only two parallel arcs from $x$ to $y$ whose labels are $\alpha'$ and $\beta'$,
and $\bH_{i-1}$ is obtained from $\bH_i$ by some 2-contraction.
The same sequence of 2-contractions can be applied to $\G'$,
which results in $\G$.
This implies that $(\G', s, t) \in \cDab$.
\end{proof}

By Lemma~\mbox{\ref{lem:cDab}-$(1)$},
it suffices to consider the case when
$\beta = \id$ and $\alpha^{-1} \neq \alpha$ (i.e., $\alpha^2 \neq \id$).
The following lemma gives a useful characterization of $\cD_{\id,\, \alpha}^0$
in Case (C) (cf.~Definition~\ref{def:cDab0}).

\begin{lemma}\label{lem:cDab0}
  Suppose that $\alpha^{-1} \neq \alpha \in \Gamma$.
  For any triplet $(\G, s, t) \in \cD_{\id,\, \alpha}^0$
  in Case {\rm (C)} in Definition~$\ref{def:cDab0}$ with $|\delta_\G(v)| \geq 3$ for all $v \in V(\G) \setminus \{s, t\}$,
  there exists an $(s, t)$-equivalent $\Gamma$-labeled graph $\G'$
  that can be embedded in the plane with the following conditions $($see Fig.~$\ref{fig:cD01}$$)$.
  \begin{itemize}
    \setlength{\itemsep}{.5mm}
  \item[$1.$]
    The edge set $E(\G)$ is partitioned into $E^0$ and $E^1$
    $($i.e., $E^0 \cup E^1 = E(\G)$ and $E^0 \cap E^1 = \emptyset$$)$,
    where $E^0 = \{\, e \in E(\G) \mid \psi_{\G'}(e, v) = \id~(\forall v \in e) \,\}$. 
  \item[$2.$]
    There exists an $s$--$t$ path $P = (s = u_0, e_1, u_1, \ldots, e_\ell, u_\ell = t)$
    along the outer boundary of $\G' - E^1$ such that\vspace{-1.5mm}
    \begin{itemize} 
      \setlength{\itemsep}{.5mm}
    \item
      every edge $e \in E^1$ connects two vertices $u_i, u_j \in V(P)$ with $i < j$ and $\psi_{\G'}(e, u_j) = \alpha$,
      and is embedded in the outer face of $\G' - E^1$, and
    \item
      for any distinct edges $e_1 = \{u_{i_1}, u_{j_1}\}$ and $e_2 = \{u_{i_2}, u_{j_2}\}$ in $E^1$ with $i_1 < j_1$ and $i_2 < j_2$,
      one of two paths $P[u_{i_1}, u_{j_1}]$ and $P[u_{i_2}, u_{j_2}]$ is a subpath of the other.
    \end{itemize}
  \end{itemize}
\end{lemma}

\begin{figure}[htbp]
  \begin{center}
   \includegraphics[scale=0.8]{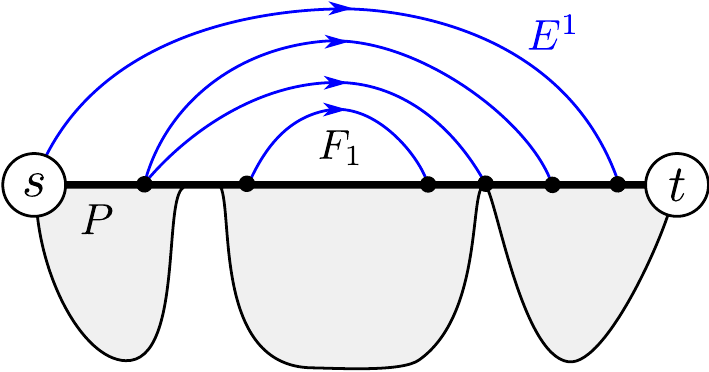}
  \end{center}\vspace{-5mm}
 \caption{An $(s, t)$-equivalent embedding of $(\G, s, t) \in \cD_{\id,\, \alpha}^0$ in Case (C).}\vspace{-2mm}
 \label{fig:cD01}
\end{figure}

\begin{proof}
Fix an embedding of $\G$ with the conditions of Case (C), and 
let $P_0$ and $P_1$ be the $s$--$t$ paths along the boundary of the outer face
$F_0$ of $\G$ whose labels are $\id$ and $\alpha$, respectively.

Let $G^\ast$ be the dual graph of the underlying graph $G = (V, E)$ of $\G$,
i.e., the vertex set of $G^\ast$ is the face set of $G$,
the edge set of $G^\ast$ coincides with $E$,
and each two faces whose boundaries share an edge $e \in E$ in $G$ are connected by the same-named edge $e$ in $G^\ast$.
Take a shortest $F_1$--$F_0$ path $Q$ in $G^\ast - E(P_0)$.
We prove that the conditions hold with $E^1 = E(Q)$.

Let $\G'' \coloneqq \G - E(Q)$.
Then, $\G''$ is connected since $Q$ is a shortest path disjoint from $P_0$,
and is balanced since $F_1$ is the unique unbalanced inner face of $\G$.
Since $\psi_{\G''}(P_0) = \psi_{\G}(P_0) = \id$,
by Lemma~\ref{lem:shifting}, by shifting at vertices in $V \setminus \{s, t\}$,
we can obtain from $\G$ an $(s, t)$-equivalent $\Gamma$-labeled graph $\G'$
such that $\psi_{\G'}(e, v) = \id$ for every $e \in E(\G'') = E \setminus E(Q)$ and $v \in e$.
For every $v \in V \setminus \{s, t\}$, since $|\delta_{\G}(v)| \geq 3$,
we have $|\delta_{\G''}(v)| \geq 2$
(otherwise, there exists a shortcut for $Q$ around $v$ with $|\delta_{\G''}(v)| = 1$, which contradicts that $Q$ is shortest),
and hence $|N_{\G''}(v)| \geq 2$ (otherwise, $\G''$ has parallel edges incident to $v$ with $|N_{\G''}(v)| = 1$, which contradicts that $\G''$ is balanced).
Thus we can take an $s$--$t$ path $P$ in $\G'' = \G - E(Q)$ so that $P$ and $P_0$ form the outer boundary of $\G''$.
Since $\psi_{\G'}({\rm bd}(F)) = \psi_\G({\rm bd}(F)) = \id$ holds for any face $F$ other than $F_0$ or $F_1$,
the edges in $E(Q)$ must satisfy the second condition in $\G'$.
\end{proof}

The following two lemmas are utilized to derive a contradiction
by constructing an $s$--$t$ path of label
$\gamma \not\in \Gamma \setminus \{\alpha, \beta\}$ in $\G$,
where $(\G, s, t) \in \cD$.

\begin{lemma}\label{lem:unbalanced_cycle}
  Suppose that a triplet $(\G, s, t) \in \cD$ satisfies $l(\G; s, t) = \{\alpha', \beta'\}$ for some distinct $\alpha', \beta' \in \Gamma$.
  Then, $\alpha'{\beta'}^{-1} \neq \beta'{\alpha'}^{-1}$ if and only if $\G$ contains no unbalanced cycle $C$ with $\psi_\G(\overline{C}) = \psi_\G(C)$.
\end{lemma}

\begin{proof}
We first note that the equality $\psi_\G(\overline{C}) = \psi_\G(C)$ is equivalent to $\psi_\G(C)^{-1} = \psi_\G(C)$ (or $\psi_\G(C)^2 = \id$),
and it does not depend on the choice of the end vertex of the cycle $C$.

Suppose 
that $\G$ contains an unbalanced cycle $C$ with $\psi_\G(\overline{C}) = \psi_\G(C)$,
and we show $\alpha'{\beta'}^{-1} = \beta'{\alpha'}^{-1}$.
By Menger's theorem (cf.~the proof of Lemma~\ref{lem:balanced2}),
for some distinct vertices $x, y \in V(C)$,
take an $s$--$x$ path $P$ and a $y$--$t$ path $Q$ in $\G$
so that $V(P) \cap V(C) = \{x\}$, $V(Q) \cap V(C) = \{y\}$,
and $V(P) \cap V(Q) = \emptyset$,
and choose $y$ as the end vertex of $C$.
Let $\alpha'' \coloneqq \psi_\G(C[x, y])$ and $\beta'' \coloneqq \psi_\G(\overline{C}[x, y])$,
which are distinct since $C$ is unbalanced.
We then have $\alpha''{\beta''}^{-1} = \psi_\G(C) = \psi_\G(\overline{C}) = \beta''{\alpha''}^{-1}$.
By extending $C[x, y]$ and $\overline{C}[x, y]$ using $P$ and $Q$,
we obtain two $s$--$t$ paths in $\G$ whose labels are
$\psi_\G(Q)\cdot\alpha''\cdot\psi_\G(P)$ and
$\psi_\G(Q)\cdot\beta''\cdot\psi_\G(P)$, which are distinct;
one is $\alpha'$ and the other is $\beta'$.
Since $\alpha''{\beta''}^{-1} = \beta''{\alpha''}^{-1}$,
we have $\alpha'{\beta'}^{-1} = \beta'{\alpha'}^{-1}$.

To prove the converse direction, suppose that $\alpha'{\beta'}^{-1} = \beta'{\alpha'}^{-1}$, and we show $\psi_\G(C)^{-1} = \psi_\G(C)$.
By Proposition~\ref{prop:2-cyclic}, we may assume that every arc around $s$ has label $\alpha'$ or $\beta'$
and every other arc has label $\id$ or $\alpha'{\beta'}^{-1}$ by shifting at some vertices in $V(\G) \setminus \{s, t\}$ if necessary
(note that any shifting preserves 
whether $\psi_\G(C)^{-1} = \psi_\G(C)$ or not). 
Since $\alpha' \neq \beta'$, by Lemma~\ref{lem:balanced2}, there exists an unbalanced cycle $C$ in $\G$,
whose label must be $\id$ or $\alpha'{\beta'}^{-1} = \beta'{\alpha'}^{-1}$ by the assumption, where we choose some vertex other than $s$ as the end vertex of $C$.
Thus we have $\psi_\G(C)^{-1} = \psi_\G(C)$.
\end{proof}

In the proof of Theorem~\ref{thm:characterization} starting in the next section,
we assume $(\G, s, t) \in \cD$ and $l(\G; s, t) = \{\alpha, \beta\}$ with $\alpha\beta^{-1} \neq \beta\alpha^{-1}$,
which implies that $\G$ contains no unbalanced cycle $C$ with $\psi_\G(\overline{C}) = \psi_\G(C)$.
In particular, this is also true for any subgraph $\bH$ of $\G$,
and hence, if $(\bH, x, y) \in \cD$ and $l(\bH; x, y) = \{\alpha', \beta'\}$ for distinct $\alpha', \beta' \in \Gamma$,
then $\alpha'{\beta'}^{-1} \neq \beta'{\alpha'}^{-1}$.

\begin{lemma}\label{lem:2-2labels}
  For a triplet $(\G, s, t) \in \cD$,
  if there exist two paths $P_i$ $(i = 1, 2)$ in $\G$
  with the following conditions $($see Fig.~$\ref{fig:2-2labels}$$)$,
  then $|l(\G; s, t)| \geq 3$$:$
  \begin{itemize}
    \setlength{\itemsep}{.5mm}
  \item
    $P_i$ is an $s$--$x_i$ path for $i = 1, 2$ and $x_i \in V(\G) \setminus \{s, t\}$,
  \item
    $\psi_\G(P_1) \neq \psi_\G(P_2)$, and
  \item
    $\{\alpha', \beta'\} \subseteq l(\G - (V(P_i) \setminus \{x_i\}); x_i, t)$
    for $i = 1, 2$, for some $\alpha', \beta' \in \Gamma$ with $\alpha'{\beta'}^{-1} \neq \beta'{\alpha'}^{-1}$.
  \end{itemize}
\end{lemma}

\begin{figure}[htbp]
 \begin{center}
  \includegraphics[scale=0.8]{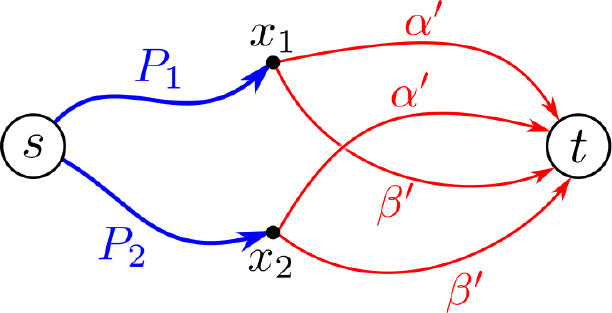}
 \end{center}\vspace{-5mm}
 \caption{Combination of two labels leads to at least three labels.}\vspace{-2mm}
 \label{fig:2-2labels}
\end{figure}

\begin{proof}
For each $i = 1, 2$,
by concatenating $P_i$ and
each of two $x_i$--$t$ paths 
in $\G - (V(P_i) \setminus \{x_i\})$ whose labels are $\alpha'$ and $\beta'$,
we construct four $s$--$t$ paths whose labels are
$\gamma_1 \coloneqq \alpha'\cdot\psi_\G(P_1)$, $\gamma_2 \coloneqq \beta'\cdot\psi_\G(P_1)$,
$\gamma_3 \coloneqq \alpha'\cdot\psi_\G(P_2)$, and $\gamma_4 \coloneqq \beta'\cdot\psi_\G(P_2)$.

Suppose to the contrary that $|l(\G; s, t)| \leq 2$.
Since $\gamma_1 \neq \gamma_2 \neq \gamma_4 \neq \gamma_3 \neq \gamma_1$,
we must have $\gamma_1 = \gamma_4$ and $\gamma_2 = \gamma_3$.
Hence, $\psi_\G(P_1) = {\alpha'}^{-1}\cdot\beta'\cdot\psi_\G(P_2)$
and $\psi_\G(P_1) = {\beta'}^{-1}\cdot\alpha'\cdot\psi_\G(P_2)$,
which implies ${\alpha'}^{-1}\beta' = {\beta'}^{-1}\alpha'$.
This is equivalent to $\alpha'{\beta'}^{-1} = \beta'{\alpha'}^{-1}$,
a contradiction.
\end{proof}

\subsection{Minimal counterexample}\label{sec:counterexample}
In what follows, we prove that,
for any $(\G, s, t) \in \cD$ and $\alpha, \beta \in \Gamma$ with $\alpha\beta^{-1} \neq \beta\alpha^{-1}$,
if $l(\G; s, t) = \{\alpha, \beta\}$, then $(\G, s, t) \in \cDab$,
which completes the proof of Theorem~\ref{thm:characterization}.
To derive a contradiction, suppose to the contrary that
there exist $(\G, s, t) \in \cD$ and $\alpha, \beta \in \Gamma$ with $\alpha\beta^{-1} \neq \beta\alpha^{-1}$
such that $l(\G; s, t) = \{ \alpha, \beta \}$ but $(\G, s, t) \not\in \cDab$.
We choose such $(\G, s, t) \in \cD$ and $\alpha, \beta \in \Gamma$
so that $|V(\G)| + |E(\G)|$ is minimized.
Note that we have $|V(\G)| \geq 3$ obviously.

Moreover, we may assume $\beta = \id$ and $\alpha^{-1} \neq \alpha$ (i.e., $\alpha^2 \neq \id$) as follows.
Let $\G'$ be the $\Gamma$-labeled graph obtained from $\G$ by shifting by $\beta$ at $s$.
Since $l(\G; s, t) = \{\alpha, \beta\}$,
we then have $l(\G'; s, t) = \{ \id, \alpha' \}$ (cf.~Definition~\ref{def:shifting}), where $\alpha' \coloneqq \alpha\beta^{-1} \neq \beta\alpha^{-1} = {\alpha'}^{-1}$.
Lemma~\mbox{\ref{lem:cDab}-(1)} implies that $(\G, s, t) \in \cDab$ if and only if $(\G', s, t) \in \cD_{\id,\, \alpha'}$,
and hence $(\G', s, t) \not\in \cD_{\id,\, \alpha'}$.
This means that $(\G', s, t) \in \cD$ with $\id, \alpha' \in \Gamma$ is also a minimal counterexample,
and hence we can choose it instead of $(\G, s, t) \in \cD$ with $\alpha, \beta \in \Gamma$.
We, however, forget this assumption (which is not essential here)
in this section, and recall it later in Section~\ref{sec:case}.

The minimality assures that $\G$ contains no contractible vertex set as follows.
Recall that a vertex set $X \subseteq V(\G) \setminus \{s, t\}$ is said to be 2-contractible
if $|N_\G(X)| = 2$ holds, $\around{\G}{X}$ is connected, and $\around{\G}{X} \neq \G$ (cf.~Definition~\ref{def:2-contraction}).
In particular, $V(\G) \setminus \{s, t\}$ is not 2-contractible unless $\{s, t\} \in E(\G)$.

\begin{claim}\label{cl:2-contractible}
  There is no $2$-contractible vertex set in $\G$.
\end{claim}

\begin{proof}
Suppose to the contrary that $\G$ contains
a 2-contractible vertex set $X \subseteq V(\G) \setminus \{s, t\}$
with $N_\G(X) = \{x, y\}$.
Since $(\G, s, t) \in \cD$, we also have $(\around{\G}{X}, x, y) \in \cD$.
If $|l(\around{\G}{X}; x, y)| \geq 3$,
then we also have $|l(\G; s, t)| \geq 3$
(since $\G$ contains two disjoint paths between $\{s, t\}$ and $\{x, y\}$
by Lemma~\ref{lem:cD} and Menger's theorem), a contradiction. 
Hence, we have either $|l(\around{\G}{X}; x, y)| = 1$ or $l(\around{\G}{X}; x, y) = \{\alpha', \beta'\}$ for some distinct $\alpha', \beta' \in \Gamma$, for which $\alpha'{\beta'}^{-1} \neq \beta'{\alpha'}^{-1}$ by Lemma~\ref{lem:unbalanced_cycle}.
In the former case, Lemma~\ref{lem:balanced2} implies that $\around{\G}{X}$ is balanced.
In the latter case, since $\around{\G}{X}$ is a proper subgraph of $\G$,
the minimality of $\G$ implies $(\around{\G}{X}, x, y) \in \cDabp$.
Thus, in both cases, $(\G \two X, s, t) \not\in \cDab$ (by Definition~\ref{def:cDab} and Lemma~\ref{lem:cDab}-(3), respectively),
and Lemma~\ref{lem:contraction} implies that $(\G \two X, s, t) \in \cD$ and $l(\G \two X; s, t) = \{\alpha, \beta\}$, 
which contradicts that $(\G, s, t)$ is a minimal counterexample.
\end{proof}

\begin{claim}\label{cl:3-contractible}
  There is no $3$-contractible vertex set in $\G$.
\end{claim}

\begin{proof}
Suppose to the contrary that $\G$ contains
a 3-contractible vertex set $X \subseteq V(\G) \setminus \{s, t\}$.
By Lemma~\ref{lem:contraction},
the minimality of $\G$ implies $(\G \three X, s, t) \in \cDab$.
Then, by Lemma~\ref{lem:G_0},
there exists a sequence $\G_0, \G_1, \ldots, \G_r = \G \three X$ such that $\G_0$ consists of only two parallel arcs from $s$ to $t$ with labels $\alpha$ and $\beta$,
and $\G_{i-1}$ is obtained from $\G_i$ by some 2-contraction for each $i = 1, 2, \dots, r$.

Since $|V(\G_0)| = 2$ and $|V(\G_r)| \geq |N_\G(X)| = 3$,
we have $N_\G(X) \cap (V(\G_j) \setminus V(\G_{j-1})) \neq \emptyset$ for some $j$ with $1 \leq j \leq r$.
Let $j$ be the maximum such index,
and then we can apply to $\G$ the same sequence of 2-contractions as that to construct $\G_j$ from $\G_r = \G \three X$.
Let $\bH_j$ be the resulting graph,
$Y \coloneqq V(\G_j) \setminus V(\G_{j-1})$ (i.e., $\G_{j-1} = \G_j \two Y$),
and $Z \coloneqq X \cup Y \subseteq V(\bH_j)$.
Note that Lemma~\ref{lem:G_0} requires that either $\around{\G_j}{Y}$ is balanced or $(\around{\G_j}{Y}, x, y) \in \cDabp^1$ for some $\alpha', \beta' \in \Gamma$ with $\alpha'{\beta'}^{-1} \neq \beta'{\alpha'}^{-1}$, where $N_{\G_j}(Y) = \{x, y\}$.
In what follows, we show that $\bH_j \two Z = \G_j \two Y$ and either $\around{\bH_j}{Z}$ is balanced or $(\around{\bH_j}{Z}, x, y) \in \cDabp^1$ (respectively).
This concludes that almost the same sequence of 2-contractions in which $Y$ is just replaced by $Z$ can be applied to $\G$
and results in $\G_0$, which leads to $(\G, s, t) \in \cDab$ by Lemma~\ref{lem:G_0}, a contradiction.

By the maximality of $j$, we have $N_{\bH_j}(X) = N_\G(X)$, $\bH_j[X] = \G[X]$, and $\around{\bH_j}{X} = \around{\G}{X}$,
which implies $X$ is 3-contractible in $\bH_j$ as well as in $\G$.
Moreover, $\around{\bH_j}{Z}$ is a subgraph of $\bH_j$ that includes $\around{\bH_j}{X}$ as a subgraph,
and hence this is true also in $\around{\bH_j}{Z}$.
If $|N_{\bH_j}(X) \cap Y| \geq 2$, then $\around{\G_j}{Y} = \around{\bH_j}{Z} \three X$,
and the claim immediately follows (recall that the 3-contraction of $X$ in $\bH_j$ is replacing a balanced subgraph $\around{\bH_j}{X}$
with a balanced triangle on $N_{\bH_j}(X)$, and also Definition~\ref{def:cDab1}).
Otherwise, let $z$ be a unique vertex in $N_{\bH_j}(X) \cap Y$.
Then, by the definition of 3-contraction, $z$ and the other two vertices in $N_{\bH_j}(X)$ are adjacent in $\G_j$ as well as in $\G_r = \G \three X$.
Hence, $N_{\bH_j}(X) \setminus \{z\} \subseteq N_{\G_j}(Y) = \{x, y\}$, which implies $N_{\bH_j}(X) = \{x, y, z\}$.
We then have $\around{\G_j}{Y} = \around{\bH_j}{Z} \three X - e$, where $e = \{x, y\} \in E(\around{\bH_j}{Z} \three X)$ is an edge added by the 3-contraction of $X$ in $\around{\bH_j}{Z}$.
Since $e$ is in the balanced triangle on $N_{\bH_j}(X) = \{x, y, z\}$ whose other two edges are contained in $\around{\G_j}{Y}$
and form an $x$--$y$ path of label $l(\around{\bH_j}{Z}; x, y)$,
if $\around{\G_j}{Y}$ is balanced then so is $\around{\bH_j}{Z}$,
and if $(\around{\G_j}{Y}, x, y) \in \cDabp^1$ then $(\around{\bH_j}{Z}, x, y) \in \cDabp^1$.
Thus the claim holds in either case, and we complete the proof.
\end{proof}

Fix an arbitrary edge $e_0 = \{s, v_0\} \in \delta_\G(s)$, and let $\G' \coloneqq \G - e_0$.
Claim~\ref{cl:2-contractible} implies that $\G$ contains no edge between $s$ and $t$, and hence $v_0 \neq t$.
The following two claims lead to $(\G', s, t) \in \cDab$ (since $(\G, s, t) \not\in \cDab$ is a minimal counterexample).

\begin{claim}\label{cl:cD}
$(\G', s, t) \in \cD$. 
\end{claim}

\begin{proof}
By Lemma~\ref{lem:cD},
it suffices to show that $\G' + e_{st}$ is 2-connected, where $e_{st} = \{s, t\}$ is a new edge (with an arbitrary label),
and suppose to the contrary that it is not.
If $s$ is isolated in $\G'$, then $(\G', v_0, t) \in \cDabp$ by the minimality of $\G$
(where $\alpha' = \alpha \cdot \psi_\G(e_0, v_0)^{-1}$ and $\beta' = \beta \cdot \psi_\G(e_0, v_0)^{-1}$),
and hence $(\G, s, t) \in \cDab$ by Lemma~\ref{lem:cDab}-(2), a contradiction.
Otherwise, $\G'$ is connected (because $(\G, s, t) \in \cD$)
and has a 1-cut $w \in V(\G)$ that separates some nonempty $X \subseteq V(\G) \setminus \{s, t, w\}$ from both $s$ and $t$
(note that possibly $w \in \{s, t\}$).
Take such $X$ so that $\G'[X]$ is connected.
Since $N_{\G'}(X) = \{w\}$ holds and $\G = \G' + e_0$ has no such 1-cut (because $(\G, s, t) \in \cD$),
we have $v_0 \in X$ and hence $N_\G(X) = \{s, w\}$.
Claim~\ref{cl:2-contractible} says that $X$ is not 2-contractible in $\G$, and hence $X = V(\G) \setminus \{s, t\}$ (and $\{s, t\} \not\in E(\G)$).
Then, $w = t$ cannot separate $X$ from $s$ in $\G'$, a contradiction.
\end{proof}

\begin{claim}\label{cl:2label}
  $l(\G'; s, t) = \{\alpha, \beta\}$.
\end{claim}

\begin{proof}
Since each $s$--$t$ path in $\G'$ is also in $\G$,
we see that $l(\G'; s, t) \subseteq l(\G; s, t) = \{\alpha, \beta\}$.
Suppose to the contrary that $|l(\G'; s, t)| = 1$.
Then, $\G'$ is balanced by Lemma~\ref{lem:balanced2} and Claim~\ref{cl:cD},
and hence $\G - s$ is also balanced (which is a subgraph of $\G' = \G - e_0$).
This implies that $(\G, s, t) \in \cDab^0 \subseteq \cDab$
(cf.~Case (A1) in Definition~\ref{def:cDab0}), a contradiction.
\end{proof}

\subsection{Case analysis}\label{sec:case}
Note again that, by the minimality of $\G$, Claims~\ref{cl:cD} and \ref{cl:2label} imply $(\G', s, t) \in \cDab$.
We consider the following two cases separately: when $(\G', s, t) \in \cDab^1$ and when not.
In other words, the former case does not need any 2-contraction to reduce $(\G', s, t)$ to some triplet in $\cDab^0$ (cf.~Definition~\ref{def:cDab1} and Lemma~\ref{lem:3-contraction}),
and the latter involves some 2-contraction.

\subsubsection{Case 1: Without 2-contraction (when $(\G', s, t) \in \cDab^1$)}\label{sec:case1}
Suppose that $(\G', s, t) \in \cDab^1$.
If $(\G', s, t) \not\in \cDab^0$, then, by Lemma~\ref{lem:3-contraction},
there exist nonadjacent 3-contractible vertex sets whose 3-contractions reduce $(\G', s, t) \in \cDab^1$ to some triplet in $\cDab^0$.
Moreover, any 3-contractible vertex set in $\G'$ contains $v_0$ as seen below,
and hence there exists $X \subseteq V(\G) \setminus \{s, t\}$ such that $(\G' \three X, s, t) \in \cDab^0$.
Suppose to the contrary that $\G'$ contains a 3-contractible vertex set $X' \subseteq V(\G) \setminus \{s, t\}$ with $v_0 \not\in X'$. 
Then, the edge $e_0 = \{s, v_0\}$ is disjoint from $X'$,
and hence $N_{\G'}(X') = N_{\G}(X')$, $\G'[X'] = \G[X']$, and $\around{\G'}{X'} = \around{\G}{X'}$.
This implies that $X'$ is 3-contractible also in $\G$, contradicting Claim~\ref{cl:3-contractible}.

If $(\G', s, t) \not\in \cDab^0$, take $X \subseteq V(\G) \setminus \{s, t\}$ with $(\G' \three X, s, t) \in \cDab^0$
and let $\tilde\G \coloneqq \G' \three X$, and otherwise $\tilde\G \coloneqq \G'$ (i.e., if $(\G', s, t) \in \cDab^0$).
Here, recall that we may assume that $\beta = \id$ and $\alpha^{-1} \neq \alpha$ (i.e., $\alpha^2 \neq \id$),
i.e., in what follows, we assume $(\tilde\G, s, t) \in \cD_{\id,\, \alpha}^0$.
We discuss the three cases in Definition~\ref{def:cDab0} separately.

\medskip
\noindent\underline{{\bf Case~1.1.}~~When $(\tilde\G, s, t) \in \cD_{\id,\, \alpha}^0$ is in Case (A).}

\medskip
We first see that also $(\G', s, t) \in \cD_{\id,\, \alpha}^0$ is in Case (A).
The case when $\tilde\G = \G'$ is trivial, and we consider the case when $\tilde\G = \G' \three X$ for some $X \subseteq V(\G) \setminus \{s, t\}$.
By the symmetry of $s$ and $t$ (corresponding to Cases (A1) and (A2)), we show that, if $(\G' \three X, s, t) \in \cD_{\id,\, \alpha}^0$ is in Case (A1),
then so is $(\G', s, t) \in \cD_{\id,\, \alpha}^0$.

By definition, we can obtain from $\G' \three X$ an $(s, t)$-equivalent graph $\tilde\bH$
by shifting at vertices in $V(\G') \setminus (X \cup \{s, t\})$ in which almost all arcs are with label $\id$
and the only exception is that some but not all arcs leaving $s$ are with label $\alpha$ (cf.~Fig.~\ref{fig:cDab0A1}).
Let $\bH$ be the $(s, t)$-equivalent graph obtained from $\G'$ by the same shifting operations.
Then, $\tilde\bH = \bH \three X$ is obtained just by replacing a balanced subgraph $\around{\bH}{X}$ with a balanced triangle on $N_{\bH}(X) = N_{\G'}(X)$.
Let $N_{\bH}(X) = \{x, y, z\}$ so that $x = s$ if $s \in N_{\bH}(X)$. 
Then, an arc $yz$ (or $zy$) exists in $\tilde{\bH} = \bH \three X$ by the definition of 3-contraction
and is with label $\id$ by $y \neq s \neq z$, which implies $l(\around{\bH}{X}; y, z) = \id$.
Hence, by applying Lemma~\ref{lem:shifting} to $\around{\bH}{X}$ with $y$ and $z$,
we can obtain a $(y, z)$-equivalent graph by shifting at vertices in $X \cup \{x\}$ in which all the arcs are with label $\id$.
That is, by applying to $\around{\bH}{X}$ the same shifting operations at the vertices in $X$ (i.e., except for $x$),
we obtain an $(x, y)$- and $(x, z)$-equivalent graph in which almost all arcs in $\around{\bH}{X}$ are with label $\id$ and the only exception is that
all arcs around $x$ leave $x$ with label $\gamma \coloneqq l(\around{\bH}{X}; x, y)$ (by replacing some arcs with equivalent arcs if necessary).
Since an arc $xy$ with label $\gamma$ (or an equivalent arc) exists in $\tilde\bH = \bH \three X$ by the definition of the 3-contraction of $X$,
we have $\gamma = \id$ unless $x = s$ and $\gamma \in \{\id, \alpha\}$ if $x = s$.
In either case, by combining the above shifting operations (at the vertices in $V(\G') \setminus (X \cup \{s, t\})$ and in $X$),
we can obtain from $\G'$ an $(s, t)$-equivalent graph with the condition of Case (A1).
Thus $(\G', s, t) \in \cD_{\id,\, \alpha}^0$ is in Case (A1).

We then assume that $\G' = \G - e_0$ satisfies the condition of Case (A)
by shifting at vertices in $V(\G) \setminus \{s, t\}$ in advance of removing $e_0$ from $\G$ if necessary.
Since $\G$ contains no 2-contractible vertex set (by Claim~\ref{cl:2-contractible}),
$\G - \{s, t\}$ is connected,
which implies that there exists a $v_0$--$w$ path in $\G - \{s, t\}$
for each $w \in N_\G(t)$ (recall that $v_0 \neq t$).
Therefore, if the edge $e_0 = \{s, v_0\}$ violates the condition of Case (A)
(i.e., $\psi_\G(e_0, v_0) \not\in \{\id, \alpha\}$ in Case (A1),
and $\psi_\G(e_0, v_0) \neq \id$ in Case (A2)),
then it is easy to see that $|l(\G; s, t)| \geq 3$
(see Figs.~\ref{fig:case_1.1.1} and \ref{fig:case_1.1.2}), a contradiction.
Note that we use Lemma~\ref{lem:2-2labels} in Case (A2)
(let $P_1 \coloneqq (s)$ and $P_2 \coloneqq (s, e_0, v_0)$).

\begin{figure}[htbp]\vspace{2mm}\hspace{-3mm}
  \begin{tabular}{cc}
    \begin{minipage}[b]{0.5\hsize}
      \begin{center}
        \includegraphics[scale=0.7]{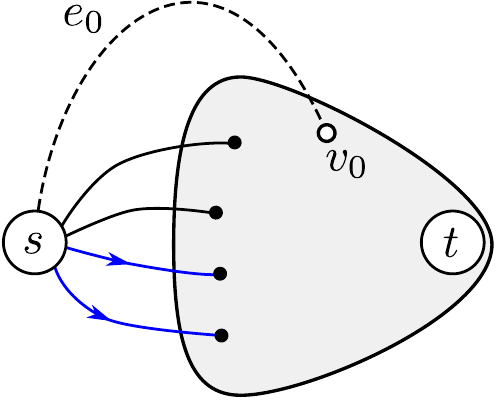}
      \end{center}\vspace{-5mm}
      \caption{$(\G', s, t) \in \cD_{\id,\,\alpha}^0$ is in Case (A1).}
      \label{fig:case_1.1.1}
    \end{minipage}
    \begin{minipage}[b]{0.5\hsize}
      \begin{center}
        \includegraphics[scale=0.7]{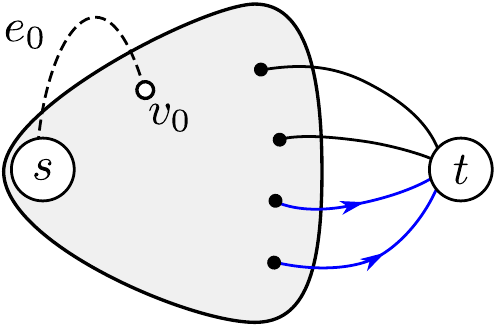}
      \end{center}\vspace{-5mm}
      \caption{$(\G', s, t) \in \cD_{\id,\,\alpha}^0$ is in Case (A2).}
      \label{fig:case_1.1.2}
    \end{minipage}
  \end{tabular}\vspace{-1mm}
\end{figure}

\noindent\underline{{\bf Case~1.2.}~~When $(\tilde\G, s, t) \in \cD_{\id,\, \alpha}^0$ is in Case (B).}

\medskip
If $\tilde\G = \G'$, then it is easy to see $|l(\G; s, t)| \geq 3$ by Lemma~\ref{lem:2-2labels},
since $\G$ contains no equivalent arcs (see Fig.~\ref{fig:case_1.2.1}).
Otherwise, $\tilde\G = \G' \three X$ for some $X \subseteq V(\G) \setminus \{s, t\}$.
If $N_{\G'}(X) = \{s, v_1, v_2\}$,
then $\around{\G}{X}$ is not balanced because $N_\G(X) = N_{\G'}(X)$ but $X$ is not 3-contractible in $\G$ (by Claim~\ref{cl:3-contractible}),
and hence $|l(\G; s, t)| \geq 3$ by Lemma~\ref{lem:2-2labels}
(e.g., we can take two $s$--$v_1$ paths $P_1$ and $P_2$
in $\around{\G}{X}$ with $\psi_\G(P_1) \neq \psi_\G(P_2)$), a contradiction.

Suppose that $N_{\G'}(X) = \{v_3, v_4, t\}$ (see Fig.~\ref{fig:case_1.2.2}).
If there exist two disjoint paths between $\{v_0, t\}$ and $\{v_3, v_4\}$
in $\around{\G}{X}$, then $|l(\G; s, t)| \geq 3$ by Lemma~\ref{lem:2-2labels}, a contradiction;
e.g., if a $v_0$--$v_3$ path and a $t$--$v_4$ path can be taken disjointly in $\around{\G}{X}$,
one can construct two $s$--$v_1$ paths $P_1$ and $P_2$ in $\G[X \cup \{v_1, v_3\}]$ with $\psi_\G(P_1) \neq \psi_\G(P_2)$
and $l(\G - (V(P_i) \setminus \{v_1\}); v_1, t) = \{\id, \alpha\}$ $(i = 1, 2)$.
Otherwise, by Menger's theorem,
$\around{\G}{X}$ contains a 1-cut $w \in X$
separating $\{v_0, t\}$ from $\{v_3, v_4\}$ (possibly $w = v_0$).
In this case, $\{s, w\}$ is a 2-cut in $\G$,
which contradicts Claim~\ref{cl:2-contractible}.

\begin{figure}[htbp]\vspace{1mm}\hspace{-3mm}
  \begin{tabular}{cc}
    \begin{minipage}[t]{0.5\hsize}
      \begin{center}
        \includegraphics[scale=0.7]{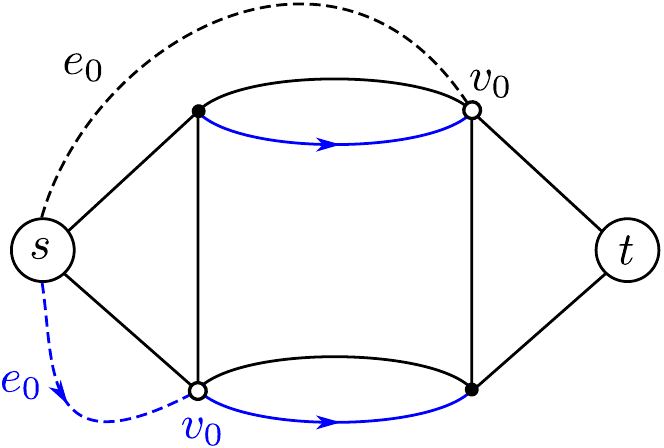}
      \end{center}\vspace{-5mm}
      \caption{\mbox{$(\G', s, t) \in \cD_{\id,\,\alpha}^0$ is in Case (B);\quad\ } \mbox{two possibilities of $e_0 = \{s, v_0\}$ are depicted.}}
      \label{fig:case_1.2.1}
    \end{minipage}
    \begin{minipage}[t]{0.5\hsize}
      \begin{center}
        \includegraphics[scale=0.7]{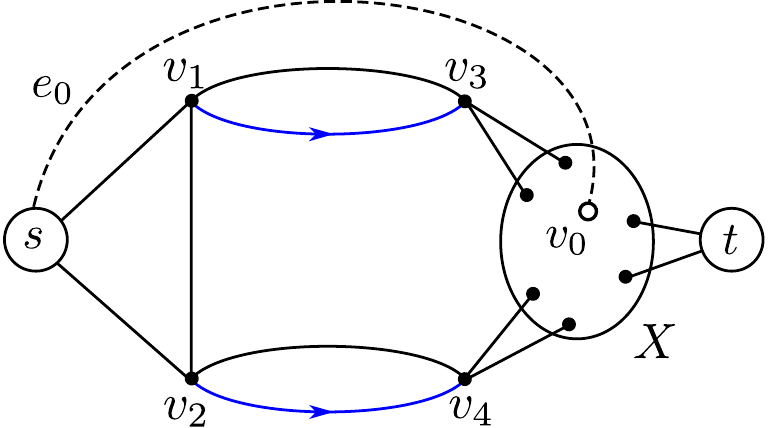}
      \end{center}\vspace{-5mm}
      \caption{$(\G' \three X, s, t) \in \cD_{\id,\,\alpha}^0$ is in Case (B).}
      \label{fig:case_1.2.2}
    \end{minipage}
  \end{tabular}
\end{figure}

\noindent\underline{{\bf Case~1.3.}~~When $(\tilde\G, s, t) \in \cD_{\id,\, \alpha}^0$ is in Case (C).}

\medskip
In order to apply Lemma~\ref{lem:cDab0}, we first confirm the degree condition.

\begin{claim}\label{cl:degree}
$|\delta_{\tilde\G}(v)| \geq 3$ for every $v \in V(\tilde\G) \setminus \{s, t\}$.
\end{claim}

\begin{proof}
Suppose to the contrary that there exists a vertex $v \in V(\tilde\G) \setminus \{s, t\}$ with $|\delta_{\tilde\G}(v)| \leq 2$.
If $\delta_{\tilde\G}(v) \subseteq \delta_{\G}(v) \setminus \{e_0\}$,
then $\{v\}$ is contractible in $\G$ (or $v$ is not contained in any $s$--$t$ path, contradicting $(\G, s, t) \in \cD$), which contradicts Claim~\ref{cl:2-contractible} or \ref{cl:3-contractible}.
Otherwise, $\tilde\G = \G' \three X$ for some $X \subseteq V(\G) \setminus \{s, t\}$ with $v \in N_{\G'}(X)$ (and then $v_0 \in X$).
In this case, $\delta_{\tilde\G}(v)$ consists of two edges in the balanced triangle on $N_{\G'}(X)$, and hence $N_\G(v) \subseteq X \cup (N_{\G'}(X) \setminus \{v\})$.
If $s \not\in N_{\G'}(X)$ (see Fig.~\ref{fig:case_1.3}), then $N_{\G}(X \cup \{v\}) = (N_{\G'}(X) \setminus \{v\}) \cup \{s\}$,
and hence $X \cup \{v\}$ is 3-contractible in $\G$, which contradicts Claim~\ref{cl:3-contractible},
where the connectivity of $\G[X \cup \{v\}]$ is guaranteed by that of $\G'[X] = \G[X]$ and $v \in N_{\G'}(X) = N_{\G}(X) \setminus \{s\}$,
and the balancedness of $\around{\G}{X \cup \{v\}}$ follows from that of $\around{\G'}{X}$ and $\delta_\G(X \cup \{v\}) \setminus \delta_{\G'}(X) \subseteq \delta_{\tilde\G}(v) \cup \{e_0\}$.
Otherwise, $N_{\G'}(X) = \{s, v, x\}$ for some $x \in V(\G) \setminus \{s, v\}$ and $N_{\G}(v) \subseteq X \cup \{s, x\}$.
If $\around{\G}{X \cup \{v\}} = \G$, then $\around{\G'}{X \cup \{v\}} = \G - e_0$ is balanced as well as $\around{\G'}{X}$
(note that $\delta_{\G'}(X \cup \{v\}) \setminus \delta_{\G'}(X) \subseteq \delta_{\tilde\G}(v)$),
and hence $(\G, s, t) \in \cDab^0$ is in Case (A1) in Definition~\ref{def:cDab0}, a contradiction.
Otherwise, $X \cup \{v\}$ is 2-contractible in $\G$, which contradicts Claim~\ref{cl:2-contractible}.
\end{proof}

\begin{figure}[htbp]\vspace{-1mm}
 \begin{center}
  \includegraphics[scale=0.7]{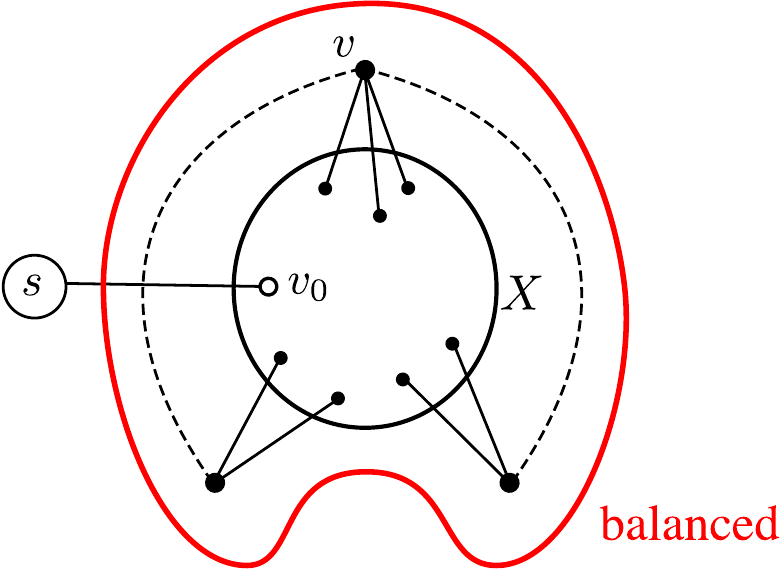}
 \end{center}\vspace{-4mm}
\caption{$X \cup \{v\}$ is 3-contractible in $\G$ when $|\delta_{\tilde\G}(v)| \leq 2$, $\tilde\G = \G' \three X$, and $s \not\in N_{\G'}(X)$.}
 \label{fig:case_1.3}
\end{figure}

We then suppose that $\tilde\G$ is embedded as in Lemma~\ref{lem:cDab0},
where we apply shifting at vertices in $V(\G) \setminus \{s, t\}$ to $\G$ in advance of constructing $\tilde\G$ if necessary.
We denote by $\tE^i \subseteq E(\tilde{\G})$ the edge set corresponding to $E^i$ in Lemma~\ref{lem:cDab0} for each $i = 0, 1$,
and refer to the path $P = (s = u_0, e_1, u_1, \ldots, e_\ell, u_\ell = t)$ along the outer boundary of $\tilde\G - \tE^1$ as $P$ itself,
i.e., $\psi_{\tilde\G}(e, v) = \id$ for each $e \in \tE^0$ and $v \in e$, and $\psi_{\tilde\G}(e, u_j) = \alpha$ for each $e = \{u_i, u_j\} \in \tE^1$ with $i < j$.
We then observe the following property.

\begin{claim}\label{cl:inner}
Some edge in $\tE^1$ connects inner vertices of $P$, i.e., $\tE^1 \setminus \delta_{\tilde\G}(s) \neq \emptyset \neq \tE^1 \setminus \delta_{\tilde\G}(t)$.
\end{claim}

\begin{proof} 
Suppose to the contrary that $\tE^1 \setminus \delta_{\tilde\G}(s) = \emptyset$ or $\tE^1 \setminus \delta_{\tilde\G}(t) = \emptyset$.
Then, $\tilde\G - s$ or $\tilde\G - t$ is balanced, respectively,
which implies that $(\tilde\G, s, t) \in \cDab^0$ is in Case (A) in Definition~\ref{def:cDab0}.
Thus we derive a contradiction by reducing to Case 1.1.
\end{proof}

In what follows,
we derive a contradiction by showing either that $(\G, s, t) \in \cD_{\id,\, \alpha}^0$,
that $\gamma \in l(\G; s, t)$ for some $\gamma \in \Gamma \setminus \{\id, \alpha\}$
(in particular, $\gamma = \alpha^2$ or $\alpha^{-1}$), or that $\G$ contains a contractible vertex set
(which contradicts Claim~\ref{cl:2-contractible} or \ref{cl:3-contractible}).
We first discuss with the assumption $\tilde\G = \G'$,
and later explain that the case when $\tilde\G = \G' \three X$ for some $X \subseteq V(\G) \setminus \{s, t\}$
can be dealt with in almost the same way with the aid of Theorem~\ref{thm:2path} (cf.~Case 1.3.3).

Assume $\tilde\G = \G' = \G - e_0$.
We then have $(\tilde\G, s, t) \in \cD$ (Claim~\ref{cl:cD}) and hence $\tilde\G - s$ is connected.
Since every edge in $\tE^1$ connects two vertices
on the path $P$ in $\tilde\G - \tE^1$, 
also $\tilde\G - \tE^1 - s$ is connected, and in particular $\tilde\G - s$ contains a $v_0$--$t$ path of label $\id$.
Hence, we may assume that $\psi_\G(e_0, v_0) \in l(\G; s, t) = \{\id, \alpha\}$
(otherwise, we immediately obtain an $s$--$t$ path of label $\gamma \in \Gamma \setminus \{\id, \alpha\}$ in $\G$),
and consider the following two cases separately:
when $\psi_\G(e_0, v_0) = \id$, and when $\psi_\G(e_0, v_0) = \alpha$.

\medskip
\noindent\underline{{\bf Case 1.3.1.}~~When $\tilde\G = \G' = \G - e_0$ and $\psi_\G(e_0, v_0) = \id$.}

\medskip
Let $\tF_0$ and $\tF'_0$ denote the outer faces of $\tilde\G$ and of $\tilde\G - s = \G - s$, respectively.
Let us begin with the case when $v_0 \in V({\rm bd}(\tF'_0))$, which is rather easy.
\begin{description}
  \setlength{\itemsep}{1mm}
\item[Case 1.3.1.1.]
  Suppose that $v_0 \in V({\rm bd}(\tF'_0)) \setminus V(P)$.
  In this case, we can embed $\G = \tilde\G + e_0$ in the plane by adding $e_0 = \{s, v_0\}$ on the outer face $\tF_0$
  so that $(\G, s, t)$ satisfies the conditions of Case (C), a contradiction.

\item[Case 1.3.1.2 {\rm (Fig.~\ref{fig:case_1.3.1.2})}.]
  Otherwise, $v_0 = u_h \in V({\rm bd}(\tF'_0)) \cap V(P)$.
  Take an $s$--$t$ path $P'$ so that the closed walk obtained by removing $s$ from the concatenation of $P$ and $\overline{P'}$
  is the outer boundary of $\tilde\G - \tE^1 - s$.
  Let $j$ be the minimum index such that $E(P[u_j, t]) \subseteq E(P')$,
  and $i < j$ the index such that the concatenation of $P[u_i, u_j]$ and $\overline{P'}[u_j, u_i]$ is a cycle
  (i.e., they intersect only at $u_i$ and $u_j$).

  Take an edge $e' = \{u_{i'}, u_{j'}\} \in \tE^1 \setminus \delta_{\tilde\G}(s)$ so that $j' - i' > 0$ is maximized.
  If $j' \leq i$, then $\G$ contains a 2-cut $\{s, u_i\}$ separating
  $u_{i-1} \neq s$ from $t \neq u_i$, which contradicts Claim~\ref{cl:2-contractible}.
  Hence, $i < j'$.

  If $v_0 = u_h \in V(P) \cap V(P')$ or $h \leq i'$ ((a) in Fig.~\ref{fig:case_1.3.1.2}),
  then we can embed $e_0 = \{s, v_0\}$ without violating the conditions of Case (C).
  Otherwise, we have $j' \leq h < j$ ((b) in Fig.~\ref{fig:case_1.3.1.2})
  since $u_h = v_0 \in V({\rm bd}(\tF'_0)) \cap V(P)$.
  In this case, we can construct an $s$--$t$ path of label
  $\alpha^{-1} \in \Gamma \setminus \{\id, \alpha\}$ in $\G$,
  a contradiction, e.g., by concatenating
  $e_0$, $\overline{P}[u_h, u_{j'}]$, $e' = \{u_{i'}, u_{j'}\}$, $P[u_{i'}, u_i]$, $P'[u_i, u_j]$, and $P[u_j, t]$ if $0 < i' \leq i$.
\end{description}

\begin{figure}[htbp]
 \begin{center}
  \includegraphics[scale=0.9]{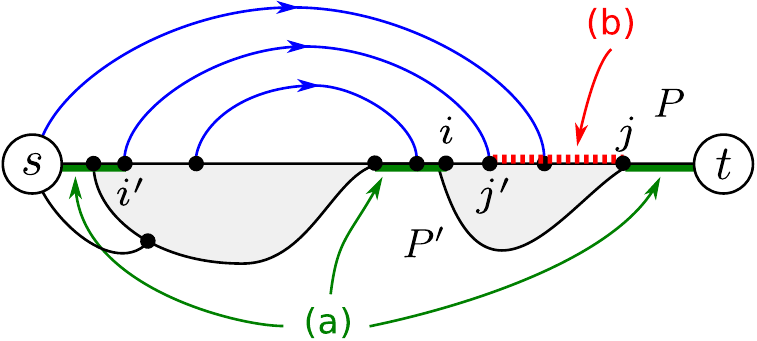}
 \end{center}\vspace{-5mm}
\caption{Case~1.3.1.2 ((a) embeddable, (b) label $\alpha^{-1}$).}\vspace{-2mm}
 \label{fig:case_1.3.1.2}
\end{figure}

We next consider the case when $v_0 \not\in V({\rm bd}(\tF'_0))$.
Take the $s$--$t$ path $P'$ as with Case 1.3.1.2 (i.e., so that $P$ and $P'$ form the outer boundary of $\tilde{\G} - \tE^1 - s$).
Since $v_0 \not\in V({\rm bd}(\tF'_0))$, some $Q \coloneqq P'[u_i, u_j]$ and $P[u_i, u_j]$ form a cycle $C$ that encloses $v_0 \not\in V(Q)$ (possibly $v_0 \in V(P)$),
i.e., $V(C)$ separates $v_0$ from both $s$ and $t$ in $\tilde{\G}$ (and if $v_0 = u_h \in V(P)$ then $i < h < j$).

Suppose that $V(Q)$ separates $v_0$ from $V(P)$ in $\tilde\G = \G - e_0$.
Let $w_1, w_2 \in V(Q)$ be the vertices closest to $u_i, u_j \in V(P) \cap V(Q)$, respectively,
among those which are reachable from $v_0$ in $\tilde\G$ without intersecting $Q$ in between (see Fig.~\ref{fig:case_1.3.1.2-3}).
Then, $\{s, w_1, w_2\}$ is a 3-cut (or a 2-cut if $w_1 = w_2$) in $\G$,
and there exists a 3-contractible (or 2-contractible, respectively) vertex set $X \subseteq V(\G) \setminus V(P)$
such that $v_0 \in X$ and $N_\G(X) = \{s, w_1, w_2\}$, which contradicts Claim~\ref{cl:3-contractible} (or \ref{cl:2-contractible}, respectively).
Thus we can take a $v_0$--$u_h$ path $R$ in $\tilde\G - V(Q)$
(possibly of length $0$, i.e., $v_0 = u_h$)
with $i < h < j$. If there are multiple choices of $R$,
then choose $R$ so that $h$ is maximized under the condition that $V(R) \cap V(P) = \{u_h\}$.

\begin{figure}[htbp]
 \begin{center}\vspace{1mm}
  \includegraphics[scale=0.9]{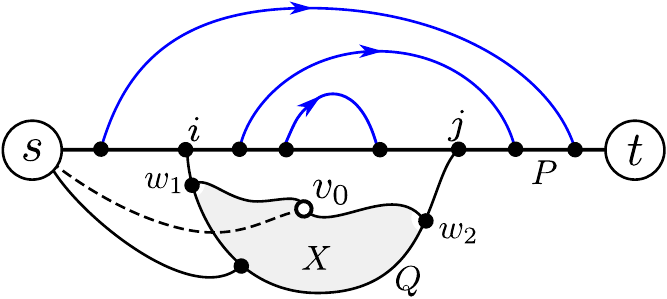}
 \end{center}\vspace{-5mm}
\caption{$X \subseteq V(\G) \setminus V(P)$ with $N_\G(X) = \{s, w_1, w_2\}$ is 3-contractible in $\G$.}
 \label{fig:case_1.3.1.2-3}
\end{figure}

We now focus on the edges in $\tE^1 \setminus \delta_{\tilde\G}(s) \neq \emptyset$ (cf.~Claim~\ref{cl:inner}).
\begin{description}
  \setlength{\itemsep}{1mm}
\item[Case 1.3.1.3 {\rm (Figs.~\ref{fig:case_1.3.1.3a} and \ref{fig:case_1.3.1.3b})}.]
  Suppose that no edge in $\tE^1 \setminus \delta_{\tilde\G}(s)$
  is incident to an inner vertex of $P[u_i, u_j]$.
  If every edge in $\tE^1 \cap \delta_{\tilde\G}(s)$ is incident to
  a vertex on $P[u_j, t]$ (see Fig.~\ref{fig:case_1.3.1.3a}),
  then $\G$ contains a 3-contractible vertex set
  $X \subseteq V(\G) \setminus \{s, u_i, u_j\}$
  such that $v_0 \in X$ and $N_{\G}(X) = \{s, u_i, u_j\}$, a contradiction. 
  Otherwise, there exists an edge $\{s, u_{j'}\} \in \tE^1 \cap \delta_{\tilde\G}(s)$ with $j' < j$ (see Fig.~\ref{fig:case_1.3.1.3b}).
  Then, no edge in $\tE^1 \setminus \delta_{\tilde\G}(s)$ is incident to a vertex on $P[u_j, t]$ (by the condition 2 in Lemma~\ref{lem:cDab0}),
  and hence every edge in $\tE^1 \setminus \delta_{\tilde\G}(s)$
  connects two vertices on $P[u_1, u_i]$.
  Since $\tE^1 \setminus \delta_{\tilde\G}(s) \neq \emptyset$,
  we have $i \geq 2$, and $\G$ contains a 2-cut $\{s, u_i\}$ separating $u_{i-1}$ from $t$,
  which contradicts Claim~\ref{cl:2-contractible}.

\item[Case 1.3.1.4 {\rm (Fig.~\ref{fig:case_1.3.1.4})}.]
  Suppose that there exists an edge $e' = \{u_{i'}, u_{j'}\} \in \tE^1 \setminus \delta_{\tilde\G}(s)$ with $i' < j'$
  such that $i' < h$ and $i < j' < j$.
  In this case, we can construct an $s$--$t$ path of label $\alpha^{-1} \in \Gamma \setminus \{\id, \alpha\}$
  in $\G$, a contradiction, e.g., by concatenating
  $e_0$, $R$, $P[u_h, u_{j'}]$, $e'$, $\overline{P}[u_{i'}, u_i]$, $Q$, and $P[u_j, t]$
  if $i \leq i'$ and $h \leq j'$.

\item[Case 1.3.1.5 {\rm (Fig.~\ref{fig:case_1.3.1.5})}.]
  Suppose that every edge in $\tE^1 \setminus \delta_{\tilde\G}(s)$
  connects two vertices on $P[u_h, t]$.
  In this case, every edge in $\tE^1 \cap \delta_{\tilde\G}(s)$
  is also incident to a vertex on $P[u_h, t]$,
  and $v_0 \neq u_h$ since $v_0 \not\in V({\rm bd}(\tF'_0))$.
  Let $w \in V(Q)$ be the vertex closest to $u_j$ among those which are reachable from $v_0$ in $\G - u_h$ without intersecting $Q$ in between
  (if there is no such vertex $w$, then $\{s, u_h\}$ is a 2-cut separating $v_0$ from $u_j$ in $\G$, which contradicts Claim~\ref{cl:2-contractible}).
  By the maximality of $j$ and $h$ (i.e., the choice of $Q$ and $R$),
  $\{s, u_h, w\}$ separates $v_0 \in V(\G) \setminus \{s, u_h, w\}$
  from $V(P[u_h, t])$ in $\G$,
  and hence $\G$ contains a 3-contractible vertex set
  $X \subseteq V(\G) \setminus \{s, u_h, w\}$
  such that $v_0 \in X$ and $N_{\G}(X) = \{s, u_h, w\}$,
  a contradiction.
\end{description}

\begin{figure}[htbp]\hspace{-3mm}
  \begin{tabular}{cc}
    \begin{minipage}[b]{0.5\hsize}
      \begin{center}
        \includegraphics[scale=0.9]{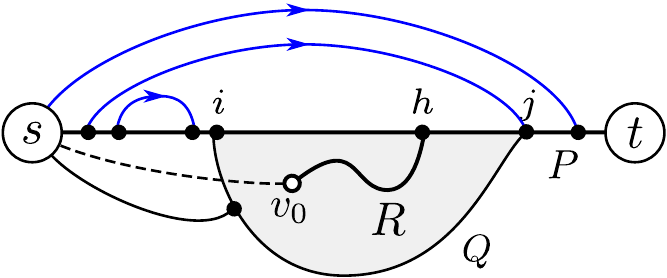}
      \end{center}\vspace{-5mm}
      \caption{Case~1.3.1.3 (a 3-cut $\{s, u_i, u_j\}$).}
      \label{fig:case_1.3.1.3a}
    \end{minipage}
    \begin{minipage}[b]{0.5\hsize}
      \begin{center}
        \includegraphics[scale=0.9]{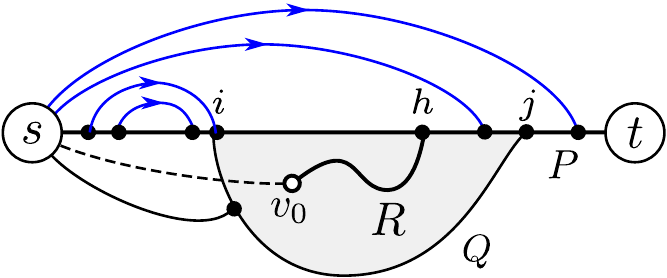}
      \end{center}\vspace{-5mm}
      \caption{Case~1.3.1.3 (a 2-cut $\{s, u_i\}$).}
      \label{fig:case_1.3.1.3b}
    \end{minipage}\\[5mm]
    \begin{minipage}[b]{0.5\hsize}
      \begin{center}
        \includegraphics[scale=0.9]{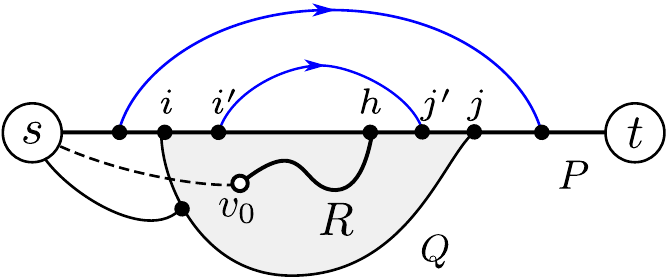}
      \end{center}\vspace{-5mm}
      \caption{Case~1.3.1.4 (label $\alpha^{-1}$).}
      \label{fig:case_1.3.1.4}
    \end{minipage}
    \begin{minipage}[b]{0.5\hsize}
      \begin{center}
        \includegraphics[scale=0.9]{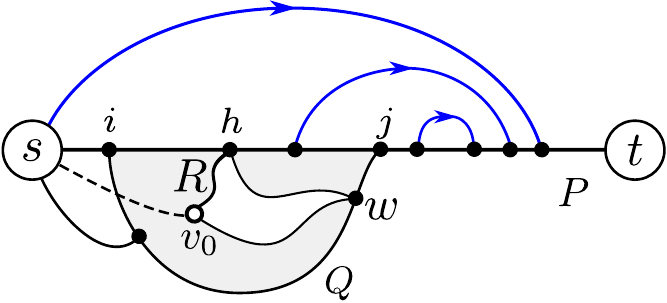}
      \end{center}\vspace{-5mm}
      \caption{Case~1.3.1.5 (a 3-cut $\{s, u_h, w\}$).}
      \label{fig:case_1.3.1.5}
    \end{minipage}
  \end{tabular}\vspace{0mm}
\end{figure}

Take an edge $e' = \{u_{i'}, u_{j'}\} \in \tE^1 \setminus \delta_{\tilde\G}(s) \neq \emptyset$ so that $j' - i' > 0$ is maximized.
Then, we may assume $i < j'$ by Case~1.3.1.3, and $h \leq i'$ or $j \leq j'$ by Case~1.3.1.4,
The former case is, however, forbidden by Case~1.3.1.5, and hence we conclude $i' < h < j \leq j'$.
This implies also that no edge in $\tE^1 \cap \delta_{\tilde\G}(s)$ is incident to an inner vertex of $P[s, u_j]$.
\begin{description}
  \setlength{\itemsep}{1mm}
\item[Case 1.3.1.6 {\rm (Fig.~\ref{fig:case_1.3.1.6})}.]
  Suppose that there exists $i^\ast < h$ such that every edge in $\tE^1 \setminus \delta_{\tilde\G}(s)$
  connects $u_{i^\ast}$ and some $u_{j'}$ with $i^\ast < j'$.
  In this case, by Case~1.3.1.4,
  we may assume every edge in $\tE^1 \setminus \delta_{\tilde\G}(s)$
  is incident to a vertex on $P[u_j, t]$.
  Then, since $\{s, u_{i^\ast}, u_j\}$ separates $v_0 \in V(\G) \setminus \{s, u_{i^\ast}, u_j\}$ from $V(P[u_j, t])$ in $\G$,
  there exists a 3-contractible
  vertex set $X \subseteq V(\G) \setminus \{s, u_{i^\ast}, u_j\}$ in $\G$
  such that $v_0 \in X$ and $N_{\G}(X) = \{s, u_{i^\ast}, u_j\}$,
  a contradiction.
\item[Case 1.3.1.7 {\rm (Fig.~\ref{fig:case_1.3.1.7})}.]
  Suppose that there exists $j^\ast \geq j$ such that every edge in $\tE^1 \setminus \delta_{\tilde\G}(s)$
  connects some $u_{i'}$ with $i' < j^\ast$ and $u_{j^\ast}$.
  In this case, $\{s, u_j, u_{j^\ast}\}$ (possibly $u_j = u_{j^\ast}$) separates $v_0 \in V(\G) \setminus \{s, u_j, u_{j^\ast}\}$ from $V(P[u_j, t])$ in $\G$,
  and hence $\G$ contains a (2- or 3-)contractible vertex set $X \subseteq V(\G) \setminus \{s, u_j, u_{j^\ast}\}$
  such that $v_0 \in X$ and $N_{\G}(X) = \{s, u_j, u_{j^\ast}\}$, a contradiction.
  Note that $u_{j^\ast} \neq t$ by $\tE^1 \setminus \delta_{\tilde\G}(t) \neq \emptyset$ (Claim~\ref{cl:inner}).
\end{description}

\begin{figure}[htbp]\hspace{-3mm}
  \begin{tabular}{cc}
    \begin{minipage}[b]{0.5\hsize}
      \begin{center}
        \includegraphics[scale=0.9]{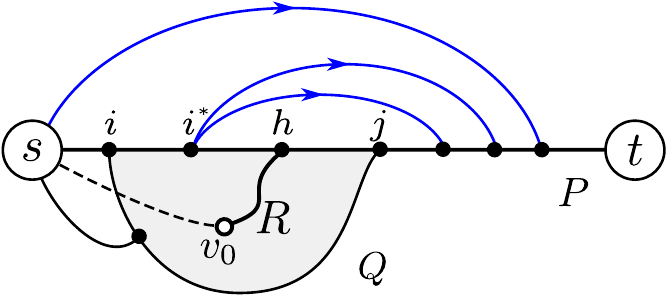}
      \end{center}\vspace{-5mm}
      \caption{Case~1.3.1.6 (a 3-cut $\{s, u_{i^\ast}, u_j\}$).}
      \label{fig:case_1.3.1.6}
    \end{minipage}
    \begin{minipage}[b]{0.5\hsize}
      \begin{center}
        \includegraphics[scale=0.9]{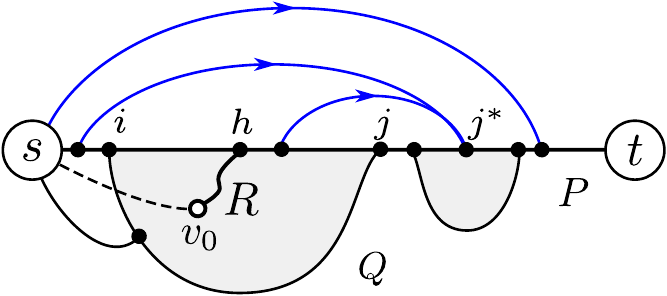}
      \end{center}\vspace{-5mm}
      \caption{Case~1.3.1.7 (a 3-cut $\{s, u_j, u_{j^\ast}\}$).}
      \label{fig:case_1.3.1.7}
    \end{minipage}
  \end{tabular}
\end{figure}

By Cases~1.3.1.6 and 1.3.1.7, we may assume that there exist two edges
$e_1 = \{u_{i_1}, u_{j_1}\}$ and $e_2 = \{u_{i_2}, u_{j_2}\}$ in $\tE^1 \setminus \delta_{\tilde\G}(s)$
such that $i_2 < i_1 < j_1 < j_2$.
We choose $e_2$ so that $j_2 - i_2$ is maximized.
We then have $i_2 < h$ and $j \leq j_2$ by the argument just after Case~1.3.1.5 (for $i_2, j_2$ instead of $i', j'$). 
Since there exists an edge in $\tE^1 \setminus \delta_{\tilde\G}(s)$
incident to an inner vertex of $P[u_i, u_j]$ by Case~1.3.1.3,
we can choose $e_1$ so that $i < i_1$
(which is obvious if $i \leq i_2$, and follows from Case~1.3.1.4 otherwise).
We then have $h < j_1$,
since otherwise we have $i < i_1 < j_1 \leq h < j$,
which implies that $e_1$ satisfies the condition of Case~1.3.1.4.
We choose $e_1$ so that $i_1$ is minimized under the condition that $i < i_1$.
To sum up, we have $i_2 < h <  j_1$, $j \leq j_2$, and $i < i_1$.
\begin{description}
  \setlength{\itemsep}{1mm}
\item[Case 1.3.1.8 {\rm (Fig.~\ref{fig:case_1.3.1.8})}.]
  Suppose that $j \leq i_1$.
  Then, $\{s, u_{i_2}, u_j\}$ separates
  $v_0 \in V(\G) \setminus \{s, u_{i_2}, u_j\}$ from $P[u_j, t]$ in $\G$,
  and hence $\G$ contains a 3-contractible vertex set
  $X \subseteq V(\G) \setminus \{s, u_{i_2}, u_j\}$
  such that $v_0 \in X$ and $N_{\G}(X) = \{s, u_{i_2}, u_j\}$,
  a contradiction.

\item[Case 1.3.1.9 {\rm (Fig.~\ref{fig:case_1.3.1.10})}.]
  Suppose that $j_2 > j$. We then have $i < i_1 < j < j_2$ (also recall that $i_2 < i_1 < j_1 < j_2$ and $i_2 < h < j_1$).
  In this case, we can construct an $s$--$t$ path of label $\alpha^2 \in \Gamma \setminus \{\id, \alpha\}$
  in $\G$, a contradiction, e.g., by concatenating
  $e_0$, $R$, $\overline{P}[u_h, u_{i_1}]$, $e_1$, $\overline{P}[u_{j_1}, u_j]$, $\overline{Q}$, $P[u_i, u_{i_2}]$, $e_2$, and $P[u_{j_2}, t]$
  if $i_1 \leq h$, $j \leq j_1$, and $i \leq i_2$.

\item[Case 1.3.1.10 {\rm (Figs.~\ref{fig:case_1.3.1.9_a} and \ref{fig:case_1.3.1.9_b})}.]
  Otherwise, we have $h \leq i_1$ by $i < i_1 < j_1 < j_2 = j$ and Case~1.3.1.4.
  Let $h^\ast$ be the maximum index such that there exists a $w$--$u_{h^\ast}$
  path $R^\ast$ in $\tilde\G - u_j$ for some $w \in (V(Q) \setminus V(P)) \cup \{v_0\}$
  such that $V(R^\ast) \cap V(Q) \subseteq \{w\}$
  and $V(R^\ast) \cap V(P) = \{u_{h^\ast}\}$.
  Note that $h \leq h^\ast$.

  If $i_1 < h^\ast$, then we have $h < h^\ast$ (recall $h \leq i_1$).
  In this case (see Fig.~\ref{fig:case_1.3.1.9_a}), since $R$ and $R^\ast$ are disjoint
  by the maximality of $h$ and $h^\ast$,
  we can construct an $s$--$t$ path of label
  $\alpha^2 \in \Gamma \setminus \{\id, \alpha\}$
  in $\G$, a contradiction, e.g.,
  by concatenating $e_0$, $R$, $P[u_h, u_{i_1}]$, $e_1$,
  $\overline{P}[u_{j_1}, u_{h^\ast}]$, $\overline{R}^\ast$, $\overline{Q}[w, u_i]$, $\overline{P}[u_i, u_{i_2}]$, $e_2$,
  and $P[u_j, t]$ if $h^\ast \leq j_1$ and $i_2 \leq i$.
  Otherwise (i.e., if $h^\ast \leq i_1$), by the minimality of $i_1$
  and the maximality of $h^\ast$,
  there exists a 2-cut $\{u_{h^\ast}, u_j\}$ separating $u_{j_1}$ from $u_i$
  $(i < h \leq h^\ast \leq i_1 < j_1 < j_2 = j)$ in $\G$ (see Fig.~\ref{fig:case_1.3.1.9_b}), a contradiction.
\end{description}

\begin{figure}[htbp]\vspace{-.5mm}\hspace{-3mm}
  \begin{tabular}{cc}
    \begin{minipage}[b]{0.5\hsize}
      \begin{center}
        \includegraphics[scale=0.9]{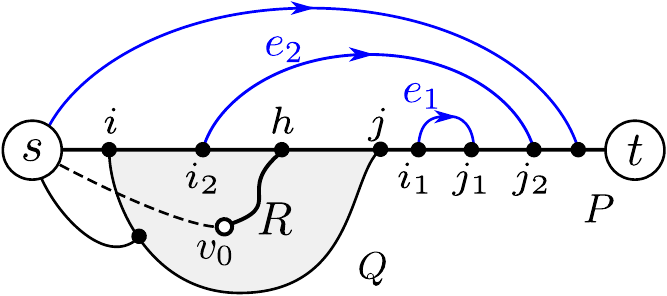}
      \end{center}\vspace{-5mm}
      \caption{Case~1.3.1.8 (a 3-cut $\{s, u_{i_2}, u_j\}$).}
      \label{fig:case_1.3.1.8}
    \end{minipage}
    \begin{minipage}[b]{0.5\hsize}
      \begin{center}
        \includegraphics[scale=0.9]{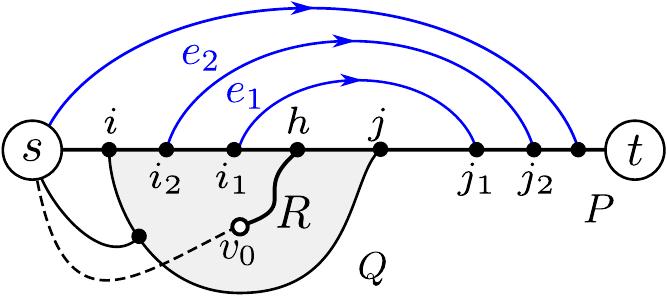}
      \end{center}\vspace{-5mm}
      \caption{Case~1.3.1.9 (label $\alpha^2$).}
      \label{fig:case_1.3.1.10}
    \end{minipage}\\[5mm]
    \begin{minipage}[b]{0.5\hsize}
      \begin{center}
        \includegraphics[scale=0.8]{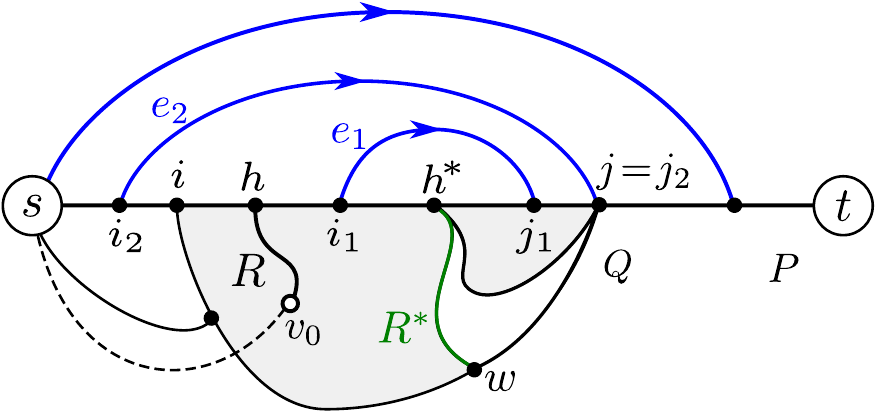}
      \end{center}\vspace{-5mm}
      \caption{Case~1.3.1.10 (label $\alpha^2$).}
      \label{fig:case_1.3.1.9_a}
    \end{minipage}
    \begin{minipage}[b]{0.5\hsize}
      \begin{center}
        \includegraphics[scale=0.8]{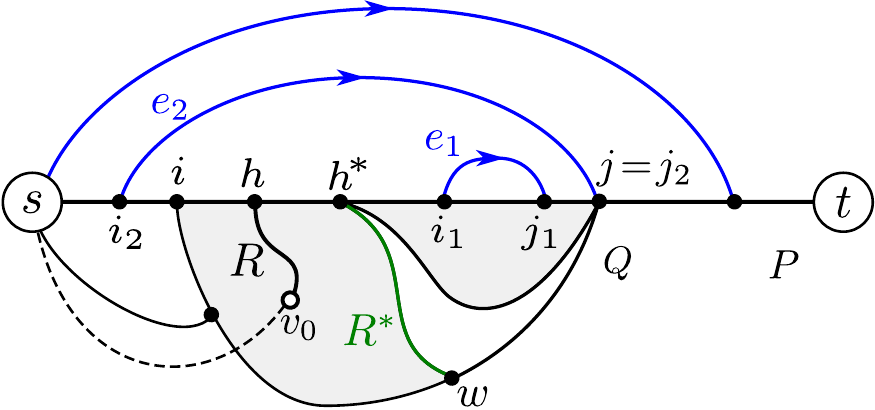}
      \end{center}\vspace{-5mm}
      \caption{Case~1.3.1.10 (a 2-cut $\{u_{h^\ast}, u_j\}$).}
      \label{fig:case_1.3.1.9_b}
    \end{minipage}
  \end{tabular}\vspace{-.5mm}
\end{figure}

\noindent\underline{{\bf Case~1.3.2.}~~When $\tilde\G = \G' = \G - e_0$ and $\psi_\G(e_0, v_0) = \alpha$.}

\medskip
This case is rather easier than Case~1.3.1.
Note that, if there exists a $v_0$--$t$ path of label $\alpha$ in $\tilde\G = \G' = \G - e_0$,
then we can construct an $s$--$t$ path of label $\alpha^2 \in \Gamma \setminus \{\id, \alpha\}$ in $\G$
by extending the $v_0$--$t$ path using $e_0 = \{s, v_0\}$.
%
\begin{description}
  \setlength{\itemsep}{1mm}
\item[Case~1.3.2.1 {\rm (Fig.~\ref{fig:case_1.3.2.1})}.]
  Suppose that $v_0 = u_h \in V(P)$.
  If there exists an edge $e' = \{u_{i'}, u_{j'}\} \in \tE^1 \setminus \delta_{\tilde\G}(s)$ with $i' < j'$
  and $h < j'$ ((a) in Fig.~\ref{fig:case_1.3.2.1}),
  then we can construct an $s$--$t$ path of label $\alpha^2$ in $\G$,
  a contradiction, e.g.,
  by concatenating $e_0$, $P[u_h, u_{i'}]$, $e'$, and $P[u_{j'}, t]$ if $h \leq i'$.
  Otherwise, every edge in $\tE^1 \setminus \delta_{\tilde\G}(s) \neq \emptyset$
  connects two vertices on $P[u_1, u_h]$ ((b) in Fig.~\ref{fig:case_1.3.2.1}).
  Hence, we can embed $\G = \tilde\G + e_0$ in the plane by adding $e_0 = \{s, u_h\}$ on the outer face $\tF_0$ without violating the conditions
  of Case (C) in Definition~\ref{def:cDab0} (cf.~Lemma~\ref{lem:cDab0}), a contradiction.
  
\begin{figure}[htbp]
 \begin{center}
  \includegraphics[scale=0.9]{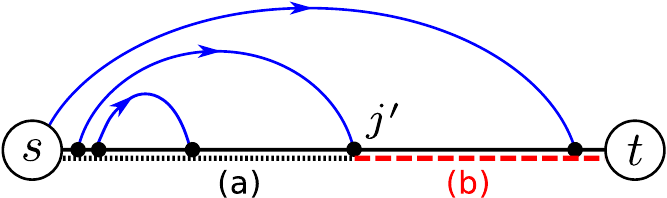}
 \end{center}\vspace{-5mm}
 \caption{Case~1.3.2.1 ((a) label $\alpha^2$, (b) embeddable).}\vspace{-2mm}
 \label{fig:case_1.3.2.1}
\end{figure}

\item[Case~1.3.2.2 {\rm (Figs.~\ref{fig:case_1.3.2.2a} and \ref{fig:case_1.3.2.2b})}.]
  Otherwise, $v_0 \not\in V(P)$.
  Let $i$ and $j$ be the minimum and maximum indices, respectively,
  such that there exist a $v_0$--$u_i$ path $Q$ and a $v_0$--$u_j$ path $R$
  in $\tilde\G - \tE^1 - s$ that do not intersect $P$ in between.
  If there exists an edge $e' = \{u_{i'}, u_{j'}\} \in \tE^1 \setminus \delta_{\tilde\G}(s)$
  with $i' < j'$ and $i < j'$ (see Fig.~\ref{fig:case_1.3.2.2a}),
  then we can construct an $s$--$t$ path of label $\alpha^2$ in $\G$,
  a contradiction, e.g.,
  by concatenating $e_0$, $Q$, $\overline{P}[u_i, u_{i'}]$, $e'$, and $P[u_{j'}, t]$ if $i' \leq i$.

  Otherwise, every edge in $\tE^1 \setminus \delta_{\tilde\G}(s) \neq \emptyset$
  connects two vertices on $P[u_1, u_i]$ (see Fig.~\ref{fig:case_1.3.2.2b}).
  Since $\G$ contains no 3-contractible vertex set (by Claim~\ref{cl:3-contractible}),
  there exists an edge in $\tE^0 \cap \delta_{\tilde\G}(s)$ incident to some vertex in the connected component of $\tilde\G - \{s, u_i, u_j\}$
  that contains $v_0$.
  Hence, because of the minimality of $i$ (and since $\tilde\G$ is embedded in the plane),
  there is no path from an inner vertex of $P[s, u_i]$ to a vertex on $P[u_j, t]$
  in $\tilde\G - \tE^1 - s$ that does not intersect $P$ in between.
  This implies that $\G$ contains a 2-cut $\{s, u_i\}$
  separating $u_1 \neq u_i$ from $t$, a contradiction.
\end{description}

\begin{figure}[htbp]\hspace{-3mm}
  \begin{tabular}{cc}
    \begin{minipage}[b]{0.5\hsize}
      \begin{center}
        \includegraphics[scale=0.9]{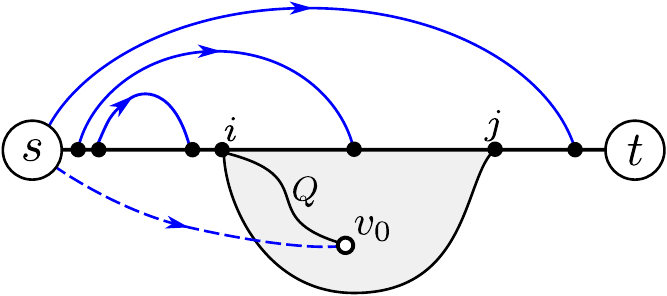}
      \end{center}\vspace{-5mm}
      \caption{Case~1.3.2.2 (label $\alpha^2$).}
      \label{fig:case_1.3.2.2a}
    \end{minipage}
    \begin{minipage}[b]{0.5\hsize}
      \begin{center}
        \includegraphics[scale=0.9]{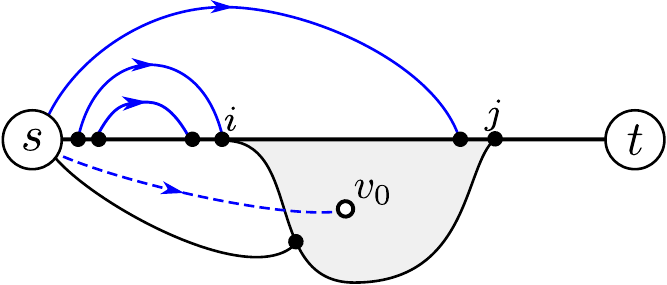}
      \end{center}\vspace{-5mm}
      \caption{Case~1.3.2.2 (a 2-cut $\{s, u_i\}$).}
      \label{fig:case_1.3.2.2b}
    \end{minipage}
  \end{tabular}
\end{figure}

\noindent\underline{{\bf Case~1.3.3.}~~When $\tilde\G = \G' \three X$ for some $X \subseteq V(\G) \setminus \{s, t\}$.}

\medskip
Recall that $X$ must contain $v_0$ by Claim~\ref{cl:3-contractible}.
Suppose that $N_{\G'}(X) = \{y_1, y_2, y_3\}$.
Since the triangle on $N_{\G'}(X)$ yielded by the 3-contraction of $X$ in $\G'$ is a balanced cycle (cf.~Definition~\ref{def:3-contraction}),
it consists of either three edges in $\tE^0$ or one edge in $\tE^0$ and two edges in $\tE^1$
(the case of exactly two edges in $\tE^0$ cannot occur due to the balancedness,
and the case of three edges in $\tE^1$ is forbidden by Lemma~\ref{lem:cDab0}).
Without loss of generality (by the symmetry of $y_1, y_2, y_3$),
assume that one of the three edges in the balanced triangle is $\{y_2, y_3\} \in \tE^0$,
i.e., $l(\around{\G'}{X}; y_2, y_3) = \id$.
Then, by applying Lemma~\ref{lem:shifting} to $\around{\G'}{X}$ with $y_2$ and $y_3$,
we observe that $\around{\G'}{X}$ is $(y_2, y_3)$-equivalent to the graph with all arcs labeled $\id$.
By applying to $\G$ the same shiftings at the vertices in $X = V(\around{\G'}{X}) \setminus \{y_1, y_2, y_3\}$
(in advance of removing $e_0 = \{s, v_0\}$ from $\G$),
we may assume that the label of every arc in $\around{\G'}{X} - y_1$ is $\id$ and every arc around $y_1$ is $\gamma$,
where $\gamma$ is a fixed element in $\{\id, \alpha, \alpha^{-1}\}$
and all arcs around $y_1$ are assumed to enter $y_1$.

Let $\tilde\G'$ be the $\Gamma$-labeled graph
obtained from $\G'$ by the following procedure:\vspace{-1mm}
\begin{itemize}
  \setlength{\itemsep}{0mm}
\item
  merge all vertices in $X$ into $v_0$ with removing the resulting loops and merging the resulting equivalent arcs $y_iv_0$ into as a single arc for each $i = 1, 2, 3$, and
\item
  for each $\{i, j, k\} = \{1, 2, 3\}$,
  add an arc from $y_j$ to $y_k$ with label $l(\around{\G'}{X}; y_j, y_k)$
  if an equivalent arc not yet exists and
  there are two disjoint paths, a $v_0$--$y_i$ path and a $y_j$--$y_k$ path, in $\around{\G'}{X}$
  (note that otherwise, by Theorem~\ref{thm:2path},
  $\around{\G'}{X}$ can be embedded in the plane
  so that $v_0, y_j, y_i, y_k$ are on the outer boundary in this order after some contractions).
\end{itemize}

\begin{figure}[htbp]
 \begin{center}
  \includegraphics[scale=0.7]{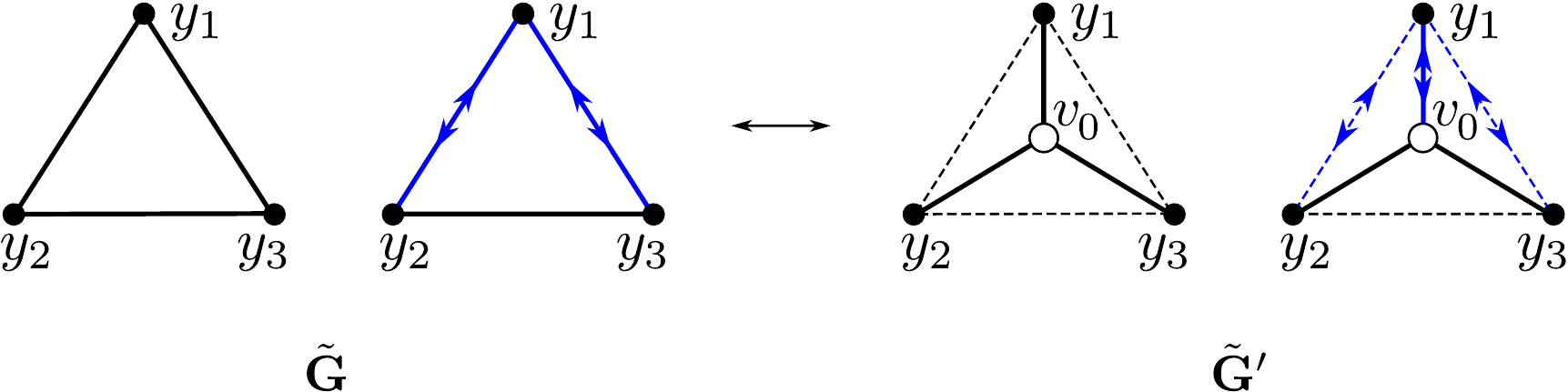}
 \end{center}\vspace{-4mm}
 \caption{Corresponding parts of $\tilde\G$ and $\tilde\G'$ in Case~1.3.3.}
 \label{fig:case_1.3.3}
\end{figure}

Since $\tilde\G$ is embedded as Lemma~\ref{lem:cDab0},
we can naturally embed $\tilde\G'$ (see Fig.~\ref{fig:case_1.3.3}).
Then, by the same case analysis for $\tilde\G'$ instead of $\tilde\G$,
we derive a contradiction of the following four types:
\begin{itemize}
  \setlength{\itemsep}{.5mm}
\item
  $\tilde\G' + e_0$ contains a 2-cut separating some vertex from $\{s, t\}$,
  which is also such a 2-cut in $\G = \G' + e_0$,
\item
  $\tilde\G' + e_0$ contains a 3-contractible vertex set $Y \subseteq V(\tilde\G') \setminus \{s, t\}$ with $v_0 \in Y$,
  which implies that $X \cup Y$ is 3-contractible in $\G = \G' + e_0$,
\item
  $\tilde\G' + e_0$ has an $s$--$t$ path of label $\alpha^{-1}$ or $\alpha^2$,
  which can be expanded to an $s$--$t$ path of the same label in $\G = \G' + e_0$ 
  (possibly by using two disjoint paths, a $v_0$--$y_i$ path and a $y_j$--$y_k$ path, in $\around{\G}{X}$), and
\item
  $e_0$ can be added to the embedding of $\tilde\G'$ without violating the conditions of Case (C),
  which implies that $\G = \G' + e_0$ can be embedded so.
\end{itemize}
In the first three cases, it is almost trivial to derive a contradiction for $\G = \G' + e_0$
from each one for $\tilde\G' + e_0$,
and we only remark the last case in what follows.
Fix an embedding of $\tilde\G'$ with the conditions in Lemma~\ref{lem:cDab0}
to which $e_0$ can be added without violating the conditions.

Suppose that $s = y_i$ for some $i \in \{1, 2, 3\}$.
We then add $e_0 = \{s, v_0\} = \{y_i, v_0\}$ to $\tilde\G'$ in the interior of the triangle on $\{y_1, y_2, y_3\}$.
If $\psi_\G(e_0, v_0) = \psi_{\tilde\G'}(e_i, v_0)$ for the merged edge $e_i = \{y_i, v_0\} \in E(\tilde\G')$, then $\around{\G}{X}$ is balanced
and hence $X$ is 3-contractible also in $\G = \G' + e_0$, which contradicts Claim~\ref{cl:3-contractible}.
Otherwise, in the obtained embedding of $\tilde\G' + e_0$,
the parallel arcs $e_0$ and $e_i$ with different labels from $s = y_i$ to $v_0$ form an unbalanced inner face.
This, however, cannot occur because it violates the conditions of Case (C),
because the embedding of $\tilde\G'$ as well as $\tilde\G$ already has one unbalanced inner face,
whose boundary does not contain $s$ by Claim~\ref{cl:inner}.

Otherwise (i.e., if $s \not\in \{y_1, y_2, y_3\} = N_{\G'}(X)$), since we can add $e_0 = \{s, v_0\}$ to the embedding of $\tilde\G'$,
the vertex $v_0$ must be on the outer boundary of $\tilde\G' - s$.
Suppose that $\tilde\G'$ has an edge between $y_i$ and $y_j$ for every distinct $i, j \in \{1, 2, 3\}$.
Then, some $y_i$ is in interior of the cycle on $\{v_0, y_j, y_k\}$, where $\{i, j, k\} = \{1, 2, 3\}$.
This implies $N_\G(X \cup \{y_i\}) = \{s, y_j, y_k\}$.
Since $X \cup \{y_i\}$ is not 3-contractible in $\G$ (by Claim~\ref{cl:3-contractible}),
there are parallel edges between $y_i$ and either $y_j$ or $y_k$,
which form an unbalanced inner face of $\tilde\G'$.
This however yields another unbalanced inner face (a cycle on $\{v_0, y_i, y_j\}$, $\{v_0, y_i, y_k\}$, or $\{y_i, y_j, y_k\}$),
which contradicts the condition of Case (C) (as well as Lemma~\ref{lem:cDab0}).

Thus we may assume that $\tilde\G'$ has no edge between $y_i$ and $y_j$ for some distinct $i, j \in \{1, 2, 3\}$.
Then, by definition, $\around{\G'}{X}$ does not have two disjoint paths,
a $v_0$--$y_k$ path and a $y_i$--$y_j$ path, where $\{i, j, k\} = \{1, 2, 3\}$.
Since $\around{\G'}{X}$ has no contractible vertex set that does not contain $v_0$
(otherwise, it contradicts Claim~\ref{cl:2-contractible} or \ref{cl:3-contractible}),
by Theorem~\ref{thm:2path}, one can embed $\around{\G'}{X}$ itself in the plane so that
$v_0, y_i, y_k, y_j$ are on the outer boundary in this order.
Using this embedding,
we can expand the embedding of $\tilde\G' + e_0$ to that of $\G = \G' + e_0$
by replacing the corresponding part of $\tilde\G'$ with $\around{\G'}{X}$
without violating the conditions in Lemma~\ref{lem:cDab0},
where recall that all arcs in $\around{\G'}{X} - y_1$ are with label $\id$
and all arcs around $y_1$ enter $y_1$ with the same label $\gamma \in \{\id, \alpha, \alpha^{-1}\}$.

\subsubsection{Case 2: Involving 2-contraction (when $(\G', s, t) \in \cDab \setminus \cDab^1$)}\label{sec:case2}
In this case,
$\G'$ contains a 2-contractible vertex set $X \subseteq V(\G) \setminus \{s, t\}$
by the definition of $\cDab$ (Definition~\ref{def:cDab}).
Due to Section~\ref{sec:case1} (Case 1 implies a contradiction),
we may assume that this situation occurs
regardless of the choice of the edge $e_0 = \{s, v_0\} \in \delta_\G(s)$.
If $e_0 = \{s, v_0\} \in \delta_\G(s)$ is unique (i.e., $|\delta_\G(s)| = 1$),
then we can construct a smaller counterexample $(\G - s, v_0, t) \in \cD$ by Lemma~\mbox{\ref{lem:cDab}-(2)},
and hence there are at least two candidates of $e_0$.

We first observe several useful properties used throughout this section. 

\begin{claim}\label{cl:3-cut}
  Let $X \subseteq V(\G) \setminus \{s, t\}$ be a vertex set
  with $N_\G(X) = \{x, y, z\}$ for some distinct $x, y, z \in V(\G)$
  such that $\G[X]$ is connected $($see Fig.~$\ref{fig:3-cut_general}$$)$.
  Then, the following properties hold.
  \begin{itemize}
  \setlength{\itemsep}{.5mm}
  \item[$(1)$]
  $(\around{\G}{X} - x, y, z) \in \cD$, $(\around{\G}{X} - y, z, x) \in \cD$, and $(\around{\G}{X} - z, x, y) \in \cD$.
  \item[$(2)$]
  At least two of $|l(\around{\G}{X} - x; y, z)| \geq 2$,  $|l(\around{\G}{X} - y; z, x)| \geq 2$, and $|l(\around{\G}{X} - z; x, y)| \geq 2$ hold.
  \item[$(3)$]
  If $|l(\around{\G}{X}; y, z)| = 1$,
  then $X = \{v\}$ for some $v \in V(\G) \setminus \{s, x, y, z, t\}$,
  and $E(\around{\G}{X})$ consists of a single edge between $v$ and $y$,
  one between $v$ and $z$, and two parallel edges $($with distinct labels$)$ between $v$ and $x$
  $($see Fig.~$\ref{fig:3-cut_special}$$)$.
  \end{itemize}
\end{claim}

\begin{figure}[htbp]\vspace{-3mm}\hspace{-3mm}
  \begin{tabular}{cc}
    \begin{minipage}[b]{0.5\hsize}
      \begin{center}
        \includegraphics[scale=0.8]{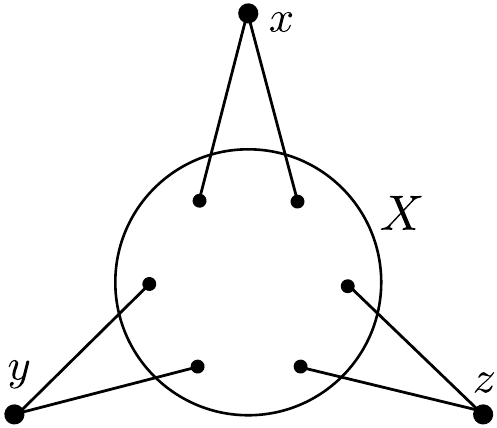}
      \end{center}\vspace{-5mm}
      \caption{The situation of Claim~\ref{cl:3-cut}.}
      \label{fig:3-cut_general}
    \end{minipage}
    \begin{minipage}[b]{0.5\hsize}
      \begin{center}
        \includegraphics[scale=0.8]{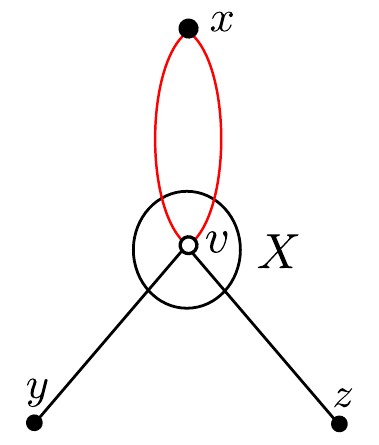}
      \end{center}\vspace{-5mm}
      \caption{When $|l(\around{\G}{X}; y, z)| = 1$.}
      \label{fig:3-cut_special}
    \end{minipage}
  \end{tabular}\vspace{-3mm}
\end{figure}

\begin{proof}
$(1)$
Suppose to the contrary that $(\around{\G}{X} - x, y, z) \not\in \cD$.
Then, $(\around{\G}{X} - x) + e_{yz}$ contains a 1-cut $w \in X \cup \{y, z\}$
by Lemma~\ref{lem:cD}, where $e_{yz} = \{y, z\}$ is a new edge (with an arbitrary label).
The vertex set of a connected component of $(\around{\G}{X} - x) + e_{yz} - w$
that contains none of $y$ and $z$ is separated from both $s$ and $t$
by $\{w, x\}$ in $\G$ (possibly $t = x$),
which is 2-contractible, contradicting Claim~\ref{cl:2-contractible}.
Thus, we have $(\around{\G}{X} - x, y, z) \in \cD$,
and also $(\around{\G}{X} - y, z, x) \in \cD$ and $(\around{\G}{X} - z, x, y) \in \cD$ by the symmetry of $x$, $y$, and $z$.

\medskip\noindent
$(2)$
By symmetry, suppose that $|l(\around{\G}{X} - x; y, z)| \leq 1$, and we show $|l(\around{\G}{X} - y; z, x)| \geq 2$ and $|l(\around{\G}{X} - z; x, y)| \geq 2$.
Since $\G[X]$ is connected, we have $|l(\around{\G}{X} - x; y, z)| = 1$,
and hence $\around{\G}{X} - x$ is balanced by Lemma~\ref{lem:balanced2} (with (1)).
Then, by Lemma~\ref{lem:shifting}, we may assume that all the arcs in $\around{\G}{X} - x$ are with label $\id$ by shifting at vertices in $X \cup \{y\}$ if necessary (note that any shifting operation preserves the balancedness and the number of possible labels of paths between two vertices). 
On the other hand, since $X$ is not 3-contractible in $\G$ (by Clam~\ref{cl:3-contractible}), $\around{\G}{X}$ is not balanced.
Hence, there exist two edges $e_1, e_2 \in \delta_{\G[\![X]\!]}(x)$ such that $\psi_{\G}(e_1, x) \neq \psi_{\G}(e_2, x)$.
Since $\G[X]$ is connected, $\around{\G}{X} - y$ contains two $z$--$x$ paths ending with $e_1$ and with $e_2$, whose labels are distinct,
and $\around{\G}{X} - z$ contains two $x$--$y$ paths starting with $e_1$ and with $e_2$, whose labels are distinct.
Thus we have $|l(\around{\G}{X} - y; z, x)| \geq 2$ and $|l(\around{\G}{X} - z; x, y)| \geq 2$.

\medskip\noindent
$(3)$
Suppose that $|l(\around{\G}{X}; y, z)| = 1$.
Then, since $\around{\G}{X}$ is not balanced (by Clam~\ref{cl:3-contractible}), Lemma~\ref{lem:balanced2} implies $(\around{\G}{X}, y, z) \not\in \cD$,
and hence $\around{\G}{X} + e_{yz}$ contains a 1-cut $v \in X$ by Lemma~\ref{lem:cD}, where $e_{yz} = \{y, z\}$ is a new edge (with an arbitrary label).
Since $(\around{\G}{X} - x, y, z) \in \cD$ implies that $v$ is not a 1-cut in $(\around{\G}{X} + e_{yz}) - x$,
the 1-cut $v$ separates $x$ from all the other vertices in $\around{\G}{X} + e_{yz}$,
which means that $v$ is a unique vertex in $N_{\G[\![X]\!]}(x)$.
On the other hand, since $(\around{\G}{X} - x, y, z) \in \cD$ and $l(\around{\G}{X} - x; y, z) \subseteq l(\around{\G}{X}; y, z)$,
Lemma~\ref{lem:balanced2} implies that $\around{\G}{X} - x$ is balanced.
Hence, there are two parallel edges between $v$ and $x$ that form an unbalanced cycle (recall that $\around{\G}{X}$ is not balanced).
Moreover, if $X \setminus \{v\} \neq \emptyset$, then $\G$ contains a contractible vertex set
$Y \subseteq X \setminus \{v\}$ with $N_\G(Y) \subseteq \{v, y, z\}$, a contradiction.
Thus we are done.
\end{proof}

We now start to derive a contradiction in Case 2, i.e., when $(\G', s, t) \in \cDab \setminus \cDab^1$.
Choose a minimal 2-contractible vertex set $X$ in $\G' = \G - e_0$, and let $N_{\G'}(X) = \{x, y\}$.
We then have $v_0 \in X$ and $s \not\in \{x, y\}$ by Claim~\ref{cl:2-contractible}
($X$ is not 2-contractible in $\G = \G' + e_0$),
$\G'[X] = \G[X]$ is connected (otherwise, some connected component $\G[Y]$ of $\G[X]$ does not contain $v_0$
and hence $N_\G(Y) \subseteq \{x, y\}$, contradicting Claim~\ref{cl:2-contractible} or $(\G, s, t) \in \cD$),
and $\around{\G'}{X}$ is not balanced by Claim~\ref{cl:3-contractible}
($X$ is not 3-contractible in $\G = \G' + e_0$).
Since Claim~\ref{cl:cD} implies $(\around{\G'}{X}, x, y) \in \cD$ (any $s$--$t$ path in $\G'$ intersecting some vertex in $X$ must intersect both $x$ and $y$),
we have $|l(\around{\G'}{X}; x, y)| \geq 2$ by Lemma~\ref{lem:balanced2}.
In particular, by Lemma~\ref{lem:unbalanced_cycle} and Claim~\ref{cl:2label},
we have $l(\around{\G'}{X}; x, y) = \{\alpha', \beta'\}$ for some $\alpha', \beta' \in \Gamma$ with $\alpha'{\beta'}^{-1} \neq \beta'{\alpha'}^{-1}$. 
This concludes $(\around{\G'}{X}, x, y) \in \cDabp^1$,
since $(\G, s, t) \not\in \cDab$ is a minimal counterexample and $X$ is a minimal 2-contractible vertex set in $\G'$.

\medskip
\noindent\underline{{\bf Case~2.1.}~~When $t \in \{x, y\}$.}

\medskip
Without loss of generality (by the symmetry of $x$ and $y$),
we may assume that $y = t$.

\medskip
\noindent\underline{{\bf Case~2.1.1.}~~When $V(\G) = X \cup \{s, x, t\}$.}

\medskip
Recall that $\G$ contains no edge between $s$ and $t$ (otherwise, $V(\G) \setminus \{s, t\}$ is 2-contractible, contradicting Claim~\ref{cl:2-contractible}).
Hence, 
by Lemma~\mbox{\ref{lem:cDab}-(2)}, 
$\G$ contains an edge between $s$ and $x$,
and there exists exactly one such edge $e = \{s, x\} \in E(\G)$ (see Fig.~\ref{fig:case_2.1.1}),
since $(\around{\G'}{X}, x, t) \in \cDabp^1$ and $|l(\G; s, t)| = 2$.
We assume $\psi_\G(e, x) = \id$ by shifting at $x$ if necessary, and then $(\around{\G'}{X}, x, t) \in \cD_{\id,\,\alpha}^1$
(recall that we may assume $\beta = \id$ and $\alpha^{-1} \neq \alpha$).
As with the previous section (Case 1),
let $\tilde\G \coloneqq \around{\G'}{X} \three Y$
for some 3-contractible vertex set $Y \subseteq X$ with $v_0 \in Y$ if exists,
and $\tilde\G \coloneqq \around{\G'}{X}$ otherwise,
so that $(\tilde\G, x, t) \in \cD_{\id,\,\alpha}^0$.
Consider the four cases in Definition~\ref{def:cDab0} separately.
\begin{figure}[b]\hspace{-3mm}
  \begin{tabular}{cc}
    \begin{minipage}[b]{0.5\hsize}
      \begin{center}
        \includegraphics[scale=0.7]{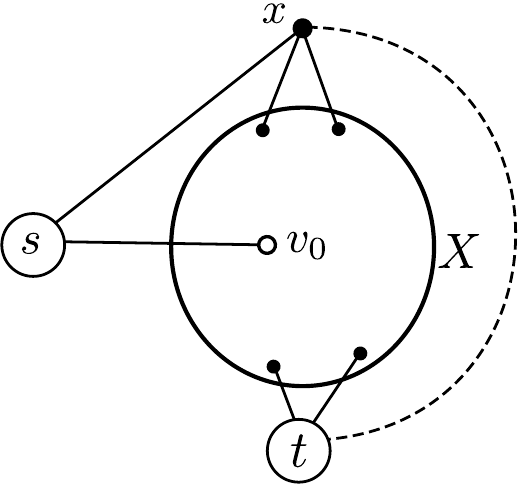}
      \end{center}\vspace{-3mm}
      \caption{Case~2.1.1.}
      \label{fig:case_2.1.1}
    \end{minipage}
    \begin{minipage}[b]{0.5\hsize}
      \begin{center}
        \includegraphics[scale=0.7]{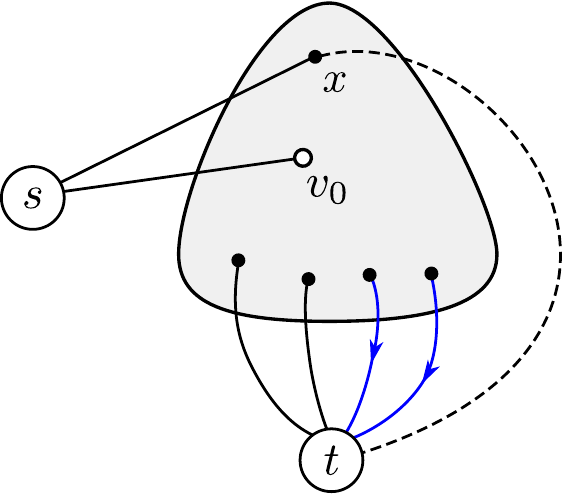}
      \end{center}\vspace{-3mm}
      \caption{Case~2.1.1.1.}
      \label{fig:case_2.1.1.1}
    \end{minipage}
  \end{tabular}
\end{figure}
\begin{description}
  \setlength{\itemsep}{1mm}
\item[Case~2.1.1.1.]
  Suppose that $(\tilde\G, x, t) \in \cD_{\id,\,\alpha}^0$ is in Case~(A2) (see Fig.~\ref{fig:case_2.1.1.1}).
  As with Case~1.1, we also have $(\around{\G'}{X}, x, t) \in \cD_{\id,\, \alpha}^0$,
  and we may assume that the label of every arc in $\around{\G'}{X} - t$ is $\id$ (by shifting at vertices in $X$ if necessary).
  If $\psi_\G(e_0, v_0) = \id$, then obviously $(\G, s, t) \in \cD_{\id,\,\alpha}^0$.
  Otherwise (i.e., if $\psi_\G(e_0, v_0) \neq \id$),
  since $\G[X]$ is connected, 
  there exists a $v_0$--$w$ path in $\around{\G'}{X}$
  for each neighbor $w \in N_\G(t)$, and hence $|l(\G, s, t)| \geq 3$ by Lemma~\ref{lem:2-2labels}, a contradiction.

\item[Case~2.1.1.2.]
  Suppose that $(\tilde\G, x, t) \in \cD_{\id,\,\alpha}^0$ is in Case~(A1) (see Fig.~\ref{fig:case_2.1.1.2_a}).
  Similarly (cf.~Case 1.1), we also have $(\around{\G'}{X}, x, t) \in \cD_{\id,\, \alpha}^0$,
  and we may assume that the label of every arc in $\around{\G'}{X} - x$ is $\id$ and
  every arc around $x$ other than $\vec{e} = sx$ leaves $x$ with label $\id$ or $\alpha$.

  Let $\bH$ be the graph obtained from $\G - s$
  (which coincides with $\around{\G'}{X}$ if $\{x, t\} \not\in E(\G)$)
  by splitting $x$ into two vertices $x_0$ and $x_1$
  so that every arc leaving $x$ in $\G - s$ with label $\alpha^i \in \{\id, \alpha\}$
  leaves $x_i$ in $\bH$ for each $i = 0, 1$ (see Fig.~\ref{fig:case_2.1.1.2_b}).

\begin{figure}[tbp]\vspace{0mm}\hspace{-3mm}
  \begin{tabular}{cc}
    \begin{minipage}[b]{0.5\hsize}
      \begin{center}
        \includegraphics[scale=0.7]{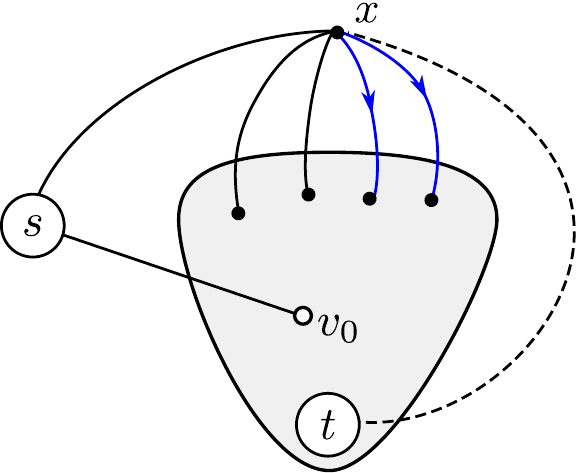}
      \end{center}\vspace{-3mm}
      \caption{Case~2.1.1.2.}
      \label{fig:case_2.1.1.2_a}
    \end{minipage}
    \begin{minipage}[b]{0.5\hsize}
      \begin{center}
        \includegraphics[scale=0.7]{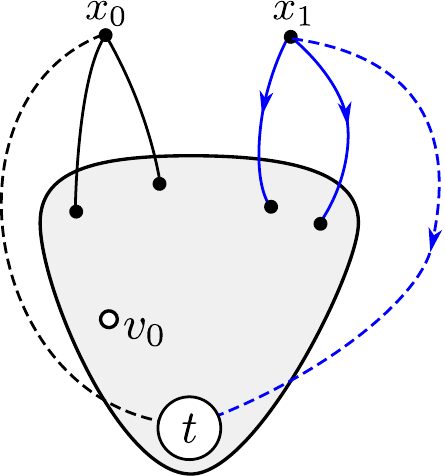}
      \end{center}\vspace{-3mm}
      \caption{$\bH$ in Case~2.1.1.2.}
      \label{fig:case_2.1.1.2_b}
    \end{minipage}
  \end{tabular}\vspace{-3mm}
\end{figure}
  
  Since $l(\G; s, t) = \{\id, \alpha\}$,
  either $\psi_\G(e_0, v_0) = \id$ or $\psi_\G(e_0, v_0) = \alpha$.
  Suppose that $\psi_\G(e_0, v_0) = \id$.
  If $\bH$ contains two disjoint paths, a $v_0$--$x_1$ path $P$ and an $x_0$--$t$ path $Q$,
  then we can construct an $s$--$t$ path of label $\alpha^{-1} \in \Gamma \setminus \{\id, \alpha\}$
  in $\G$ by concatenating $e_0$, $P$, and $Q$ with identifying $x_0, x_1 \in V(\bH)$ as $x \in V(\G)$.
  Otherwise, by Theorem~\ref{thm:2path}, $\bH$ can be embedded in the plane so that
  $v_0, x_0, x_1, t \in V(\bH)$ are on the outer boundary in this order
  (note that if there exists a vertex set
  $Y \subseteq V(\bH) \setminus \{v_0, x_0, x_1, t\} = V(\G) \setminus \{v_0, x, t\}$
  such that $|N_\bH(Y)| \leq 3$, then either $|N_\G(Y)| \leq 2$ or
  $|N_\G(Y)| \leq 3$ and $\around{\G}{Y}$ is balanced, which contradicts
  Claim~\ref{cl:2-contractible} or \ref{cl:3-contractible}, respectively).
  This embedding can be straightforwardly extended to an embedding of $\G$
  by merging $x_0, x_1 \in V(\bH)$ into $x \in V(\G)$
  and by adding the vertex $s$ and the two edges $e_0 = \{s, v_0\}$ and $e = \{s, x\}$ on the outer face.
  The resulting embedding satisfies
  the conditions of Case~(C) in Definition~\ref{def:cDab0}
  (cf.~Lemma~\ref{lem:cDab0}),
  which implies $(\G, s, t) \in \cD_{\id,\,\alpha}^0$, a contradiction.

  Otherwise, $\psi_\G(e_0, v_0) = \alpha$.
  Also in this case, by a similar argument to the above,
  we can either construct an $s$--$t$ path of label
  $\alpha^2 \in \Gamma \setminus \{\id, \alpha\}$
  in $\G$ by concatenating $e_0$ and 
  two disjoint paths, a $v_0$--$x_0$ path and an $x_1$--$t$ path,
  with identifying $x_0, x_1 \in V(\bH)$ as $x \in V(\G)$,
  or embed $\G$ so that $(\G, s, t) \in \cD_{\id,\,\alpha}^0$ is in Case~(C), a contradiction. 

\item[Case~2.1.1.3.]
  Suppose that $(\tilde\G, x, t) \in \cD_{\id,\,\alpha}^0$ is in Case~(B).
  If $\tilde\G = \around{\G'}{X}$, it is easy to confirm that
  $\{x\}$ is 3-contractible in $\G$ (if there is no edge between $x$ and $t$)
  or $|l(\G; s, t)| \geq 3$ (otherwise, i.e., if $\{x, t\} \in E(\G)$) by Lemma~\ref{lem:2-2labels} (see Fig.~\ref{fig:case_2.1.1.3_a}).

  Otherwise (i.e., if $\tilde\G = \around{\G'}{X} \three Y$ for some $Y \subseteq X$),
  we have either $N_{\G'}(Y) = \{x, v_1, v_2\}$
  or $N_{\G'}(Y) = \{v_3, v_4, t\}$.
  Suppose that $N_{\G'}(Y) = \{v_3, v_4, t\}$.
  In this case, we can derive a contradiction by Menger's theorem similarly to Case~1.2.
  That is, $\around{\G'}{Y}$ contains either two disjoint paths between
  $\{v_0, t\}$ and $\{v_3, v_4\}$ or a 1-cut $w \in Y$ separating them (possibly $w = v_0$).
  In the former case, $|l(\G; s, t)| \geq 3$ by Lemma~\ref{lem:2-2labels},
  and in the latter case, $\G$ contains a 2-cut $\{x, w\}$
  separating $\{v_3, v_4\}$ from $\{s, v_0, t\}$,
  which contradicts Claim~\ref{cl:2-contractible}.

\begin{figure}[htbp]\vspace{0mm}\hspace{-3mm}
  \begin{tabular}{cc}
    \begin{minipage}[b]{0.5\hsize}
      \begin{center}
        \includegraphics[scale=0.7]{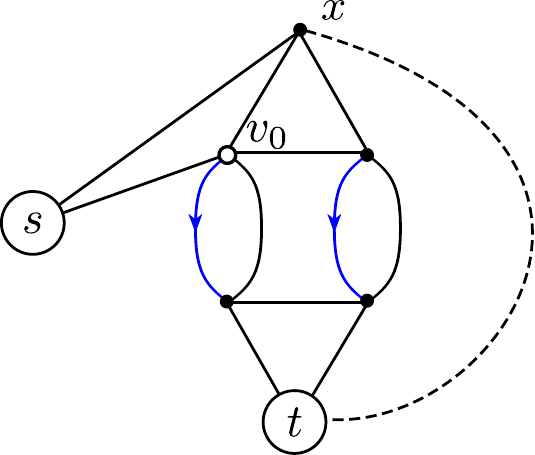}
      \end{center}\vspace{-2mm}
      \caption{Case~2.1.1.3 ($\tilde\G = \around{\G'}{X}$).}
      \label{fig:case_2.1.1.3_a}
    \end{minipage}
    \begin{minipage}[b]{0.5\hsize}
      \begin{center}
        \includegraphics[scale=0.7]{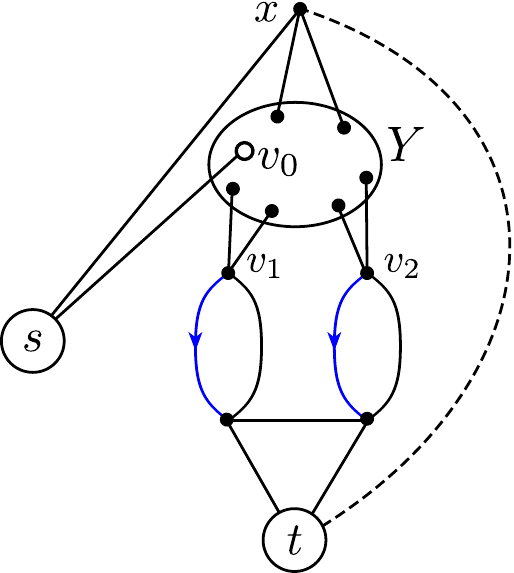}
      \end{center}\vspace{-5mm}
      \caption{Case~2.1.1.3 ($\tilde\G = \around{\G'}{X} \three Y$).}
      \label{fig:case_2.1.1.3_b}
    \end{minipage}
  \end{tabular}\vspace{-2mm}
\end{figure}

  Otherwise, $N_{\G'}(Y) = \{x, v_1, v_2\}$ (see Fig.~\ref{fig:case_2.1.1.3_b}).
  If $\{x, t\} \in E(\G)$, then we can similarly derive a contradiction by Menger's theorem, i.e.,
  either $|l(\G; s, t)| \geq 3$ by Lemma~\ref{lem:2-2labels}
  (if $\around{\G'}{Y}$ contains two disjoint paths between
  $\{v_0, x\}$ and $\{v_1, v_2\}$) or $\G$ contains a 2-cut $\{w, t\}$
  (if $\around{\G'}{Y}$ contains a 1-cut $w \in Y \cup \{x\}$ separating $\{v_0, x\}$ and $\{v_1, v_2\}$).
  Otherwise (i.e., $\{x, t\} \not\in E(\G)$), we have $N_\G(Y \cup \{x\}) = \{s, v_1, v_2\}$.
  Since $Y \cup \{x\}$ is not 3-contractible in $\G$ by Claim~\ref{cl:3-contractible},
  $\around{\G}{Y \cup \{x\}}$ is not balanced.
  If $|l(\around{\G}{Y \cup \{x\}}; s, v_1)| = 1$,
  then $(\around{\G}{Y \cup \{x\}}, s, v_1) \not\in \cD$ by Lemma~\ref{lem:balanced2}
  and hence there exists a 1-cut $w \in Y$ separating both $s$ and $v_1$ from all unbalanced cycles in $\around{\G}{Y \cup \{x\}}$
  by Lemma~\ref{lem:cD}.
  That is, $\G$ contains a 3-contractible vertex set $Z \subseteq Y \cup \{x\}$
  with $N_\G(Z) = \{s, v_1, w\}$ for some $w \in Y$
  (note that $\G'[Y] = \G[Y]$ is connected by Definition~\ref{def:3-contraction}),
  a contradiction.
  Otherwise, i.e., if $|l(\around{\G}{Y \cup \{x\}}; s, v_1)| \geq 2$, 
  we have $|l(\G; s, t)| \geq 3$ by Lemma~\ref{lem:2-2labels}, a contradiction.

\item[Case~2.1.1.4.]
  Suppose that $(\tilde\G, x, t) \in \cD_{\id,\,\alpha}^0$ is in Case~(C).
  In this case, by extending the $x$--$t$ path $P$ (in Lemma~\ref{lem:cDab0})
  to an $s$--$t$ path using the edge $e = \{s, x\}$,
  we can see that $(\G', s, t) \in \cD_{\id,\,\alpha}^0$
  (or $(\G' \three Y, s, t) \in \cD_{\id,\,\alpha}^0$ if $\tilde\G = \around{\G'}{X} \three Y$ for some 3-contractible $Y \subseteq V(\G) \setminus \{s, t\}$)
  is also in Case~(C) (see Fig.~\ref{fig:case_2.1.1.4}),
  which contradicts $(\G', s, t) \not\in \cD_{\id,\,\alpha}^1$.
\end{description}

\begin{figure}[htbp]\vspace{-2mm}
 \begin{center}
  \includegraphics[scale=0.7]{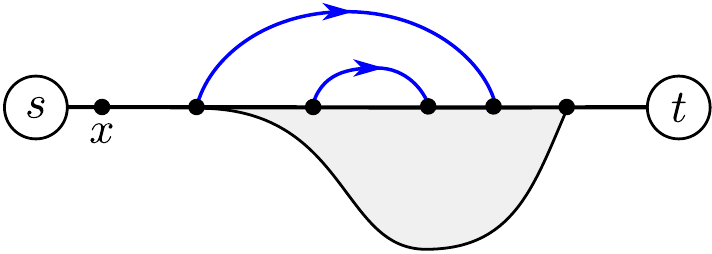}
 \end{center}\vspace{-5mm}
 \caption{Case~2.1.1.4 ($(\G', s, t) \in \cD_{\id,\, \alpha}^0$ is in Case (C) when so is $(\around{\G'}{X}, x, t) \in \cD_{\id,\, \alpha}^0$).}
 \label{fig:case_2.1.1.4}
\end{figure}

\noindent\underline{{\bf Case~2.1.2.}~~When $V(\G) \neq X \cup \{s, x, t\}$.}

\medskip
Let $Y \subseteq V(\G) \setminus (X \cup \{s, x, t\}) \neq \emptyset$
be the vertex set of a connected component of $\G - (X \cup \{s, x, t\})$.
We then have $N_\G(Y) \subseteq \{s, x, t\}$,
and by the definition of $\cD$ (any vertex in $Y$ is contained in some $s$--$t$ path in $\G$) and Claim~\ref{cl:2-contractible} ($Y$ is not 2-contractible in $\G$),
we must have $N_\G(Y) = \{s, x, t\}$ (see Fig.~\ref{fig:case_2.1.2}).
Moreover, we have $(\around{\G}{Y} - t, s, x) \in \cD$ by Claim~\ref{cl:3-cut}-(1).
These hold regardless of the choice of $Y$, and hence $(\G - (X \cup \{t\}), s, x) \in \cD$.
If $|l(\G - (X \cup \{t\}); s, x)| \geq 2$, then $|l(\G; s, t)| \geq 3$ by Lemma~\ref{lem:2-2labels} (recall that $(\around{\G'}{X}, x, t) \in \cDabp^1$), a contradiction.
Thus, we have $|l(\G - (X \cup \{t\}); s, x)| = 1$, and hence $\G - (X \cup \{t\})$ is balanced by Lemma~\ref{lem:balanced2}.
By Lemma~\ref{lem:shifting}, we may assume that all the arcs in $\G - (X \cup \{t\})$ are with label $\id$ by shifting at some vertices in $V(\G) \setminus (X \cup \{s, t\})$ if necessary.

\begin{figure}[htb]
 \begin{center}
  \includegraphics[scale=0.7]{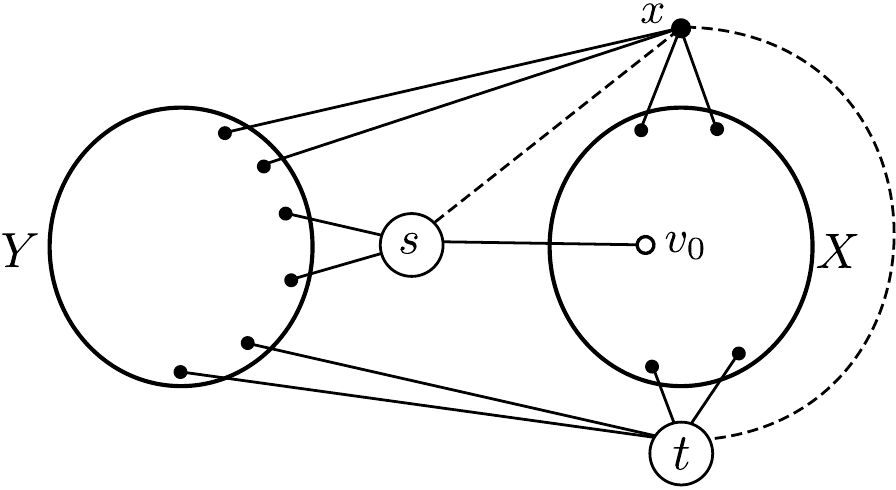}
 \end{center}\vspace{-5mm}
 \caption{Case~2.1.2.}
 \label{fig:case_2.1.2}
\end{figure}

We fix a connected component of $\G - (X \cup \{s, x, t\})$ with vertex set $Y$ (see again Fig.~\ref{fig:case_2.1.2}).
Claim~\ref{cl:3-contractible} implies that $\around{\G}{Y}$ is not balanced.
Since $|l(\around{\G}{Y} - t; s, x)| = 1$ (note that $\around{\G}{Y} - t$ is a subgraph of $\G - (X \cup \{t\})$),
we have $|l(\around{\G}{Y} - s; x, t)| \geq 2$ by Claim~\ref{cl:3-cut}-(2).
Since $l(\G; s, t) = \{\alpha, \beta\}$ with $\alpha\beta^{-1} \neq \beta\alpha^{-1}$,
we have $l(\around{\G}{Y} - s; x, t) = \{\alpha'', \beta''\}$ for some $\alpha'', \beta'' \in \Gamma$
with $\alpha''{\beta''}^{-1} \neq \beta''{\alpha''}^{-1}$ by Lemma~\ref{lem:unbalanced_cycle}. 
Moreover, since $(\around{\G}{X} - t, s, x) \in \cD$ by Claim~\ref{cl:3-cut}-(1),
if $\around{\G}{X} - t$ is not balanced, then $|l(\G; s, t)| \geq 3$ by Lemmas~\ref{lem:balanced2} and \ref{lem:2-2labels}, a contradiction.
Thus, $\around{\G}{X} - t$ is balanced,
and hence by Lemma~\ref{lem:shifting}, we may assume that all the arcs in $\around{\G}{X} - t - e_0$ are with label $\id$
by shifting at some vertices in $X$ after shifting at the vertices in $V(\G) \setminus (X \cup \{s, t\})$ if necessary
(recall that $\delta_{\G[\![X]\!]}(s) = \{e_0\}$ since $X$ is 2-contractible in $\G' = \G - e_0$).

We now have that almost all the arcs in $\G - t$ but $\vec{e}_0 = sv_0$ are with label $\id$.
If $\psi_\G(e_0, v_0) = \id$, then $\G - t$ is balanced, and hence $(\G, s, t) \in \cDab^0$ is in Case (A2) in Definition~\ref{def:cDab0}, a contradiction.
Otherwise, we see $|l(\G; s, t)| \geq 3$ by Lemma~\ref{lem:2-2labels}, a contradiction, as follows.
We choose $P_1 \coloneqq (s, e_0, v_0)$ and $P_2$ as an arbitrary $s$--$x$ path in $\G - (X \cup \{t\})$
so that $\psi_\G(P_1) = \psi_\G(e_0, v_0) \neq \id = \psi_\G(P_2)$.
Recall that $(\around{\G'}{X}, x, t) \in \cDabp^1$ for some $\alpha', \beta' \in \Gamma$ with $\alpha'{\beta'}^{-1} \neq \beta'{\alpha'}^{-1}$,
and note that shifting at $x$ preserves this inequality by definition (cf.~Definition~\ref{def:shifting}).
Since all the arcs in $\around{\G'}{X} = \around{\G}{X} - e_0$ but those around $t$ are with label $\id$,
there are (at least) two arcs entering $t$ from $X$ with different labels $\alpha', \beta' \in \Gamma$ with $\alpha'{\beta'}^{-1} \neq \beta'{\alpha'}^{-1}$.
Also, recall that $\G[X]$ is connected as discussed just before starting Case~2.1,
which implies that, for any neighbor $z \in \delta_{\G[\![X]\!]}(t)$, there exist a $v_0$--$z$ path in $\G[X]$ and an $x$--$z$ path in $\around{\G}{X} - \{s, t\}$.
Thus we can apply Lemma~\ref{lem:2-2labels} to derive $|l(\G; s, t)| \geq 3$, a contradiction.

\medskip
\noindent\underline{{\bf Case~2.2.}~~When $t \not\in \{x, y\}$.}

\medskip
Suppose that $V(\G) = X \cup \{s, x, y, t\}$ (see Fig.~\ref{fig:case_2.2.0}).
Then, by the symmetry of $x$ and $y$,
we may assume that there exists an edge $e = \{s, x\} \in \delta_\G(s)$,
for which $(\G - e, s, t) \in \cDab \setminus \cDab^1$
(otherwise, we can choose $e$ instead of $e_0$, and reduce this case to Case 1).
Moreover, $t$ is adjacent to both $x$ and $y$ since otherwise $\{s, y\}$ or $\{s, x\}$ is a 2-cut in $\G$,
which contradicts Claim~\ref{cl:2-contractible}.
Hence, by choosing $e$ instead of $e_0$,
we can reduce this case to Case~2.1 (since $x$ and $t$ are adjacent,
$t$ must be a neighbor of any 2-contractible vertex set in $\G - e$ that contains $x$).

In what follows, we assume $V(\G) \neq X \cup \{s, x, y, t\}$. Then, the following claim holds.

\begin{figure}[bp]\vspace{-1mm}\hspace{-3mm}
  \begin{tabular}{cc}
    \begin{minipage}[b]{0.5\hsize}
      \begin{center}
        \includegraphics[scale=0.7]{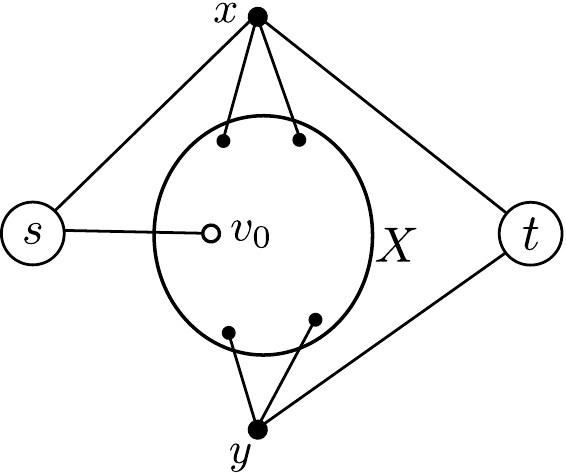}
      \end{center}\vspace{-5mm}
      \caption{Case~2.2 ($V(\G) = X \cup \{s, x, y, t\}$).}
      \label{fig:case_2.2.0}
    \end{minipage}
    \begin{minipage}[b]{0.5\hsize}
      \begin{center}
        \includegraphics[scale=0.7]{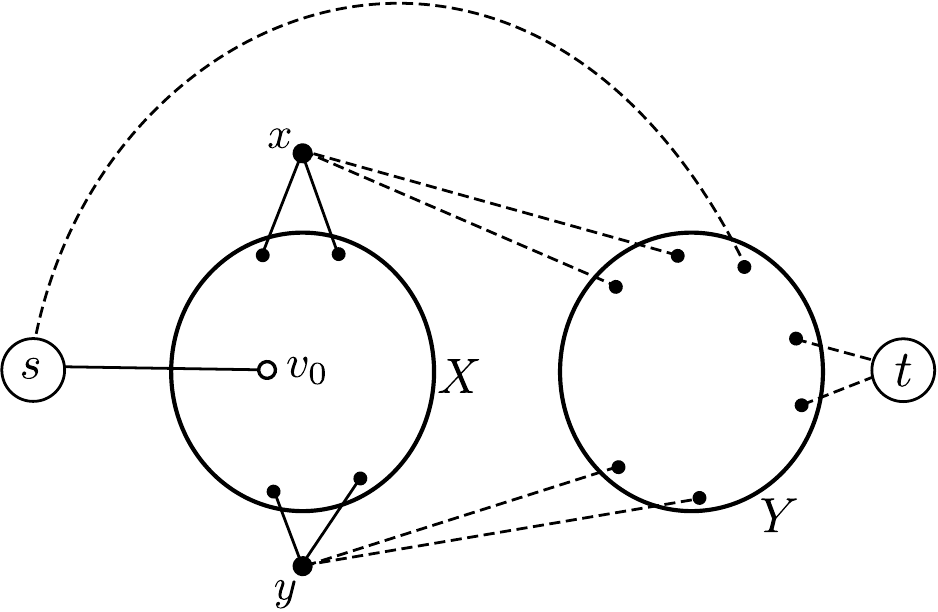}
      \end{center}\vspace{-5mm}
      \caption{Case~2.2 ($V(\G) \neq X \cup \{s, x, y, t\}$).}
      \label{fig:case_2.2}
    \end{minipage}
  \end{tabular}\vspace{-3mm}
\end{figure}

\begin{claim}\label{cl:Y}
  Let $Y \subseteq V(\G) \setminus (X \cup \{s, x, y, t\}) \neq \emptyset$ be the vertex set of a connected component of $\G - (X \cup \{s, x, y, t\})$ $($see Fig.~$\ref{fig:case_2.2}$$)$.
  Then, either $N_\G(Y) = \{s, x, y, t\}$ or $|N_\G(Y)| = 3$, and $\around{\G}{Y}$ is not balanced in the latter case.
\end{claim}

\begin{proof}
Since $\{s, x, y\}$ is a 3-cut separating $X$ from $t$ in $\G$, we have $N_\G(Y) \subseteq \{s, x, y, t\}$.
Moreover, since $Y$ is not contractible in $\G$ by Claims~\ref{cl:2-contractible} and \ref{cl:3-contractible} (and $(\G, s, t) \in \cD$),
we have $|N_\G(Y)| \geq 3$ and if $|N_\G(Y)| = 3$ then $\around{\G}{Y}$ is not balanced.
\end{proof}

Fix such $Y \subseteq V(\G) \setminus (X \cup \{s, x, y, t\})$, and consider the two cases in Claim~\ref{cl:Y} separately.

\medskip
\noindent\underline{{\bf Case~2.2.1.}~~When $N_{\G}(Y) = \{s, x, y, t\}$.}

\medskip
In particular, we may assume that there exists an edge $e' = \{s, v'\} \in \delta_\G(s)$ with $v' \in Y$
such that $\G - e'$ contains a 2-contractible vertex set $X' \subseteq V(\G) \setminus \{s, t\}$
with $v' \in X'$ and $N_{\G - e'}(X') = \{x', y'\}$
for some distinct $x', y' \in V(\G) \setminus \{s, v'\}$
(recall that, if $\G - e'$ contains no 2-contractible vertex set,
then we can reduce this case to Case~1 by choosing $e'$ instead of $e_0$).
Choose minimal $X'$, and then $\G[X']$ is connected.
If $t \in \{x', y'\}$, then this case reduces to Case~2.1 by choosing $e'$ instead of $e_0$.
Otherwise, since $\G[Y]$ is connected, we have $\{x', y'\} \cap Y \neq \emptyset$.
Without loss of generality, we assume $y' \in Y$.
We first consider the case when $x' \in V(\G) \setminus (Y \cup \{s, x, y, t\})$. 
\begin{description}
\item[Case~2.2.1.1 {\rm (Fig.~\ref{fig:case_2.2.1})}.]
Suppose that $x' \in V(\G) \setminus (Y \cup \{s, x, y, t\})$.
Since $\G[Y]$ is connected, $y'$ $(\not\in \{v', t\})$ separates $v'$ from $t$ in $\around{\G}{Y} - s$.
If $y'$ separates $v'$ also from both $x$ and $y$ in $\around{\G}{Y} - s$, then $\{s, y'\}$ is a 2-cut separating $v'$ from $t$ in $\G$,
which contradicts Claim~\ref{cl:2-contractible}.
Moreover, if $y'$ separates $v'$ from either $x$ or $y$ in $\around{\G}{Y} - s$,
then $\{y, y'\}$ or $\{x, y'\}$, respectively,
is a 2-cut separating $v'$ from $t$ in $\G - e'$, which contradicts the minimality of $X'$ (since $y \in X'$ or $x \in X'$, respectively).
Hence, we may assume $\{x, y\} \subseteq X'$, and then $\{x, t\}, \{y, t\} \not\in E(\G)$.

In particular, $\around{\G}{Y} - y'$ contains a $v'$--$x$ path and a $v'$--$y$ path,
which can be extended to $v'$--$s$ paths in $\G - y'$ through $\around{\G}{X}$.
The other vertex $x'$ must be on such a path (otherwise, $s \in X'$, a contradiction) but $x' \not\in \{s, x, y\}$, and hence $x' \in X$. 
If $Y \neq V(\G) \setminus (X \cup \{s, x, y, t\})$,
then there exists $Y' \subseteq V(\G) \setminus (X \cup Y \cup \{s, x, y, t\})$ such that $\G[Y']$ is connected
and $N_\G(Y') \cap \{x, y\} \neq \emptyset \neq N_\G(Y') \cap \{s, t\}$ (by Claim~\ref{cl:Y}),
and hence $\{x', y'\}$ cannot separate $v'$ from both $s$ and $t$ in $\G - e'$, a contradiction.
Thus, we have $Y = V(\G) \setminus (X \cup \{s, x, y, t\})$,
and hence $\{s, y'\}$ is a 2-cut in $\G$ separating $v'$ from $t$,
which contradicts Claim~\ref{cl:2-contractible}. 
\end{description}

We next consider the case when $x' \in Y \cup \{x, y\}$.
We then have $X' \subseteq Y$ and $(\around{\G}{X'} - s, x', y') \in \cD_{\alpha'', \beta''}^1$ for some $\alpha'', \beta'' \in \Gamma$
with $\alpha''{\beta''}^{-1} \neq \beta''{\alpha''}^{-1}$ as with $(\around{\G}{X} - s, x, y) \in \cDabp^1$ discussed just before starting Case~2.1 (where note that $\around{\G}{X'} - s = \around{(\G - e')}{X'}$ and $\around{\G}{X} - s = \around{\G'}{X}$).
By Claim~\ref{cl:3-cut}-(2), without loss of generality (by the symmetry of $x$ and $y$),
we may assume that $|l(\around{\G}{X} - y; s, x)| \geq 2$.

Let $Z \subseteq Y \cup \{x, y, t\}$ be the set of vertices that are contained in some $x$--$t$ path in $\G[Y \cup \{x, y, t\}]$,
i.e., $\G[Z] + e_{xt}$ is a 2-connected component of $\G[Y \cup \{x, y, t\}] + e_{xt}$ by Lemma~\ref{lem:cD}, where $e_{xt} = \{x, t\}$ is a new edge (with an arbitrary label).
Since $(\G[Z], x, t) \in \cD$, if $\G[Z]$ is not balanced, then $|l(\G[Z]; x, t)| \geq 2$ by Lemma~\ref{lem:balanced2},
and hence $|l(\G; s, t)| \geq 3$ by Lemmas~\ref{lem:unbalanced_cycle} and \ref{lem:2-2labels}, a contradiction.
In particular, since $(\around{\G}{X'} - s, x', y') \in \cD_{\alpha'', \beta''}^1$ and $\{x', y'\}$ is a 2-cut in $\G[Y \cup \{x, y, t\}]$ as well as in $\G - e'$,
if $Z \cap X' \neq \emptyset$, then $Z \supseteq X' \cup \{x', y'\} = V(\around{\G}{X'} - s)$, and hence $\G[Z]$ is not balanced as well as $\around{\G}{X'} - s$.
Thus, this cannot occur, and we have $Z \cap X' = \emptyset$ (and equivalently $|Z \cap \{x', y'\}| \leq 1$).
Moreover, if $y \in Z$, then there exists a 1-cut $w \in Z \setminus \{x, t\} \subseteq Y \setminus X'$ separating $X'$ from $\{x, y, t\}$ in $\G[Y \cup \{x, y, t\}]$ by Lemma~\ref{lem:cD},
which implies that $\{s, w\}$ is a 2-cut separating $X'$ from $t$ in $\G$, contradicting Claim~\ref{cl:2-contractible}.

\begin{figure}[tbp]\hspace{-2mm}
  \begin{tabular}{cc}
    \begin{minipage}[b]{0.45\hsize}
      \begin{center}
        \includegraphics[scale=0.7]{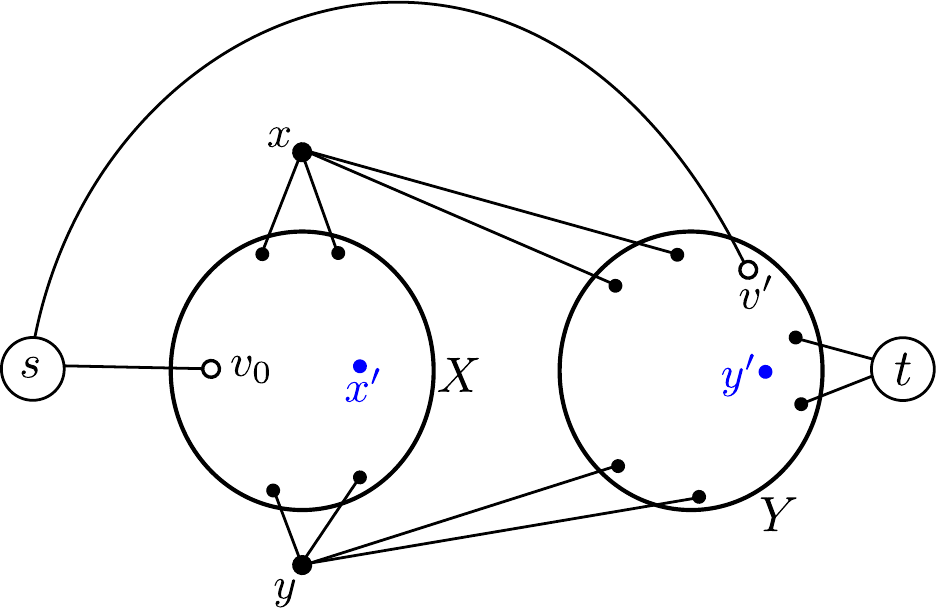}
      \end{center}\vspace{0mm}
      \caption{Case~2.2.1.1.}
      \label{fig:case_2.2.1}
    \end{minipage}
    \begin{minipage}[b]{0.55\hsize}
      \begin{center}
        \includegraphics[scale=0.7]{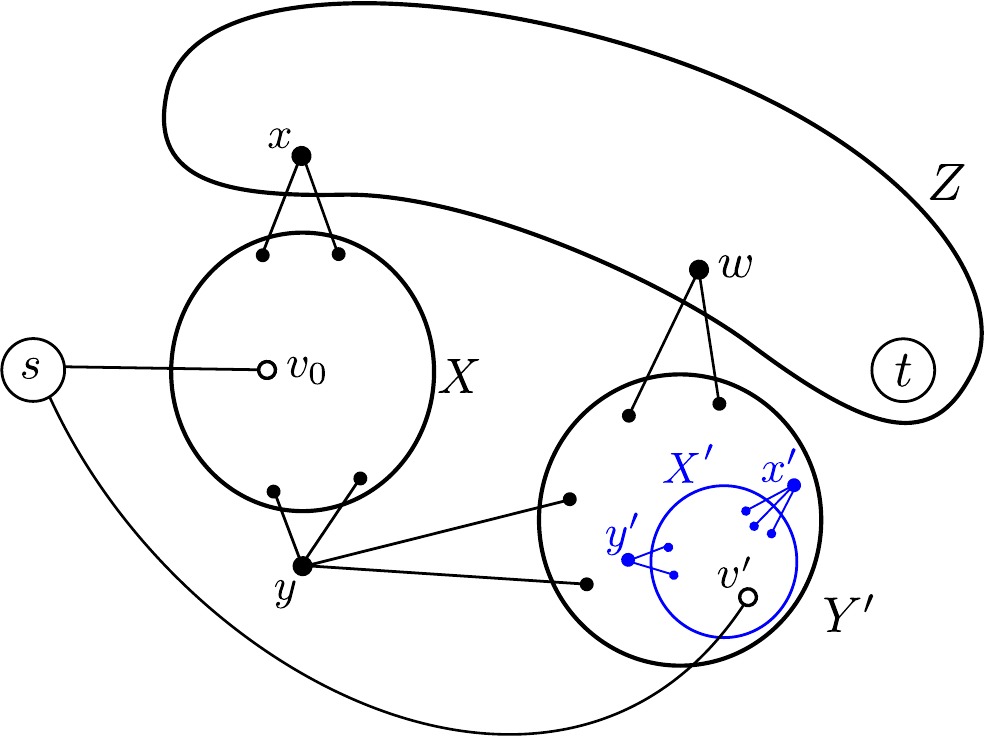}
      \end{center}\vspace{-3mm}
      \caption{$Z$ and $Y'$ in $\G[Y \cup \{x, y, t\}]$.}
      \label{fig:case_2.2.1.12}
    \end{minipage}
  \end{tabular}\vspace{-2mm}
\end{figure}

Thus, there exists a 1-cut $w \in Z \setminus \{x, t\}$ separating $X' \cup \{y\}$ from $\{x, t\}$ in $\G[Y \cup \{x, y, t\}]$.
Let $Y' \subseteq Y \setminus Z \neq \emptyset$ be the vertex set of a connected component of $\G[Y \setminus Z]$ (see Fig.~\ref{fig:case_2.2.1.12}),
and then $N_\G(Y') = \{s, w, y\}$ (otherwise, $Y'$ is 2-contractible in $\G$ or consists of vertices not contained in any $s$--$t$ path in $\G$, contradicting Claim~\ref{cl:2-contractible} or $(\G, s, t) \in \cD$, respectively).
By Claim~\ref{cl:3-cut}-(2), we have at least one of $|l(\around{\G}{Y'} - w; s, y)| \geq 2$ and $|l(\around{\G}{Y'} - s; y, w)| \geq 2$.
If the former holds, then we immediately obtain $|l(\G; s, t)| \geq 3$, a contradiction, from Lemma~\ref{lem:2-2labels}
by concatenating two $s$--$y$ paths of distinct labels in $\around{\G}{Y'} - w$, two $y$--$x$ paths of distinct labels in $\around{\G'}{X} = \around{\G}{X} - s$, and an arbitrary $x$--$t$ path in $\G[Z]$.
Hence, we have $|l(\around{\G}{Y'} - s; y, w)| \geq 2$.
\begin{description}
  \setlength{\itemsep}{1mm}
\item[Case~2.2.1.2.]
Suppose that $V(\G) \neq X \cup Y \cup \{s, x, y, t\}$ in addition to $x' \in Y \cup \{x, y\}$.
Then, there exists a vertex set $W \subseteq V(\G) \setminus (X \cup Y \cup \{s, x, y, t\}) \neq \emptyset$ such that
$\G[W]$ is connected and $N_\G(W)$ includes at least one of $\{s, x\}$ and $\{y, t\}$ (by Claim~\ref{cl:Y}).
When $\{s, x\} \subseteq N_\G(W)$ (see Fig.~\ref{fig:case_2.2.1.2-1}), we derive $|l(\G; s, t)| \geq 3$, a contradiction, from Lemma~\ref{lem:2-2labels}
by concatenating an $s$--$x$ path in $\around{\G}{W}$, two $x$--$y$ paths of distinct labels in $\around{\G'}{X} = \around{\G}{X} - s$, two $y$--$w$ paths of distinct labels in $\around{\G}{Y'} - s$, and a $w$--$t$ path in $\G[Z] - x$ (recall that $\G[Z]$ contains an $x$--$t$ path intersecting $w \neq x$).
When $\{y, t\} \subseteq N_\G(W)$ (see Fig.~\ref{fig:case_2.2.1.2-2}), we also derive $|l(\G; s, t)| \geq 3$, a contradiction, from Lemma~\ref{lem:2-2labels}
by concatenating two $s$--$x$ paths of distinct labels in $\around{\G}{X} - y$, an $x$--$w$ path in $\G[Z] - t$ (recall that $\G[Z]$ contains an $x$--$t$ path intersecting $w \neq t$),
two $w$--$y$ paths of distinct labels in $\around{\G}{Y'} - s$, and a $y$--$t$ path in $\around{\G}{W}$.

\begin{figure}[htbp]\vspace{3mm}\hspace{-3mm}
  \begin{tabular}{lr}
    \begin{minipage}[b]{0.5\hsize}
      \begin{center}
        \includegraphics[scale=0.67]{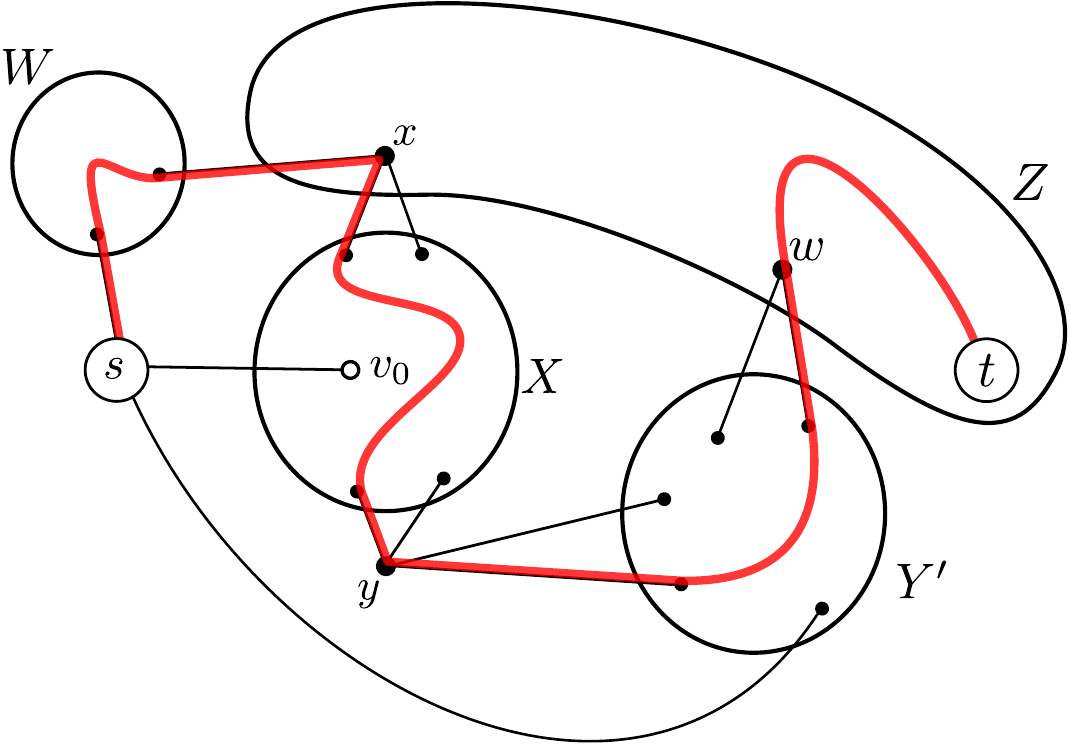}
      \end{center}\vspace{-3mm}
      \caption{Case~2.2.1.2 ($\{s, x\} \subseteq N_\G(W)$).}
      \label{fig:case_2.2.1.2-1}
    \end{minipage}
    \begin{minipage}[b]{0.5\hsize}
      \begin{center}
        \includegraphics[scale=0.67]{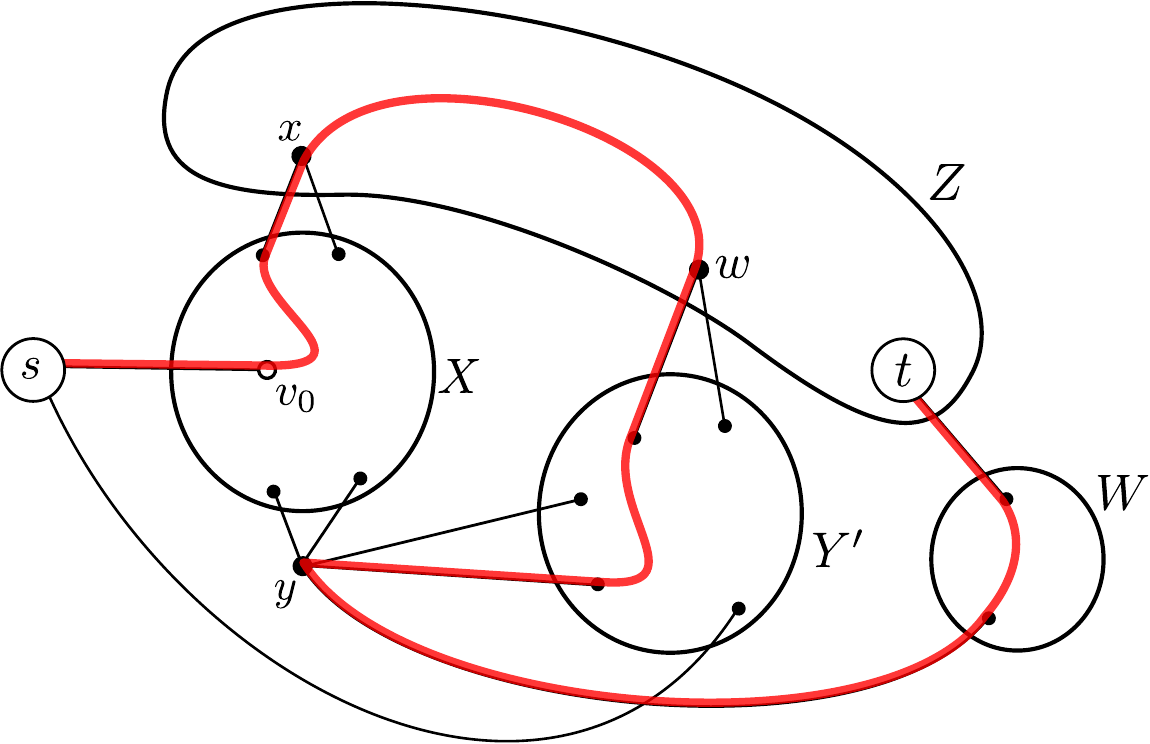}
      \end{center}\vspace{-3mm}
      \caption{Case~2.2.1.2 ($\{y, t\} \subseteq N_\G(W)$).}
      \label{fig:case_2.2.1.2-2}
    \end{minipage}
  \end{tabular}
\end{figure}

\item[Case~2.2.1.3.]
Otherwise, $V(\G) = X \cup Y \cup \{s, x, y, t\}$ and $x' \in Y \cup \{x, y\}$.
If $w$ separates $x$ from $t$ in $\G[Z]$,
then $\{s, w\}$ is a 2-cut in $\G$ (recall that $\G[Z] + e_{xt}$ is a 2-connected component of $\G[Y \cup \{x, y, t\}] + e_{xt}$, where $e_{xt} = \{x, t\}$),
contradicting Claim~\ref{cl:2-contractible}.
Hence, $\G[Z] - w$ contains an $x$--$t$ path.
If $|l(\around{\G}{Y'}; s, y)| \geq 2$, then we obtain $|l(\G; s, t)| \geq 3$, a contradiction,
by Lemma~\ref{lem:2-2labels} (take two $s$--$y$ paths of distinct labels in $\around{\G}{Y'}$, two $y$--$x$ paths of distinct labels in $\around{\G'}{X} = \around{\G}{X} - s$, and an $x$--$t$ path in $\G[Z] - w$).
Otherwise, by Claim~\ref{cl:3-cut}-(3), we have $Y' = \{v\}$ for some $v \in V(\G) \setminus (X \cup Z \cup \{s, y\})$,
and $E(\around{\G}{Y'})$ consists of a single edge between $v$ and $s$, one between $v$ and $y$, and two parallel edges between $v$ and $w$.
Moreover, if $Y' \neq Y \setminus Z$, then there exists another $Y'' \subseteq Y \setminus Z$ such that $\G[Y'']$ is connected
and $N_\G(Y'') = \{s, w, y\}$ (see Fig.~\ref{fig:case_2.2.1.3-2}), and we then have $|l(\G; s, t)| \geq 3$, a contradiction, by Lemma~\ref{lem:2-2labels} (take two $s$--$w$ paths of distinct labels in $\around{\G}{Y'} - y$, a $w$--$y$ path in $\around{\G}{Y''} - s$, two $y$--$x$ paths of distinct labels in $\around{\G'}{X} = \around{\G}{X} - s$, and an $x$--$t$ path in $\G[Z] - w$). Thus we have $Y' = Y \setminus Z$,
and also $\{y, w\} \not\in E(\G)$ by the same reason.

Since $x$ does not separate $w$ from $t$ in $\G[Z]$ (by the definition of $Z$),
we have $|l(\around{\G}{X}; s, y)| = 1$ (otherwise, Lemma~\ref{lem:2-2labels} concludes $|l(\G; s, t)| \geq 3$ as with the above discussion).
Hence, by Claim~\ref{cl:3-cut}-(3),
we have $X = \{v_0\}$ 
and $E(\around{\G}{X})$ consists of a single edge between $v_0$ and $s$, one between $v_0$ and $y$, and two parallel edges between $v_0$ and $x$.
Since $\{y\}$ is not contractible in $\G$ (by Claims~\ref{cl:2-contractible} and \ref{cl:3-contractible}),
there exist parallel edges between $s$ and $y$ (see Fig.~\ref{fig:case_2.2.1.3-3}),
where recall $Y \setminus Z = Y' = \{v\}$ and $\{y, w\} \not\in E(\G)$, and also that $w$ separates $y$ from $Z$ in $\G[Y \cup \{x, y, t\}]$.
We then obtain $|l(\G; s, t)| \geq 3$, a contradiction, by Lemma~\ref{lem:2-2labels} (take two $s$--$y$ paths of distinct labels consisting of parallel edges, two $y$--$x$ paths of distinct labels in $\around{\G'}{X} = \around{\G}{X} - s$, and an $x$--$t$ path in $\G[Z]$).
\end{description}

\begin{figure}[htb]\vspace{0mm}\hspace{-3mm}
  \begin{tabular}{cc}
    \begin{minipage}[b]{0.5\hsize}
      \begin{center}
        \includegraphics[scale=0.7]{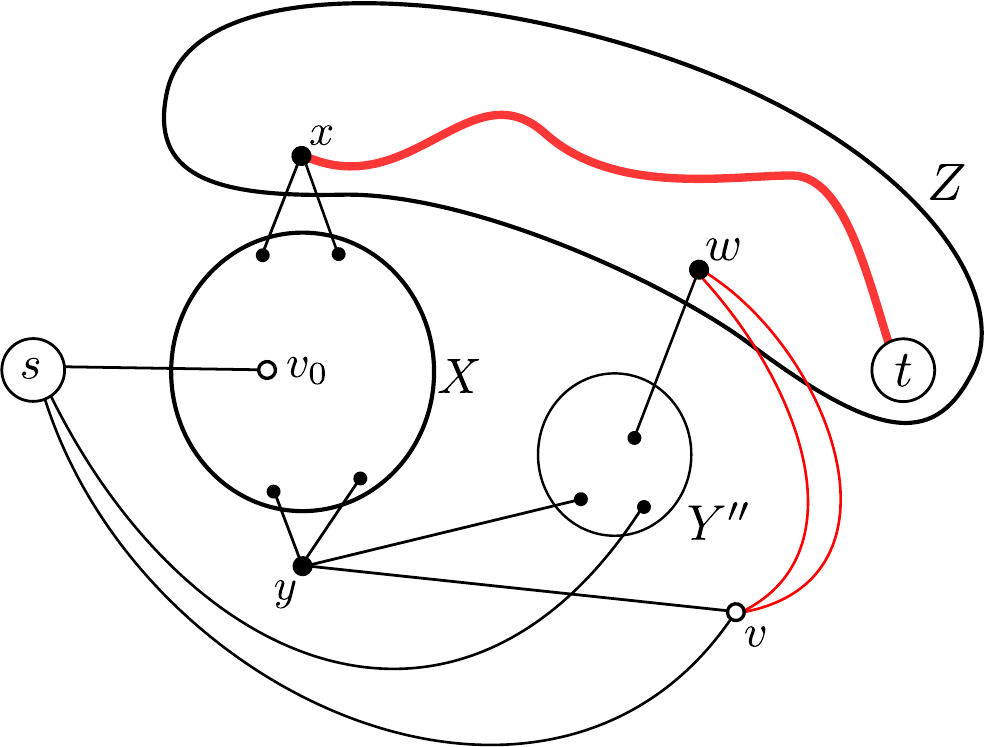}
      \end{center}\vspace{-4mm}
      \caption{Case~2.2.1.3 ($Y' \neq Y \setminus Z$).}
      \label{fig:case_2.2.1.3-2}
    \end{minipage}
    \begin{minipage}[b]{0.5\hsize}
      \begin{center}
        \includegraphics[scale=0.7]{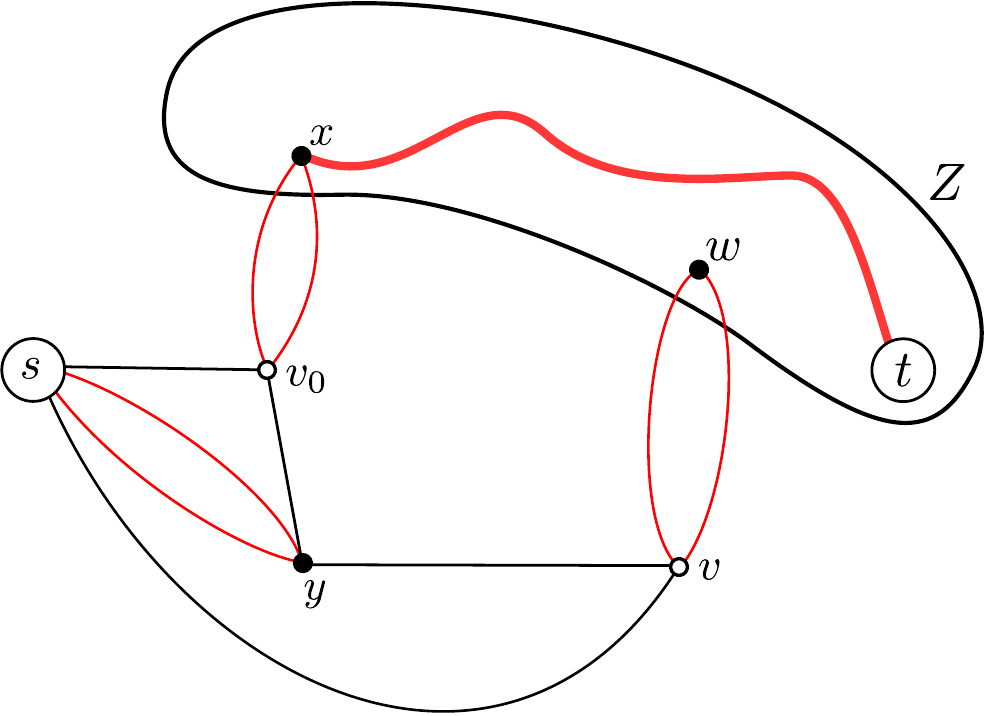}
      \end{center}\vspace{-3mm}
      \caption{Case~2.2.1.3 ($|X| = |Y'| = 1$).}
      \label{fig:case_2.2.1.3-3}
    \end{minipage}
  \end{tabular}
\end{figure}

\noindent\underline{{\bf Case~2.2.2.}~~When $|N_\G(Y)| = 3$.}

\medskip
We then have $\around{\G}{Y}$ is not balanced by Claim~\ref{cl:3-contractible}.
We discuss the following three cases separately:
when $N_\G(Y) = \{s, x, y\}$, when $N_\G(Y) = \{x, y, t\}$, and when $N_\G(Y) = \{s, x, t\}$ (including the case $N_\G(Y) = \{s, y, t\}$ by symmetry).
\begin{description}
  \setlength{\itemsep}{1mm}
\item[Case~2.2.2.1 {\rm (Fig.~\ref{fig:case_2.2.2.1})}.]
Suppose that $N_\G(Y) = \{s, x, y\}$.
Since neither $\{s, x\}$ nor $\{s, y\}$ is a 2-cut in $\G$ (by Claim~\ref{cl:2-contractible}),
there exist a $y$--$t$ path in $\G - (X \cup Y \cup \{s, x\})$ and an $x$--$t$ path in $\G - (X \cup Y \cup \{s, y\})$.
By Claim~\ref{cl:3-cut}-(2), we have $|l(\around{\G}{Y} - y; s, x)| \geq 2$ or $|l(\around{\G}{Y} - x; s, y)| \geq 2$,
and hence we derive $|l(\G; s, t)| \geq 3$, a contradiction, from Lemma~\ref{lem:2-2labels},
e.g., by taking two $s$--$x$ paths of distinct labels in $\around{\G}{Y} - y$, two $x$--$y$ paths of distinct labels in $\around{\G'}{X} = \around{\G}{X} - s$ (recall that $(\around{\G'}{X}, x, y) \in \cDabp^1$), and a $y$--$t$ path in $\G - (X \cup Y \cup \{s, x\})$.

\begin{figure}[htb]\vspace{0mm}\hspace{-3mm}
  \begin{tabular}{cc}
    \begin{minipage}[b]{0.5\hsize}
      \begin{center}
        \includegraphics[scale=0.7]{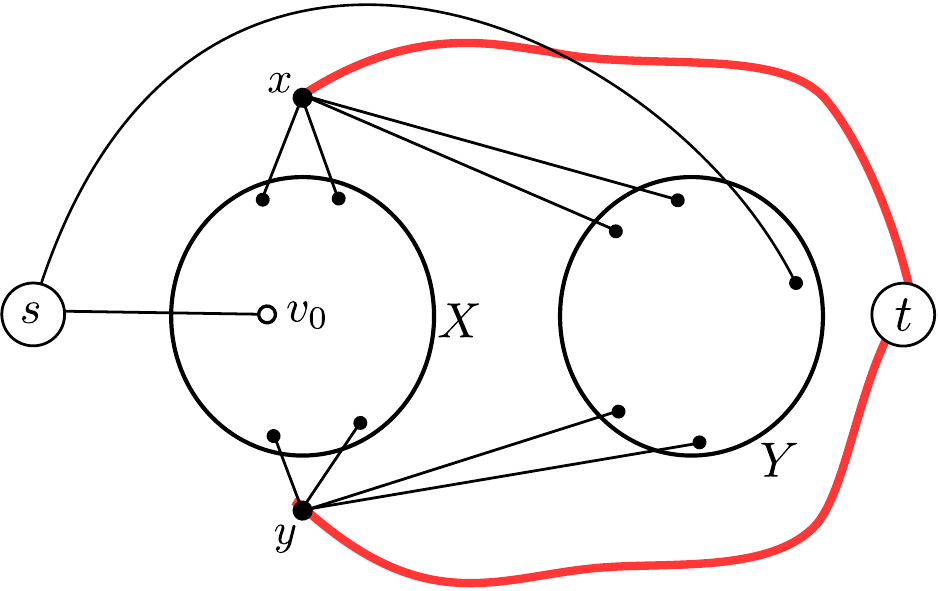}
      \end{center}\vspace{-4mm}
      \caption{Case~2.2.2.1.}
      \label{fig:case_2.2.2.1}
    \end{minipage}
    \begin{minipage}[b]{0.5\hsize}
      \begin{center}
        \includegraphics[scale=0.7]{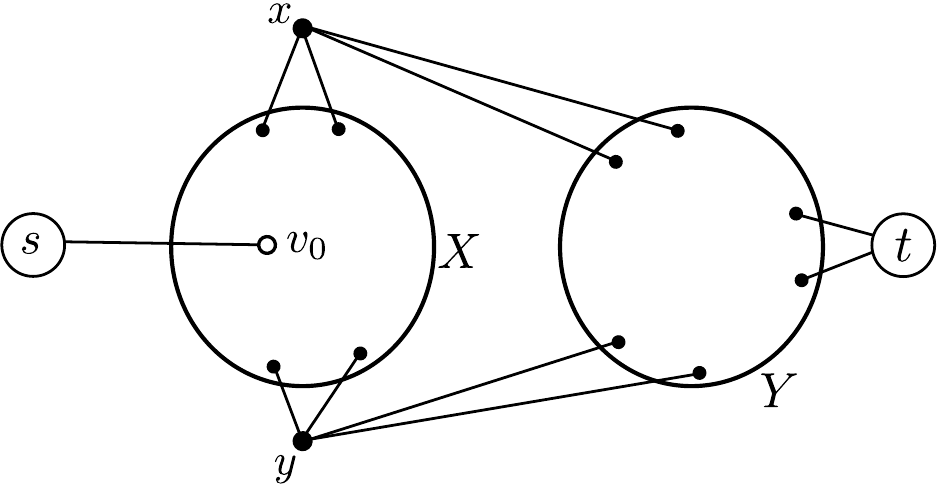}
      \end{center}\vspace{-2mm}
      \caption{Case~2.2.2.2.}
      \label{fig:case_2.2.2.2-1}
    \end{minipage}
  \end{tabular}\vspace{-2mm}
\end{figure}

\item[Case~2.2.2.2 {\rm (Fig.~\ref{fig:case_2.2.2.2-1})}.]
Suppose that $N_\G(Y) = \{x, y, t\}$.
By Claim~\ref{cl:3-cut}-(2), without loss of generality (by the symmetry of $x$ and $y$),
we may assume that $|l(\around{\G}{X} - y; s, x)| \geq 2$.
If $|l(\around{\G}{Y}; x, t)| \geq 2$, then we immediately obtain $|l(\G; s, t)| \geq 3$ by Lemma~\ref{lem:2-2labels}, a contradiction.
Otherwise, by Claim~\ref{cl:3-cut}-(3), we have $Y = \{v\}$ for some $v \in V(\G) \setminus (X \cup \{s, x, y, t\})$,
and $E(\around{\G}{Y})$ consists of a single edge between $v$ and $x$, one between $v$ and $t$, and two parallel edges between $v$ and $y$.
Then, similarly, since $|l(\around{\G}{X}; s, y)| = 1$ must hold (otherwise, Lemma~\ref{lem:2-2labels} implies $|l(\G; s, t)| \geq 3$, a contradiction),
by Claim~\ref{cl:3-cut}-(3), we have $X = \{u\}$ for some $u \in V(\G) \setminus \{s, v, x, y, t\}$ (in particular, $u = v_0$),
and $E(\around{\G}{X})$ consists of a single edge between $u$ and $s$, one between $u$ and $y$, and two parallel edges between $u$ and $x$.

If $V(\G) = \{s, u, v, x, y, t\}$ (see Fig.~\ref{fig:case_2.2.2.2-2}), then both $s$ and $t$ have at least one neighbor other than $u = v_0$ and $v$, respectively, since $\G$ contains no $2$-contractible vertex set (by Claim~\ref{cl:2-contractible}).
In this case, to satisfy $|l(\G; s, t)| = 2$, the triplet $(\G, s, t)$ must be in Case (B) of $\cDab^0$ with $v_1 = u$, $v_2 = y$, $v_3 = x$, and $v_4 = v$ (see Definition~\ref{def:cDab0} and Fig.~\ref{fig:cDab0B}), contradicting $(\G, s, t) \not\in \cDab$.
Otherwise, by Claim~\ref{cl:Y}, there exists a vertex set $Z \subseteq V(\G) \setminus \{s, u, v, x, y, t\} \neq \emptyset$ such that $\G[Z]$ is connected
and $N_\G(Z)$ includes at least one of $\{s, x\}$ and $\{y, t\}$ (see Fig.~\ref{fig:case_2.2.2.2-3}).
When $\{s, x\} \subseteq N_\G(Z)$, we derive $|l(\G; s, t)| \geq 3$, a contradiction, from Lemma~\ref{lem:2-2labels}
by concatenating an $s$--$x$ path in $\G[Z \cup \{s, x\}]$, two $x$--$y$ paths of distinct labels in $\around{\G'}{X} = \around{\G}{X} - s$, and two $y$--$t$ paths of distinct labels in $\around{\G}{Y} - x$.
When $\{y, t\} \subseteq N_\G(Z)$, we also derive $|l(\G; s, t)| \geq 3$, a contradiction, from Lemma~\ref{lem:2-2labels}
by concatenating two $s$--$x$ paths of distinct labels in $\around{\G}{X}$, two $x$--$y$ paths of distinct labels in $\around{\G}{Y} - t$, and a $y$--$t$ path in $\G[Z \cup \{y, t\}]$.

\begin{figure}[htbp]\vspace{2mm}\hspace{-3mm}
  \begin{tabular}{cc}
    \begin{minipage}[b]{0.5\hsize}
      \begin{center}
        \includegraphics[scale=0.8]{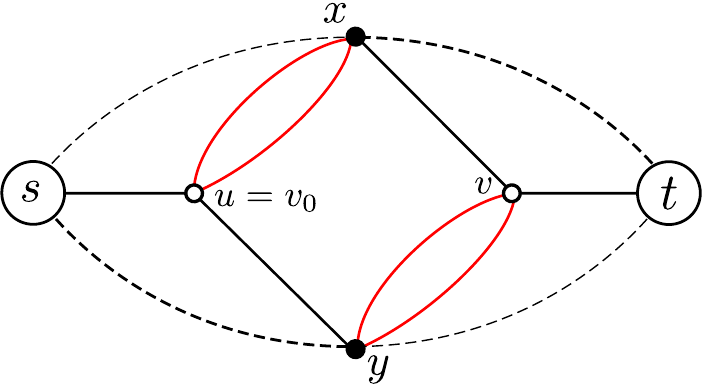}
      \end{center}\vspace{-1mm}
      \caption{Case~2.2.2.2 ($|V(\G)| = 6$).}
      \label{fig:case_2.2.2.2-2}
    \end{minipage}
    \begin{minipage}[b]{0.5\hsize}
      \begin{center}
        \includegraphics[scale=0.8]{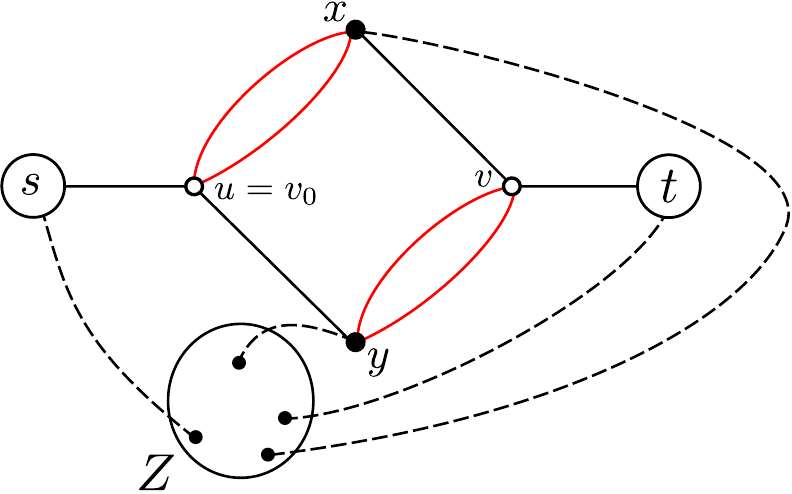}
      \end{center}\vspace{-3mm}
      \caption{Case~2.2.2.2 ($|V(\G)| > 6$).}
      \label{fig:case_2.2.2.2-3}
    \end{minipage}
  \end{tabular}
\end{figure}

\item[Case~2.2.2.3 {\rm (Fig.~\ref{fig:case_2.2.2.3-1})}.]
Suppose that $N_\G(Y) = \{s, x, t\}$, where note again that this case includes the case $N_\G(Y) = \{s, y, t\}$ by the symmetry of $x$ and $y$.
Since $\{s, x\}$ is not a 2-cut in $\G$ (by Claim~\ref{cl:2-contractible}),
there exists a $y$--$t$ path in $\G - (X \cup Y \cup \{s, x\})$.
Hence, if $|l(\around{\G}{Y} - t; s, x)| \geq 2$, then Lemma~\ref{lem:2-2labels} implies $|l(\G; s, t)| \geq 3$, a contradiction,
where recall that $(\around{\G'}{X}, x, y) \in \cDabp^1$.
Otherwise, by Claim~\ref{cl:3-cut}-(2), we have $|l(\around{\G}{Y} - s; x, t)| \geq 2$.
Then, since $|l(\around{\G}{X}; s, x)| = 1$ must hold (otherwise, Lemma~\ref{lem:2-2labels} implies $|l(\G; s, t)| \geq 3$, a contradiction),
by Claim~\ref{cl:3-cut}-(3), we have $X = \{v_0\}$ 
and $E(\around{\G}{X})$ consists of a single edge between $v_0$ and $s$, one between $v_0$ and $x$, and two parallel edges between $v_0$ and $y$ (see Fig.~\ref{fig:case_2.2.2.3-2}).

If there exists an $s$--$y$ path in $\G - (Y \cup \{x, v_0, t\})$,
then we derive $|l(\G; s, t)| \geq 3$, a contradiction, from Lemma~\ref{lem:2-2labels}
by concatenating such an $s$--$y$ path, two $y$--$x$ paths of distinct labels in $\around{\G'}{X} = \around{\G}{X} - s$,
and two $x$--$t$ paths of distinct labels in $\around{\G}{Y} - s$.
Moreover, if there exists a $y$--$x$ path in $\G - (Y \cup \{x, v_0, t\})$,
then we also derive $|l(\G; s, t)| \geq 3$, a contradiction, from Lemma~\ref{lem:2-2labels}
by concatenating two $s$--$y$ paths of distinct labels in $\around{\G}{X} - x$, such a $y$--$x$ path,
and two $x$--$t$ paths of distinct labels in $\around{\G}{Y} - s$.
Hence, we assume that there is no such path in $\G - (Y \cup \{x, v_0, t\})$,
and then there is no vertex set $Z \subseteq V(\G) \setminus (Y \cup \{s, v_0, x, y, t\}) \neq \emptyset$
such that $\G[Z]$ is connected and $y \in N_\G(Z)$ (by Claim~\ref{cl:Y}).
Thus, we have $N_\G(y) \subseteq \{v_0, t\}$,
which contradicts Claim~\ref{cl:2-contractible} (or $(\G, s, t) \in \cD$).
\end{description}

\begin{figure}[tbp]\vspace{0mm}\hspace{-3mm}
  \begin{tabular}{cc}
    \begin{minipage}[b]{0.5\hsize}
      \begin{center}
        \includegraphics[scale=0.7]{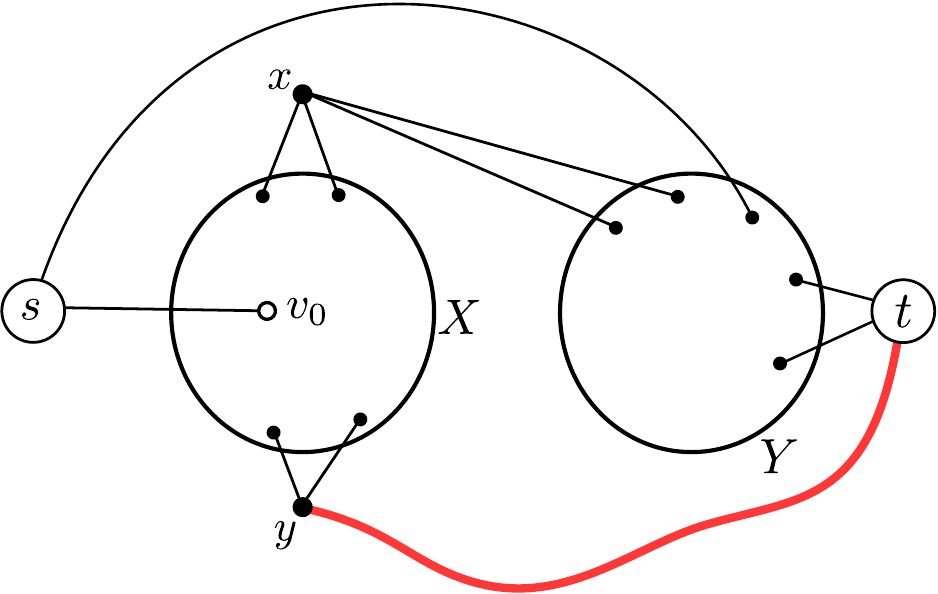}
      \end{center}\vspace{-4mm}
      \caption{Case~2.2.2.3.}
      \label{fig:case_2.2.2.3-1}
    \end{minipage}
    \begin{minipage}[b]{0.5\hsize}
      \begin{center}
        \includegraphics[scale=0.7]{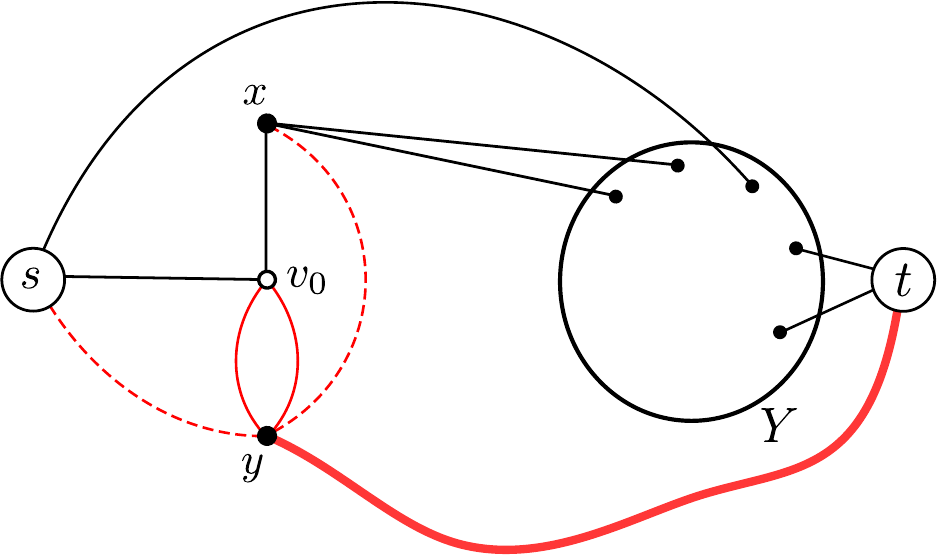}
      \end{center}\vspace{-2mm}
      \caption{Case~2.2.2.3 ($|X| = 1$).}
      \label{fig:case_2.2.2.3-2}
    \end{minipage}
  \end{tabular}\vspace{-2mm}
\end{figure}

\subsection*{Acknowledgments}
We deeply grateful to anonymous reviewers for their careful reading and giving a number of insightful comments.
This work was supported by JSPS KAKENHI Grant Numbers 24106002, 24700004, and 26887014,
by JSPS Grant-in-Aid for JSPS Research Fellow Grant Number 13J02522,
by JST ERATO Grant Number JPMJER1201, and by JST CREST Grant Number JPMJCR14D2.

\end{document}